\documentclass[11pt, reqno]{amsart}
\usepackage[dvipsnames]{xcolor}
\usepackage{amssymb}
\usepackage{verbatim}
\usepackage[vcentermath]{youngtab}
\usepackage{float}
\usepackage{hyperref}
\usepackage{tikz}
\usepackage{tikz-cd}
\usetikzlibrary{calc}
\usetikzlibrary{arrows,decorations.markings}
\usepackage{enumerate}
\usepackage[left=1.1in, right=1.1in, top=1in, bottom=.9in]{geometry}
\usepackage[all,cmtip]{xy}

\usepackage{times}
\usepackage[T1]{fontenc}
\usepackage{mathrsfs}
\usepackage{epsfig}
\usepackage{color}
\usepackage{array} 
\newcolumntype{L}{>{$}l<{$}}
\newcolumntype{R}{>{$}r<{$}}

\usepackage{enumitem}
\usepackage{longtable}
\usepackage{multirow}
\usepackage{makecell}

\newtheorem{thm}{Theorem}[section]
\newtheorem{lemma}[thm]{Lemma}
\newtheorem{conjecture}[thm]{Conjecture}
\newtheorem{cor}[thm]{Corollary}
\newtheorem{prop}[thm]{Proposition}
\newtheorem*{prop*}{Proposition}

\theoremstyle{remark}

\newtheorem*{remark}{Remark}

\theoremstyle{definition}
\newtheorem{defn}[thm]{Definition}
\newtheorem{question}[thm]{Question}
\newtheorem{example}[thm]{Example}
\newtheorem{assumption}[thm]{Assumption}

\def\BQ{\mathbb{Q}}
\def\BP{\mathbb{P}}
\def\BA{\mathbb{A}}
\def\BZ{\mathbb{Z}}
\def\BC{\mathbb{C}}
\def\BQbar{\overline{\mathbb{Q}}}
\def\BN{\mathbb{N}}
\def\BR{\mathbb{R}}
\def\BF{\mathbb{F}}

\def\CR{\mathcal{R}}

\def\CL{\mathcal{L}}
\def\CM{\mathcal{M}}

\def\CO{\mathcal{O}}
\def\CN{\mathcal{N}}
\def\CS{\mathcal{S}}

\def\SP{\mathscr{P}}

\def\Spec{\mathrm{Spec}}

\def\preper{{\mathrm{PrePer}}}

\def\deg{\mathrm{deg}}

\def\Pic{\mathrm{Pic}}
\def\inf{\mathrm{inf}}
\def\exp{\mathrm{exp}}
\def\Re{\mathrm{Re}}

\def\Hom{\mathrm{Hom}}

\def\vol{\mathrm{Vol}}
\def\res{\mathrm{res}}
\def\genus{\mathrm{genus}}
\def\preper{\mathrm{Preper}}
\def\val{\mathrm{val}}
\def\div{\mathrm{div}}
\def\disc{\mathrm{disc}}
\def\rank{\mathrm{rank}}
\def\covol{\mathrm{covol}}
\def\ev{\mathrm{ev}}
\def\det{\mathrm{det}\,}

\def\multi#1#2{\ensuremath{\left(\kern-.3em\left(\genfrac{}{}{0pt}{}{#1}{#2}\right)\kern-.3em\right)}}
\def\into{\hookrightarrow}

\newcommand{\Sym}{\mathrm{Sym}}

\newcommand{\supp}{\mathrm{Supp}}

\numberwithin{equation}{section}

\begin{document}

\title{On the number of quadratic polynomials with a given portrait}

\author{Ho Chung Siu}
\email{soarersiuhc@gmail.com}

\maketitle

\begin{abstract}
Let $F$ be a number field. Given a quadratic polynomial $f_c(z) = z^2 + c \in F[z]$, we can construct a directed graph $\preper(f_c, F)$ (also called a portrait), whose vertices are $F$-rational preperiodic points for $f_c$, with an edge $\alpha \to \beta$ if and only if $f_c(\alpha) = \beta$. Poonen \cite{Poonen} and Faber \cite{Faber} classified the portraits that occur for infinitely many $c$'s. 

Given a portrait $P$, we prove an asymptotic formula for counting the number of $c \in F$'s by height, such that $\preper(f_c, F) \cong P$. We also prove an asymptotic formula for the analogous counting problem, where $\preper(f_c, K) \cong P$ for some quadratic extension $K/F$. These results are conditioned on Morton-Silverman conjecture \cite{MortonSilverman}.
\end{abstract}

\section{Introduction}
Let $F$ be a number field. For $c \in F$, the quadratic polynomial 
$$f_c(z) = z^2 + c$$ 
is an endomorphism of the affine line $\BA^1_F$. Via iteration, we may view it as a dynamical system on $\BA^1(F) = F$. The set of $F$-rational preperiodic points - those with finite forward orbit under $f_c$, will be denoted by $\preper(f_c, F)$. We equip this set with the structure of a directed graph (called the \textbf{preperiodic portrait} or just \textbf{portrait}) by drawing an arrow from $x \to f_c(x)$ for each $x \in \preper(f_c, F)$.

In this setup, a special case of Morton-Silverman conjecture \cite{MortonSilverman} is as follows.
\begin{conjecture}
\label{Conjecture_Morton_Silverman}
Fix a number field $F$. For a given integer $d \ge 1$, there is a uniform bound $N = N(F, d)$ on the size of preperiodic points 
$$\# \,\preper(f_c, K) \leq N$$
across all number fields $K/F$ of degree $d$, and polynomials $f_c(z) = z^2 + c \in K[z]$.
\end{conjecture}
A consequence of Morton-Silverman conjecture is that there are only finitely many possible portraits $\preper(f_c, K)$ across all number fields $K/F$ of degree $d$.

Although the conjecture asserts uniformity over all number fields $K/F$ of the same degree $d$, Conjecture \ref{Conjecture_Morton_Silverman}
 is not known even for the simplest case of $d = 1$ and $F = \BQ$. However, we do know what portraits $\preper(f_c, \BQ)$ can show up for infinitely many $c \in \BQ$'s, as recorded in the following theorem.

\begin{figure}[h!]
\label{Picture_Preperiodic_Portrait_Q_infinite}
\includegraphics[scale=0.7]{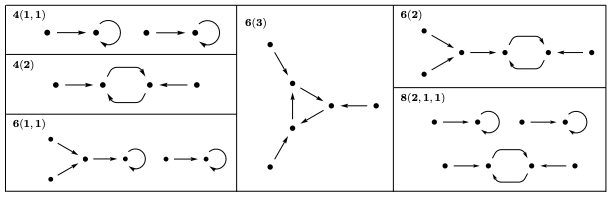}
\caption{Portraits of $\preper(f_c, \BQ)$ occurring infinitely often, excluding the empty portrait (Picture from \cite{Faber}).}
\end{figure}

\begin{thm}[{Poonen \cite{Poonen}, Faber \cite{Faber}, Doyle \cite[Proposition 4.2]{DoyleCyclotomic}}]
\label{Theorem_Base_Degree1}
The portraits that occur as $\preper(f_c, \BQ)$ for infinitely many $c \in \BQ$'s are exactly
$$\emptyset, 4(1,1), 4(2), 6(1,1), 6(2), 6(3), 8(2,1,1).$$
See Figure \ref{Picture_Preperiodic_Portrait_Q_infinite} for pictures of these graphs.

More generally, fix a number field $F$. If there are infinitely many $c \in F$'s such that $\preper(f_c, F) \cong P$, then $P$ is isomorphic to one of the following graphs in $\SP_1 := \Gamma_0 \cup \Gamma_1$, where
$$\Gamma_0 := \{\emptyset, 4(1, 1), 4(2), 6(1, 1), 6(2), 6(3), 8(2, 1, 1)\},$$
$$\Gamma_1 := \{8(1, 1)a, 8(1, 1)b, 8(2)a, 8(2)b, 10(2, 1, 1)a, 10(2, 1, 1)b\}.$$
See Appendix \ref{Appendix_Explicit_Presentation_piG} for pictures of these graphs.
\end{thm}

The purpose of this paper is to understand how often each portrait occurs. To make things precise, we will use the Weil height: if $c = \frac{a}{b} \in \BQ$ where $a, b$ are coprime integers, we define the (absolute, multiplicative) Weil height $H(c)$ to be
$$H(c) = \max\{|a|, |b|\}$$
This can be extended to $H: \BQbar \to [1, \infty)$.

\begin{thm} \label{Theorem_Main_Degree1}
Fix a number field $F$ and let $H: \BQbar \to [1, \infty)$ be the absolute multiplicative Weil height. Consider the portraits $\SP_1$ defined in Theorem \ref{Theorem_Base_Degree1}.

For each $P \in \SP_1$, define 
$$S_{F,1}(P, B) := \#\{c \in F: \preper(f_c, F) \cong P, H(c) \leq B\}$$   
Assume the Morton-Silverman conjecture for the number field $F$. Then for each $P \in \SP_1$, we have an asymptotic formula of the form
$$S_{F,1}(P, B) \sim c_{F, 1}(P) B^{a_{F, 1}(P)} (\log B)^{b_{F, 1}(P)}$$
as $B \to \infty$. 

(The constants $a_{F, 1},b_{F, 1},c_{F, 1}$ are defined in Definition \ref{Definition_Degree1_abc_constants}. The theorem with an explicit error term is restated and proved at Theorem \ref{Theorem_Main_Degree1_Details}.)
\end{thm}
An example of this theorem for the portrait $8(2, 1, 1)$ can be found at Section \ref{SectionProofMainThmDegree1Example}.

\begin{remark}
The shape of the main term and the error term in the asymptotic formula would depend on the geometry (e.g. genus) of the dynamical modular curve $X_1(P)$. Hence we defer the precise result to Definition \ref{Definition_Degree1_abc_constants} and Theorem \ref{Theorem_Main_Degree1_Details}.

In particular, we can give an interpretation of $a_{F, 1}, b_{F, 1}, c_{F, 1}$ in terms of the geometry/number theory of $X_1(P)$ and the natural forgetful map $\pi_P: X_1(P) \to \BP^1$.
\end{remark}

By the definition of $a_{F, 1}$, it is immediate that $a_{F,1}(\emptyset) > a_{F, 1}(P)$ for any $P \neq \emptyset$. We hence obtain the following corollary.
\begin{cor}
Fix a number field $F/\BQ$, and assume the Morton-Silverman conjecture for the number field $F$. 

When ordering $c \in F$ by height, 100\% of quadratic polynomials $f_c(z) = z^2 + c$ has no $F$-rational preperiodic point.
\end{cor}
This corollary is already known unconditionally by Sadek \cite[Theorem 4.8]{Sadek} and Le Boudec-Mavraki \cite[Theorem 1.1]{LeBoudecMavraki}. Our results here can be seen as a refinement of their results, since we can compare the frequency of occurrence for portraits of density 0 (although we need to assume Morton-Silverman conjecture). We also remark that Olechnowicz computed $c_{\BQ, 1}(P)$ for $P = 4(1, 1), 4(2)$ in a related line of investigation \cite[Theorem 8.3]{Olechnowicz}, using methods similar to those in Section 4.

\begin{remark}
The analogous problem for elliptic curves: counting by height the number of elliptic curves with a prescribed rational torsion group, was studied extensively in recent years. We refer readers to \cite{HarronSnowden}, \cite[Corollary 1.3.6]{CullinanKenneyVoight} for results over $\BQ$, and \cite{BruinNajman} \cite{ImKim} \cite{Philips} for results over general number fields. See also \cite{EllenbergSatrianoZureickBrown} for the general setup and conjectures when counting points on stacks by height.

Our situation is nicer: there is an analogue of modular curves in our situation (dynamical modular curves), which are smooth projective curves instead of stacks. Existing results in counting rational points on smooth projective curves by height would apply. As a result, we can give a geometric interpretation of the leading coefficient $c_{F, 1}(P)$ in the asymptotics formula.
\end{remark}

For a fixed degree $d \ge 1$, Morton-Silverman's conjecture is uniform across all number fields $K/F$ of degree $d$. Hence when counting points by height, it is natural to consider all number fields $K/F$ of degree $d$ altogether. We ask
\begin{question}
Fix a number field $F$, an integer $d \ge 1$, and a portrait $P$. When is 
$$\{c \in \overline{F}: \preper(f_c, K) \cong P \text{ for some $K/F$ of degree $d$ containing $c$}\}$$
an infinite set? When this set is infinite, how quickly does
$$\{c \in \overline{F}: \preper(f_c, K) \cong P \text{ for some $K/F$ of degree $d$ containing $c$}, H(c) \leq B\}$$
grow, as $B \to \infty$?
\end{question}

The case of $d = 1$ was answered by Theorem \ref{Theorem_Base_Degree1} and Theorem \ref{Theorem_Main_Degree1}. 

For $d = 2$, the analogue of Theorem \ref{Theorem_Base_Degree1} was recently obtained by Doyle-Krumm \cite{DoyleKrumm}. 
\begin{thm}[{Doyle-Krumm, \cite[Theorem 1.5]{DoyleKrumm}}]
\label{Theorem_Base_Degree2}
Fix a number field $F$. If there are infinitely many $c \in \overline{F}$'s such that
$$\preper(f_c, K) \cong P$$
for some $K/F$ of degree 2 containing $c$, then $P$ is isomorphic to one of the following graphs in $\SP_2 := \Gamma_0 \cup \Gamma_1 \cup \Gamma_2$.
$$\Gamma_0 := \{\emptyset, 4(1, 1), 4(2), 6(1, 1), 6(2), 6(3), 8(2, 1, 1)\}$$
$$\Gamma_1 := \{8(1, 1)a, 8(1, 1)b, 8(2)a, 8(2)b, 10(2, 1, 1)a, 10(2, 1, 1)b\}$$
$$\Gamma_2 := \{8(3), 8(4), 10(3, 1, 1), 10(3, 2)\}$$
See Appendix \ref{Appendix_Explicit_Presentation_piG} for pictures of these graphs.
\end{thm}
\begin{remark}
Although the theorem was only stated for $\BQ$ in \cite[Theorem 1.5]{DoyleKrumm}, the proof there works verbatim for any number field $F$.
\end{remark}
\begin{remark}
$\Gamma_g$ are exactly the portraits $P$ whose dynamical modular curve $X_1(P)$ has genus $g$. In particular, a consequence of the theorem is that if $X_1(P)$ has genus $\ge 3$, then $X_1(P)$ has finitely many quadratic points over $F$.
\end{remark}

Conditional on Morton-Silverman conjecture, our second theorem counts these portraits for $d = 2$.

\begin{thm}\label{Theorem_Main_Degree2}
Fix a number field $F/\BQ$ and let $H: \BQbar \to [1, \infty)$ be the absolute multiplicative Weil height. Consider the portraits $\SP_2$ defined in Theorem \ref{Theorem_Base_Degree2}.

For each $P \in \SP_2$, define 
$$S_{F,2}(P, B) := \#\{c \in \overline{F}: \preper(f_c, K) \cong P \text{ for some $K/F$ of degree $2$ containing $c$}, H(c) \leq B\}.$$   
Assume the Morton-Silverman conjecture for ground field $F$ and $d = 2$. Then for each $P \in \SP_2$, we have an asymptotic formula of the form
$$S_{F,2}(P, B) \sim c_{F, 2}(P) B^{a_{F, 2}(P)} (\log B)^{b_{F, 2}(P)}$$
as $B \to \infty$. 

(The constants $a_{F, 2},b_{F, 2},c_{F, 2}$ are defined in Definition \ref{Definition_Degree2_abc_constants}. The theorem with an explicit error term is restated and proved at Theorem \ref{Theorem_Main_Degree2_Details}.)
\end{thm}
An example of this theorem for the portrait $8(2, 1, 1)$ can be found at Section \ref{SectionProofMainThmDegree2Example}.

\begin{remark}
The shape of the main term and the error term in the asymptotic formula would depend on the geometry (e.g. genus) of the dynamical modular curve $X_1(P)$. Hence we defer the precise result to Definition \ref{Definition_Degree2_abc_constants} and Theorem \ref{Theorem_Main_Degree2_Details}.

In particular, we can give an interpretation of $a_{F, 2}, b_{F, 2}, c_{F, 2}$ in terms of the geometry/number theory of $X_1(P)$ and the natural forgetful map $\pi_P: X_1(P) \to \BP^1$.
\end{remark}

\subsection{Overview of the proof}
The proofs of the main theorems rely on the notion of dynamical modular curve. For a portrait $P$, we can define a smooth projective curve $X_1(P)$ that parametrizes the pair $(c, (z_1, \cdots, z_n))$, where $c \in \BA^1$, and $(z_1, \cdots, z_n)$ transforms under $f_c$ as prescribed by $P$, without degeneracy. We call $X_1(P)$ a dynamical modular curve. It comes with a natural forgetful map 
$$\pi_P: X_1(P) \to \BP^1$$ 
that only "remembers" $c$ and forgets the prescribed points.
\begin{remark}
The situation is entirely analogous to classical modular curves $X_1^{ell}(N)$ that parametrizes elliptic curves with a $N$-torsion point. It comes with a forgetful map $X_1^{ell}(N) \to X_1^{ell}(1)$, which remembers the elliptic curve but forgets the torsion point.
\end{remark}

Rational points on $X_1(P)$ encodes $f_c$'s whose preperiodic portrait \textbf{contains} $P$. Assuming Morton-Silverman conjecture, there are only finitely many portraits $P'$ containing $P$ where $X_1(P')$ has a rational point. By the principle of inclusion-exclusion, our theorems amount to understanding:
\begin{itemize}
    \item (Case of $d = 1$) For a fixed number field $F$, the growth rate of 
    $$N_{F, 1}(P, B) := \#\{x \in X_1(P)(F): H(\pi_P(x)) \leq B\}$$
    as $B \to \infty$.
    \item (Case of $d = 2$) For a fixed number field $F$, the growth rate of
    $$N_{F, 2}(P, B) := \#\{x \in X_1(P)(\overline{F}): [F(x): F] \leq 2, H(\pi_P(x)) \leq B\}$$
    as $B \to \infty$.
\end{itemize}

In the case of $d = 1$, Faltings' theorem tells us that if genus of $X_1(P)$ is at least 2, the number of $F$-rational points is finite, hence bounded independent of $B$. We thus focus our attention to the cases where $X_1(P)$ has genus 0 or 1. In both cases $N_{F, 1}(P, B)$ can be computed: for genus 0, this is a known case of Batyrev-Manin-Peyre Conjectures for $\BP^1$, proved by Franke-Manin-Tschinkel \cite{FrankeManinTschinkel}; for genus 1, this is due to N\'{e}ron \cite{Neron}. 

In the case of $d = 2$, we use a similar strategy on the symmetric product $\Sym^2 X_1(P)$. Faltings' theorem has an analogue in this situation \cite{FaltingsBig}, and one can show that if genus of $X_1(P)$ is at least 3, the number of quadratic points is finite, hence bounded independent of $B$ \cite[Theorem 1.5]{DoyleKrumm}. We thus focus on the portraits $P$ where $X_1(P)$ has genus $\leq 2$, and we are tasked to count rational points on $\Sym^2 X_1(P)$ by height. Again we can compute $N_{F, 2}(P, B)$: our tools are Franke-Manin-Tschinkel's theorem \cite{FrankeManinTschinkel} applied to $\BP^1$ and $\BP^2$, some calculus on arithmetic surfaces, and N\'{e}ron's theorem \cite{Neron} applied to abelian surfaces.

\subsection{Outline of the paper}
In Section 2, we recall basic facts of dynamical modular curve and set up our notations.

In Section 3, we recall the basic theory of heights via metrized line bundles, and set up our notations/conventions around absolute values and heights. We then discuss the Batyrev-Manin-Peyre conjectures \cite{BatyrevManin}\cite{Peyre}, which will allow us to interpret the leading constants $c_{F, i}$ in our asymptotic formulas, in terms of the degree of the forgetful map $\pi_P$ and a Tamagawa measure on $X_1(P)$. Finally, we document two known results of counting rational points by height: Franke-Manin-Tschinkel \cite{FrankeManinTschinkel} and N\'{e}ron \cite{Neron}.

In Section 4, we illustrate Theorem \ref{Theorem_Main_Degree1} and Theorem \ref{Theorem_Main_Degree2} for the portrait $8(2, 1, 1)$ over $\BQ$. We will also prove Theorem \ref{Theorem_Main_Degree1} directly for this portrait. This makes it clear how general point counting result comes in, as well as where we would need to assume Morton-Silverman conjecture. Readers are suggested to read Section 4 first before reading the rest of the paper.

In Section 5, we prove Theorem \ref{Theorem_Main_Degree1} (see Theorem \ref{Theorem_Main_Degree1_Details}). The main input is an asymptotic formula for $N_{F, 1}(P, B)$ as $B \to \infty$ (Theorem \ref{Theorem_Main_Degree1_CountingResult_Summary}), using Franke-Manin-Tschinkel \cite{FrankeManinTschinkel} and N\'{e}ron \cite{Neron}.

In Section 6, we prove Theorem \ref{Theorem_Main_Degree2} (see Theorem \ref{Theorem_Main_Degree2_Details}). The main input is an asymptotic formula for $N_{F, 2}(P, B)$ as $B \to \infty$ (Theorem \ref{Theorem_Main_Degree2_CountingResult_Summary}). Specifically for the case of symmetric square of genus 1 curves, we will need some calculus on arithmetic surfaces; the full details in that case would be in Appendix \ref{Appendix_Proof_Sym2_EC}.

\subsection{Notations}
We will use $f = O_{X, Y, Z}(g)$ or $f \ll_{X, Y, Z} g$ to mean that
$$|f| \leq C_{X, Y, Z} |g|$$
for some constant $C_{X, Y, Z} > 0$ depending on $X, Y, Z$.

We will use $f \asymp_{X, Y, Z} g$ to mean $f \ll_{X, Y, Z} g$ and $g \ll_{X, Y, Z} f$.

\subsection{Acknowledgements}
We thank Trevor Hyde and Simon Rubinstein-Salzedo for feedback on the draft of this paper. Special thanks to Trevor Hyde for introducing the author to arithmetic dynamics and various discussions.

\section{Preliminaries on dynamical modular curves}
We recollect the definition of dynamical modular curves and their basic properties. This sets up our notations for dynamical modular curves and their forgetful maps. Our main reference is \cite{Doyle}.

\subsection{Dynatomic modular curves}
Let $F$ be a number field, and $N$ be a positive integer. Let $f_c(z) \in F[z]$ be defined by
$$f_c(z) = z^2 + c.$$
Suppose $c, x \in F$ are such that $x$ has period $N$ for $f_c(z)$, i.e. 
$$f_c^N(x) = x$$
and for all $k < N$, $f_c^k(x) \neq x$. To pick out the $x$'s with exactly period $N$, we define the \textbf{$N$-th dynatomic polynomial} to be
$$\Phi_N(c, z) := \prod_{n | N} \left(f_c^n(z) - z\right)^{\mu(N/n)}.$$
That $\Phi_N$ is actually a polynomial is shown in \cite[Theorem 4.5]{SilvermanDynamicalSystem}. There is a natural factorization,
$$f_c^N(z) - z = \prod_{n | N} \Phi_N(c, z).$$

More generally, let $M, N \ge 1$ be integers, and suppose $c, x \in F$ are such that $x$ has preperiod $M$ and eventual period $N$ for $f_c(z)$, i.e.
$$f_c^{M + N}(x) = f_c^M(x)$$
and for $1 \leq k \leq M, 1 \leq l \leq N$, 
$$f_c^{k + l}(x) \neq f_c^{k}(x),$$ 
unless $k = M$ and $l = N$. Note that $(c, x)$ satisfies
$$\Phi_N(c, f_c^M(x)) = 0,$$
but points with preperiod $\leq M$ and eventual period $N$ also satisfy the same equation. To pick out the $z$'s with preperiod exactly $M$ and eventual period exactly $N$, we define the \textbf{generalized dynatomic polynomial} to be
$$\Phi_{M,N}(c, z) := \frac{\Phi_{N} (c, f_c^M(z))}{\Phi_{N} (c, f_c^{M-1}(z))}.$$
That $\Phi_{M, N}$ is actually a polynomial is shown in \cite{Hutz}. 

The vanishing locus of $\Phi_N$ (resp. $\Phi_{M, N}$) defines an affine curve $Y_1(N)$ (resp. $Y_1(M, N)$), which we refer to as a dynatomic modular curve. We denote $X_1(\cdot)$ to be the smooth projective model of the curve $Y_1(\cdot)$, and we also refer to $X_1(\cdot)$ as a dynatomic modular curve. These modular curves are the key examples of dynamical modular curves, that we will introduce in the next section.

\subsection{Dynamical modular curve associated to directed graphs}
Let $G$ be a finite directed graph. Fix an ordering of vertices $V(G)$, denoted as $v_1, \cdots, v_n$. We can then define \textbf{the dynamical modular curve associated to $G$}, denoted $Y_1(G)$, as follows. This is the affine curve that parametrizes the tuple $(c, z_1, \cdots, z_n)$, where $c \in \BA^1$, $(z_1, \cdots, z_n) \in \BA^n$, and that $f_c(z_i) = z_j$ if and only if $v_i$ has an edge towards to $v_j$, and $z_i$'s do not "degenerate"; see \cite[Section 2.4]{Doyle} for a precise definition of $Y_1(G)$. We denote $X_1(G)$ to be the smooth projective model of the curve $Y_1(G)$.

For $Y_1(G)$ to be non-empty, one needs to impose restrictions on $G$.
\begin{enumerate}
    \item Each vertex has at most one outward edge, since $f_c(x)$ is well-defined for any given $x$; with our focus on preperiodic points, it does not hurt to assume each vertex has exactly one outward edge. 
    \item Each vertex has at most two inward edges: $f_c(z)$ is quadratic in $z$, so for any given $c, a$, $f_c(z) = a$ has at most two solutions. 
    
    Generically, we should expect each vertex to have either no inward edge or exactly two inward edges, as long as no double roots occur.
    \item For any $N \ge 1$, the number of cycles of length $N$ is bounded. Indeed, any period $N$ point for $f_c$ is a root of $\Phi_N(c, z) = 0$; moreover, if $z$ is a period $N$ point, so is $f_c^k(z)$ for $0 < k < N$. $N$-cycles either coincide or do not overlap at all. Hence if we set
    $$D(N) := \deg_z \Phi_N(c, z) = \sum_{n|N} \mu(N/n) 2^n,$$
    and
    $$R(N) := \frac{D(N)}{N},$$
    then there are $\leq R(N)$ cycles of length $N$.
    \item There are at most 2 fixed points (vertex with an edge to itself). This is because $f_c(z)$ is quadratic in $z$, so for any given $c$, $f_c(z) = z$ has at most two solutions. 
    
    Generically, we should expect either no or exactly two fixed points, as long as no double roots occur.
\end{enumerate}

With these restrictions in mind, we define
\begin{defn}
A finite directed graph is called \textbf{strongly admissible} if it satisfies the following properties:
\begin{enumerate}[label=(\alph*)]
    \item Each vertex of $G$ has out-degree 1 and in-degree either 0 or 2.
    \item For each $N \ge 2$, $G$ contains at most $R(N)$ $N$-cycles.
    \item $G$ contains either 0 or 2 fixed points.
\end{enumerate}
\end{defn}

\begin{remark}
The definition of strong admissibility is taken from \cite[Definition 2.3]{Doyle}. We sometimes quote results from \cite{DoyleKrumm}, where "generic quadratic" is used synonymously as "strongly admissible".
\end{remark}
\begin{remark}
Given any directed graph $G_0$ where each vertex has out-degree 1, there is a unique, minimal, strongly admissible graph $G$ that contains $G_0$, with a canonical isomorphism $X_1(G) \cong X_1(G_0)$. Hence it does not hurt to only focus on strongly admissible graphs; see \cite{Doyle} for more details. 

Even if a directed graph $G_0$ is not strongly admissible, we may abuse notation and refer to dynamical modular curve $X_1(G_0)$, to really mean $X_1(G)$.
\end{remark}

\begin{remark}
We only need dynamical modular curve associated to quadratic polynomials in this paper, whose theory was developed in \cite{Doyle}. More generally one can define dynamical moduli space associated to morphisms on $\BP^n$; see \cite{DoyleSilverman}.
\end{remark}

\begin{example}
The cyclic graph $C_4$ of four vertices is not strongly admissible, because the in-degree of each vertex is 1. The minimal, strongly admissible graph $C_4'$ that contains $C_4$, can be obtained by adding vertices and edges, so that the original four vertices of $C_4$ have in-degree 2. These graphs are depicted in Figure \ref{FigureExampleCyclicGraph}.

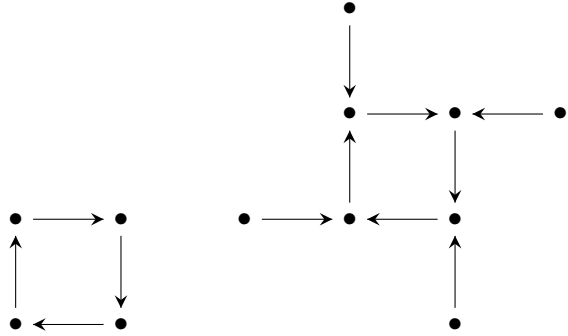
\begin{figure}[h!]
    \centering
    \begin{tikzpicture}[scale=1.4]
    \tikzset{vertex/.style = {}}
    \tikzset{every loop/.style={min distance=10mm,in=45,out=-45,->}}
    \tikzset{edge/.style={decoration={markings,mark=at position 1 with %
        {\arrow[scale=1.5,>=stealth]{>}}},postaction={decorate}}}
    \node[vertex] (a) at (0, 0) {$\bullet$};
    \node[vertex] (b) at (0, 1) {$\bullet$};
    \node[vertex] (c) at (1, 1) {$\bullet$};
    \node[vertex] (d) at (1, 0) {$\bullet$};
    
    \draw[edge] (a) to (b);
    \draw[edge] (b) to (c);
    \draw[edge] (c) to (d);
    \draw[edge] (d) to (a);
    
    \end{tikzpicture}
    \hspace{1cm}
    \begin{tikzpicture}[scale=1.4]
    \tikzset{vertex/.style = {}}
    \tikzset{every loop/.style={min distance=10mm,in=45,out=-45,->}}
    \tikzset{edge/.style={decoration={markings,mark=at position 1 with %
        {\arrow[scale=1.5,>=stealth]{>}}},postaction={decorate}}}
    \node[vertex] (a) at (0, 0) {$\bullet$};
    \node[vertex] (b) at (0, 1) {$\bullet$};
    \node[vertex] (c) at (1, 1) {$\bullet$};
    \node[vertex] (d) at (1, 0) {$\bullet$};

    \node[vertex] (a2) at  (-1,0) {$\bullet$};
    \node[vertex] (b2) at  (0,2) {$\bullet$};
    \node[vertex] (c2) at  (2,1) {$\bullet$};
    \node[vertex] (d2) at  (1,-1) {$\bullet$};
    
    \draw[edge] (a) to (b);
    \draw[edge] (b) to (c);
    \draw[edge] (c) to (d);
    \draw[edge] (d) to (a);

    \draw[edge] (a2) to (a);
    \draw[edge] (b2) to (b);
    \draw[edge] (c2) to (c);
    \draw[edge] (d2) to (d);
    \end{tikzpicture}

    \caption{Left: $C_4$, the cyclic graph of 4 vertices. Right: $C_4$', the minimal strongly admissible graph containing $C_4$.}
    \label{FigureExampleCyclicGraph}
\end{figure}
As in previous remark, $X_1(C_4') \cong X_1(C_4) \cong X_1(4)$, which recovers the dynatomic modular curve defined by $\Phi_4(c, z)$. 

More generally, 
\begin{itemize}
    \item $X_1(N)$ is the dynamical modular curve associated to the directed cyclic graph $C_N$ of $N$ vertices.
    \item $X_1(M, N)$ is the dynamical modular curve associated to the directed graph $C_{M, N}$ with an $N$-cycle and a tail of length $M$. We show an example of $C_{2, 4}$ in Figure \ref{FigureExampleCyclicGraphWithTail}.
\end{itemize}
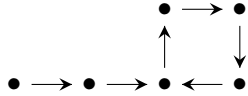
\begin{figure}[h!]
    \centering
    \begin{tikzpicture}
    \tikzset{vertex/.style = {}}
    \tikzset{every loop/.style={min distance=10mm,in=45,out=-45,->}}
    \tikzset{edge/.style={decoration={markings,mark=at position 1 with %
        {\arrow[scale=1.5,>=stealth]{>}}},postaction={decorate}}}
    \node[vertex] (a) at (0, 0) {$\bullet$};
    \node[vertex] (b) at (0, 1) {$\bullet$};
    \node[vertex] (c) at (1, 1) {$\bullet$};
    \node[vertex] (d) at (1, 0) {$\bullet$};

    \node[vertex] (a-1) at  (-1,0) {$\bullet$};
    \node[vertex] (a-2) at  (-2,0) {$\bullet$};
    \draw[edge] (a) to (b);
    \draw[edge] (b) to (c);
    \draw[edge] (c) to (d);
    \draw[edge] (d) to (a);

    \draw[edge] (a-1) to (a);
    \draw[edge] (a-2) to (a-1);
    
    \end{tikzpicture}
    \caption{The graph $C_{2, 4}$, with a 4-cycle and a tail of length 2.}
    \label{FigureExampleCyclicGraphWithTail}
\end{figure}
\end{example}

We summarize properties of dynamical modular curves in the following proposition.
\begin{prop}\label{Proposition_Properties_Of_X1G}
Let $G$ be a strongly admissible directed graph.
\begin{enumerate}[label=(\alph*)]
    \item \cite[Theorem 1.7]{Doyle} $X_1(G)$ is geometrically irreducible over $\BQ$.
    \item \cite[Proposition 3.3]{Doyle} If $H$ is another strongly admissible directed graph contained in $G$, then there is a canonical finite map $\pi_{G, H}: X_1(G) \to X_1(H)$ of degree $\ge 2$.

    In particular, if $H$ is the empty graph, $X_1(\emptyset) \cong \BP^1$, and we have the canonical forgetful map
    $$\pi_G: X_1(G) \to \BP^1$$
    that only "remembers" $c$ and forgets the prescribed points.
\end{enumerate}
\end{prop}
We will also encounter the automorphisms of $\pi_G$ in our theorems.
\begin{defn}
Let $G$ be a strongly admissible directed graph, and $\pi_G: X_1(G) \to \BP^1$ be the forgetful map. We define the automorphism group of $\pi_G$ to be
$$Aut(\pi_G) := \{\sigma: X_1(G) \to X_1(G): \text{$\sigma$ invertible morphism of curves / $\BQ$, and } \pi_G \circ \sigma = \sigma\}.$$ 
\end{defn}

\subsection{When $X_1(G)$ has small genus}
\begin{prop} \cite[Proposition 4.2]{DoyleCyclotomic}
\label{ClassificationForGenus01:AdmissibleGraphs}
Let $G$ be a strongly admissible directed graph. Then $X_1(G)$ has genus $\leq 2$ if and only if $G$ is isomorphic to one of the following graphs in $\Gamma_0 \cup \Gamma_1 \cup \Gamma_2$:
$$\Gamma_0 := \{\emptyset, 4(1, 1), 4(2), 6(1, 1), 6(2), 6(3), 8(2, 1, 1)\}$$
$$\Gamma_1 := \{8(1, 1)a, 8(1, 1)b, 8(2)a, 8(2)b, 10(2, 1, 1)a, 10(2, 1, 1)b\}$$
$$\Gamma_2 := \{8(3), 8(4), 10(3, 1, 1), 10(3, 2)\}$$
Here $\Gamma_g$ is the set of strongly admissible graphs $G$, where $X_1(G)$ has genus $g$.

Furthermore,
\begin{enumerate}[label=(\alph*)]
    \item If $G \in \Gamma_0$, then $X_1(G)$ is rational over $\BQ$.
    \item If $G \in \Gamma_1$, then $X_1(G)$ is an elliptic curve over $\BQ$.
    \item If $G \in \Gamma_2$, then $X_1(G)$ is a smooth projective genus 2 curve over $\BQ$.
\end{enumerate}
For each of the graphs above, see Appendix \ref{Appendix_Explicit_Presentation_piG} for a model of $X_1(G)$, their Cremona label/LMFDB label, a presentation of the forgetful map $\pi_G: X_1(G) \to \BP^1$, and the automorphism group $Aut(\pi_G)$.
\end{prop}

\section{Preliminaries on counting rational points by height}
\subsection{Weil heights}
We recollect some basic facts about Weil heights, and record the convention around absolute values and heights in this paper. More details can be found in \cite[Chapter 1]{BombieriGruber}.

By Ostrowski's theorem, the only non-trivial absolute values (up to equivalence) on $\BQ$ are:
\begin{itemize}
    \item $|\cdot|_p$, where $p$ is a prime. This is defined as: if $x = p^k \frac{a}{b} \in \BQ$ for coprime integers $a, b$ satisfying $gcd(p, ab) = 1$, we define
    $$|x|_p = p^{-k}.$$
    \item $|\cdot|_{\infty}$, the usual absolute value on $\BR$ restricted to $\BQ$.
\end{itemize}
We denote the set of equivalence classes of non-trivial absolute values on $\BQ$, also called places of $\BQ$, as $M_{\BQ}$.

Now let $F$ be a number field over $\BQ$ of degree $d = [F: \BQ]$. A place $v$ on $F$ is an absolute value $F \to [0, \infty)$ whose restriction to $\BQ$ is a power of either $|\cdot|_p$ or $|\cdot|_{\infty}$. Denote $F_v$ as the completion of $F$ with respect to $v$. An absolute value is normalized if
$$|x|_v = |N_{F_v/\BQ_p}(x)|_p,$$
if $v$ lies above $|\cdot|_p$. These normalized absolute values satisfy the product formula: for $x \in F^{\times}$,
$$\prod_{v \in M_F} |x|_v = 1,$$
where we denote $M_F$ as the set of places on $F$. In this paper, we may use "place" and "normalized absolute value" synonymously. We use $M_{F, finite}$ to denote the non-archimedean places lying over $|\cdot|_p$, and use $M_{F, \infty}$ to denote the archimedean places lying over $|\cdot|_{\infty}$.

Let $v$ be a place, and $F_v$ be the completion of $F$ with respect to $v$. Sometimes we need to extend $|\cdot|_v$ to $\overline{F_v}$, defined as follows. For $y \in \overline{F_v}$, define
$$|y|_v := |N_{F_v(y)/F_v}(y)|_v^{1/[F_v(y):F_v]}.$$
When restricted to a finite extension of $F_v$, this gives an unnormalized extension of the absolute value $|\cdot|_v$.

A height on $\BP^n(F)$ can be defined by
$$H_{F/\BQ}([x_0, \cdots, x_n]) = \prod_{v \in M_F} \max\{|x_0|_v, \cdots, |x_n|_v\}^{1/[F:\BQ]}.$$
One can check that if $F'/F$ is a finite extension of number fields, and $x \in \BP^n(F) \subset \BP^n(F')$, then
$$H_{F/\BQ}(x) = H_{F'/\BQ}(x).$$
Hence $H_{\cdot/\BQ}$ glues to the absolute multiplicative Weil height $H: \BP^n(\BQbar) \to [1, \infty)$.

\subsection{Haar measures}
Our choice of Haar measure is consistent with that in Peyre \cite{Peyre}.

Let $F$ be a number field, and $v \in M_F$ be a place. Let $F_v$ be the completion of $F$ at $v$, $O_{F_v}$ be the ring of integers of $F_v$, and $O_{F_v}^{\times}$ be the units of $O_{F_v}$.
\begin{itemize}
    \item The additive Haar measure $dx_v$ on $F_v$ satisfies
    $$dx_v(aU) = |a|_v dx_v(U),$$
    for a measurable set $U$ and $a \in F_v$. The measure is normalized such that
    \begin{itemize}
        \item If $F_v \cong \BR$, $dx_v$ is the usual Lebesgue measure.
        \item If $F_v \cong \BC$, $dx_v$ is twice the Lebesgue measure.
        \item If $v$ is non-archimedean,
        $$dx_v(O_{F_v}) = 1.$$
    \end{itemize}
    \item The multiplicative Haar measure $d^{\times} x_v$ on $F_v^{\times}$ is a multiple of $\dfrac{dx_v}{|x_v|_v}$. We normalize it such that
    \begin{itemize}
        \item If $v$ is archimedean, 
        $$d^{\times}x_v = \frac{dx_v}{|x_v|_v}.$$
        \item If $v$ is non-archimedean, 
        $$d^{\times}x_v (O_{F_v}^{\times}) = 1.$$
    \end{itemize}
\end{itemize}
\subsection{Heights and metrized line bundles}
We recall the definition of metrized line bundles for a projective variety over a number field $F$, and recall how heights can be defined via metrized line bundles. We generally follow the exposition of \cite{ChambertLoir}; more details can be found in \cite[Chapter 2.7]{BombieriGruber}.

Let $F$ be a number field. Let $V$ be a projective variety over $F$, and let $L$ be a line bundle on $V$.

\begin{defn}[$v$-adic metric on $L$]
Fix a place $v \in M_F$. A \textbf{$v$-adic metric} on $L$ is a choice of $v$-adic norms $\|\cdot\|_v$ on each fiber $L(x)$ that varies continuously over $x \in V(\overline{F_v})$, i.e. for each $x \in V(\overline{F_v})$, $s_x \in L(x)$ and $\lambda \in \overline{F_v}$, we have 
$$\|\lambda s_x\|_v = |\lambda|_v \|s_x\|_v.$$
Moreover, if $U \subset V$ is an open subset and $s \in \Gamma(U, L)$ is any section of $L$ over $U$, we require
$$x \to \|s(x)\|_v$$
to be a continuous function on $U(\overline{F_v})$.
\end{defn}

\begin{defn}[Adelic metric on $L$]
An \textbf{adelic metric} on $L$ is a family of compatible $v$-adic metrics for each $v \in M_F$. 

Compatibility is defined as follows: let $U \subset V$ be an affine open subset, and pick any identification of $U$ as a Zariski-closed subset of $\BA^n$ over $F$. For each non-archimedean place $v \in M_F$, we can consider the integral points $U(O_{F_v}) = U(F) \cap O_{F_v}^n$. One can check that for any two identifications of $U$ as Zariski closed subsets of $\BA^n$, $U(O_{F_v})$ is the same up to finitely many exceptions of $v$'s.

The adelic compatibility condition we need is then: for any affine open $U$, any non-vanishing section $s \in \Gamma(U, L)$, we have for all but finitely many $v \in M_F$,
$$\|s(x)\|_v \equiv 1$$
for all $x \in U(O_{F_v})$.
\end{defn}

\begin{defn}[Metrized line bundles]
A \textbf{metrized line bundle} $\CL = (L, \|\cdot\|)$ on a projective variety $V$ over $F$, is a line bundle $L$ with a specified adelic metric $\|\cdot\|$.

$\CL$ is said to be \textbf{ample} if the underlying line bundle $L$ is ample.
\end{defn}

\begin{defn}[Height on $V(F)$ induced by metrized line bundle]
Let $\CL = (L, \|\cdot\|)$ be a metrized line bundle on a projective variety $V$ over $F$.

For any $x \in V(F)$, take any non-vanishing section $s \in \Gamma(U, L)$ over some open subset $U$ containing $x$. We define the height (relative to $F$) to be
$$H_{\CL, F, s}(x) := \prod_{v \in M_{F}} \|s(x)\|_v^{-1}.$$
\end{defn}
We record the properties of heights in the following proposition.
\begin{prop}(\cite[Section 2.4]{ChambertLoir})
The height induced by metrized line bundles satisfies:
\begin{enumerate}[label=(\alph*)]
    \item Height is well-defined: it does not depend on the open subset $U$ and the section $s \in \Gamma(U, L)$ chosen. We hence use $H_{\CL, F}(x)$ to denote this relative height.
    \item If $\CL, \CL'$ are two metrized line bundles on $V$ over $F$, with the same underlying line bundle, then there exists constants $c_1, c_2 > 0$ such that for all $x \in V(F)$,
    $$c_1 H_{\CL, F}(x) \leq H_{\CL', F}(x) \leq c_2 H_{\CL, F}(x).$$
    \item (Northcott property) Let $\CL = (L, \|\cdot\|)$ be an ample metrized line bundle on $V$ over $F$. The set
    $$\{x \in V(F): H_{\CL, F}(x) \leq B\}$$
    is finite, for any given $B$.
    \item (Product) If $\CL_1 = (L_1, \|\cdot\|_1)$, $\CL_2 = (L_2, \|\cdot\|_2)$ are two metrized line bundles on $V$  over $F$, we can form their tensor product of the metrized line bundle as follows. 
    
    On the line bundle $L = L_1 \otimes L_2$, there is a unique adelic metric $\|\cdot\|$ such that: for any $v \in M_F$, any open $U \subset V$, any $x \in U(F_v)$, and sections $s_1 \in \Gamma(U, L_1)$, $s_2 \in \Gamma(U, L_2)$, we have
    $$\|(s_1 \otimes s_2)(x)\|_v = \|s_1(x)\|_{1, v} \|s_2(x)\|_{2, v}.$$
    Moreover, the induced height satisfies:
    $$H_{\CL_1 \otimes \CL_2, F}(x) = H_{\CL_1, F}(x) H_{\CL_2, F}(x)$$
    for any $x \in V(F)$.
    \item (Pullback) Let $V'$ be another variety over $F$, and $f: V' \to V$ be a morphism over $F$. Let $\CL = (L, \|\cdot\|_V)$ be a metrized line bundle on $V$.
    
    We can define the metrized line bundle $f^*L$ on $V'$ as follows. On the pullback line bundle $f^*L$, there is a unique adelic metric $\|\cdot\|_{V'}$ such that: for any $v \in M_F$, any open $U \subset V$, any $x \in f^{-1}(U)(F_v)$, and section $s \in \Gamma(U, L)$, we have
    $$\|f^*s(x)\|_{V', v} = \|s(f(x))\|_{V, v}.$$
    Moreover, the induced height satisfies:
    $$H_{f^*\CL, F}(x) = H_{\CL, F}(f(x))$$
    for any $x \in V'(F)$.
    \item If $F'/F$ is a finite extension of number fields, one may consider the base change ${\CL_{F'} = (L_{F'}, \|\cdot\|_{L, F'})}$. The $\|\cdot\|_{L, F'}$ is defined as follows: if $w \in M_{F'}$ lying over $v \in M_F$, let $d_{w|v} := [F'_w: F_v]$, and define
    $$\|\cdot\|_{L, F', w} := \|\cdot\|_v^{d_{w|v}}.$$
    For $x \in V(F) \subset V(F')$, the base change satisfies
    $$H_{\CL_{F'}, F'}(x) = H_{\CL, F}(x)^{[F': F]}.$$
    Hence we can define the \textbf{absolute multiplicative height} for $x \in V(\overline{F})$ by
    $$H_{\CL}(x) := H_{\CL_{F'}, F'}(x)^{1/[F': \BQ]}$$
    if $F' / F$ is any number field such that $x \in V(F')$.
\end{enumerate}
\end{prop}

\begin{example}[Standard height on $\BP^n$]
Let $F$ be a number field, and consider $\BP^n$ over $F$ with the line bundle $O(1)$. Let $s_0, \cdots, s_n$ be the standard basis of global sections $\Gamma(\BP^n, O(1))$: on open set $U_j := \{x_j \neq 0\}$, the section $s_i$ is defined by 
$$s_i(x) := \frac{x_i}{x_j}.$$
One can define an adelic metric on $O(1)$, such that for any $x \in \BP^n(F)$, and $s \in \Gamma(\BP^n, O(1))$,
$$\|s(x)\|_v = \frac{|s(x)|_v}{\max\{|s_0(x)|_v, \cdots, |s_n(x)|_v\}}.$$
We will call this the standard metric on $O(1)$, and denote the metrized line bundle as $\mathbf{O(1)}$.

One can check that the absolute height induced from this metric is exactly the absolute multiplicaive Weil height: for $x \in \BP^n(\overline{F})$,
$$H_{\mathbf{O(1)}}(x) = \prod_{v \in M_F} \max\{|x_0|_v, \cdots, |x_n|_v\}^{1/[F:\BQ]}.$$
Abusing notation, we may omit $\mathbf{O(1)}$ and just use $H(x)$ to denote the absolute Weil height on $\BP^n$, as before.
\end{example}

\begin{example}[Standard height on $\Sym^2 \BP^1$]
\label{Example_Standard_Height_Sym2P1}
We first start off with some properties of $\Sym^2 \BP^1$. A reference for this example is le Rudulier's thesis \cite[Section 4 and Section 6]{leRudulier}.

Let $F$ be a number field, and consider $\BP^1$ over $F$. $\BZ/2\BZ$ acts on $(\BP^1)^2$ by swapping coordinates. There is a natural quotient map,
$$\pi: (\BP^1)^2 \to \Sym^2 \BP^1.$$
On the other hand, there is a map $\sigma: (\BP^1)^2 \to \BP^2$ that sends
$$\sigma([a_1, b_1], [a_2, b_2]) = [a_1a_2, a_1b_2 + a_2b_1, b_1b_2].$$

\begin{lemma}[Properties of $\Sym^2 \BP^1$]\cite[Section 4]{leRudulier}
\label{Lemma_Properties_Sym2_P1}
\begin{enumerate}[label=(\alph*)]
    \item There is a unique isomorphism
    $$\epsilon: \Sym^2 \BP^1 \to \BP^2$$
    defined over $F$, such that $\sigma = \epsilon \circ \pi$.
    \item Let $p_i: (\BP^1)^2 \to \BP^1$ is the i-th coordinate projection map, and let
    $$O(1, 1) = p_1^*O(1) \otimes p_2^*O(1)$$
    be the line bundle on $(\BP^1)^2$. We denote $O(1, 1)$ metrized with the pullback/product metric as $\mathbf{O(1, 1)}$.

    There is a unique metrized line bundle $\CL$ on $\Sym^2 \BP^1$ such that 
    $$\pi^*\CL = \mathbf{O(1, 1)},$$
    with the underlying line bundle $L = \epsilon^* O_{\BP^2}(1)$. Moreover, the induced height satisfies
    $$H_{(\BP^1)^2, \mathbf{O(1, 1)}} \left(([a_1, b_1], [a_2, b_2])\right) = H_{\BP^1, \mathbf{O(1)}}([a_1, b_1])H_{\BP^1, \mathbf{O(1)}}([a_2, b_2]).$$
\end{enumerate}
\end{lemma}
The line bundle (resp. metric, height) in the lemma would be called the \textbf{standard line bundle (resp. metric, height)} on $\Sym^2 \BP^1$; by abuse of notation, we may denote this line bundle on $\Sym^2 \BP^1$ also as $\mathbf{O(1, 1)}$. When we consider the pushforward of this line bundle to $\BP^2$, we will call the line bundle (resp. metric, height) the \textbf{product line bundle (resp. metric, height)} on $\BP^2$.
    
To make things concrete, we write down how the product metric on $\BP^2$ is defined in a special case \cite[Section 6]{leRudulier}. Let $x_0, x_1, x_2$ be the coordinates of $\BP^2$. Let $U_0 := \{x_0 \neq 0\}$, and $s_0 = \frac{x_0}{x_0}$ be the section $\equiv 1$ on $U_0$.

For a place $v$ and a point $x \in U_0(F)$, the product metric is
$$\|s_0(x)\|_v = \prod_{\stackrel{a \in \overline{F_v}}{x_0 a^2 - x_1 a + x_2 = 0}} \max\{1, |a|_v\}^{-1}.$$
\end{example}

\begin{example}[Heights on $X_1(G)$]
Let $F$ be a number field. Let $G$ be a strongly admissible directed graph, and $X_1(G)$ be the associated dynamical modular curve defined over $\BQ$ (hence $F$).

Recall from Proposition \ref{Proposition_Properties_Of_X1G} that we have a forgetful map
$$\pi_G: X_1(G) \to \BP^1.$$

Consider the line bundle $\mathbf{O(1)}$ on $\BP^1 / F$ with standard metric, and pull this back to the metrized line bundle $\pi_G^* \mathbf{O(1)}$ on $X_1(G)$. This induces a height on $X_1(G)(\overline{F})$:
$$H_{\pi_G^* \mathbf{O(1)}} (x) = H_{\mathbf{O(1)}} (\pi_G (x)) = H (\pi_G (x)),$$
for any $x \in X_1(G)(\overline{F})$, where $H$ is the absolute multiplicative Weil height on $\BP^1$.
\end{example}

\begin{example}[Heights on $\Sym^2 X_1(G)$]
Let $F$ be a number field. Let $G$ be a strongly admissible directed graph, and $X_1(G)$ be the associated dynamical modular curve.

From Proposition \ref{Proposition_Properties_Of_X1G}, we have a forgetful map:
$$\pi_G: X_1(G) \to \BP^1.$$
This induces a corresponding map on the symmetric square:
$$\Sym^2 \pi_G: \Sym^2 X_1(G) \to \Sym^2 \BP^1.$$
Consider the standard metrized line bundle $\mathbf{O(1, 1)}$ on $\Sym^2 \BP^1$ from Example \ref{Example_Standard_Height_Sym2P1}, and pulls it back to $\pi_G^* \mathbf{O(1, 1)}$ on $\Sym^2 X_1(G)$. This induces a height on $\Sym^2 X_1(G)(\overline{F})$.

Concretely, if $P = \{Q_1, Q_2\} \in \Sym^2 X_1(G)(\overline{F})$ corresponds to the degree 2 divisor $[Q_1] + [Q_2]$ on $X_1(G)$ with $Q_1, Q_2 \in X_1(G)(\overline{F})$, then
$$H_{\Sym^2 X_1(G), \pi_G^*\mathbf{O(1, 1)}} (P) = H_{\BP^1, \mathbf{O(1)}}(\pi_G(Q_1))H_{\BP^1, \mathbf{O(1)}}(\pi_G(Q_2)).$$
\end{example}

\subsection{Batyrev-Manin-Peyre conjectures}
Let $F$ be a number field of degree $d$ over $\BQ$. Let $V$ be a smooth projective variety over number field $F$, and $\CL$ an ample, metrized line bundle on $V$. This induces the relative height $H_{\CL, F}$ on $V(F)$; Northcott property implies that
$$\#\{x \in V(F): H_{\CL, F}(x) \leq B\}$$
is finite, for any given $B$. Let
$$N(V, \CL, B) := \#\{x \in V(F): H_{\CL, F}(x) \leq B\}.$$
When $V$ has many $F$-rational points, it is natural to understand how quickly $N(V, \CL, B)$ grows as $B \to \infty$.

For smooth projective Fano varieties $V$ over $F$, Batyrev-Manin's conjectures \cite{BatyrevManin} (refined by Peyre \cite{Peyre} for anticanonical line bundle and Batyrev-Tschinkel \cite{BatyrevTschinkel} for more general $\CL$'s) predict that
$$N(V, \CL, B) = c_V(\CL) B^{\alpha_V(\CL)}(\log B)^{\beta_V(\CL) - 1} ( 1 + o(1))$$
for $B \to \infty$, for some constants $\alpha, \beta, c$ depending only on $(V, \CL)$. We will review their conjectures below; see surveys  \cite{Tschinkel} and \cite{ChambertLoir} for more information.

\begin{remark}
Since the absolute height satisfies
$$H_{\CL}(x) = H_{\CL, F}(x)^{1/d}$$
for any $x \in V(F)$, we have 
\begin{align*}
\#\{x \in V(F): H_{\CL}(x) \leq B\} = \,& \#\{x \in V(F): H_{\CL, F}(x) \leq B^d\} \\
= \, & c_V(\CL) B^{d\alpha_V(\CL)}(d\log B)^{\beta_V(\CL) - 1} ( 1 + o(1)).
\end{align*}
\end{remark}

The Batyrev-Manin-Peyre conjectures is known for many kinds of varieties, in particular for $\BP^n$ with any metrized line bundle \cite{FrankeManinTschinkel}. For non-Fano varieties, asymptotic result of this type is sometimes available as well; an example is N\'{e}ron's theorem for abelian varieties \cite{Neron}.

In the rest of this section, we define the constants $\alpha_V(\CL)$, $\beta_V(\CL)$ and $c_V(\CL)$ that shows up in the conjectural asymptotic formula. We will use them to intrepret the constants that show up in Theorem \ref{Theorem_Main_Degree1} and \ref{Theorem_Main_Degree2}. In the next section, we will document known results of Batyrev-Manin-Peyre conjectures and N\'{e}ron's theorem. These results will be used to deduce asymptotic formulas for $N_{F, 1}(P, B)$ and $N_{F, 2}(P, B)$ mentioned in the introduction.

\subsubsection{Constants $\alpha_V(\CL)$ and $\beta_V(\CL)$.}
Let $V$ be a smooth projective variety over number field $F$, and $\CL = (L, \|\cdot\|)$ an ample metrized line bundle on $V$.

Define:
\begin{itemize}
    \item $\Pic(V)_{\BR} := \Pic(V) \otimes_{\BZ} \BR$.
    \item $\Lambda_{eff} \subset \Pic(V)_{\BR}$ to be the effective cone, i.e. the closed cone in $\Pic(V)_{\BR}$ generated by effective divisors on $V$.
    \item $K_V$ to be the canonical line bundle on $V$.
\end{itemize}
Then we can define
$$\alpha_V(\CL) := \inf\{\alpha \in \BR: \alpha L + K_V \in \Lambda_{eff}\},$$
and $\beta_V(\CL)$ be the co-dimension of the minimal face of $\Lambda_{eff}$ containing $\alpha_V(\CL) L + K_V$.
\begin{remark}
Note that $\alpha_V(\CL)$ and $\beta_V(\CL)$ only depends on $L$, and does not depend on the chosen metric on $L$.
\end{remark}
\begin{example}
\label{Example_Pn_alpha_beta}
For $V = \BP^n$, it is well-known that 
$$Pic(V) \cong \BZ,$$ 
generated by the line bundle $O(1)$. Hence $Pic(V)_{\BR} \cong \BR$; under this isomorphism $O(1)$ corresponds to 1, and $\Lambda_{eff}$ corresponds to $[0, \infty)$.

Moreover, $K_V \cong O(-n-1)$. Hence for $\CL = O(1)$ with any metric, we see that 
$$\alpha_V(\CL) = \inf\{\alpha \in \BR: \alpha \cdot 1 - (n+1) \ge 0\} = n+1.$$
Since $\alpha_V(\CL)L + K_V = 0$, the minimal face of $\Lambda_{eff}$ containing $0$ is $0$ itself. Thus the co-dimension $\beta_V(\CL) = 1$.
\end{example}

\begin{example}
\label{Example_Pn_pullback_alpha_beta}
Consider a morphism 
$$f: \BP^n \to \BP^n$$ 
of degree $k$, and consider the metrized line bundle $\CL := f^*\mathbf{O(1)}$ in the source $V = \BP^n$. Note that the underlying bundle $L \cong O(k)$.

As in last example, $Pic(V)_{\BR} \cong \BR$, where $O(1)$ corresponds to 1, and $\Lambda_{eff}$ corresponds to $[0, \infty)$. Moreover, $K_V \cong O(-n - 1)$. Hence,
$$\alpha_V(\CL) = \inf\{\alpha \in \BR: \alpha \cdot k - (n+1) \ge 0\} = \frac{n+1}{k}.$$
Again since $\alpha_V(\CL)L + K_V = 0$, the minimal face of $\Lambda_{eff}$ containing $0$ is $0$ itself, and the co-dimension $\beta_V(\CL) = 1$.
\end{example}

\begin{example} 
\label{Example_Sym2P1_pullback_alpha_beta}
Consider a morphism 
$$f: \BP^1 \to \BP^1$$ 
of degree $k$. This induces a map
$$\Sym^2 f: \Sym^2 \BP^1 \to \Sym^2 \BP^1,$$
also of degree $k$. Consider the metrized line bundle $\CL := (\Sym^2 f)^* \mathbf{O(1, 1)}$ in the source $V = \Sym^2 \BP^1$.

Since $Sym^2 \BP^1 \cong \BP^2$ with $O(1, 1)$ corresponding to $O(1)$, we see that 
\begin{itemize}
    \item $Pic(V)_{\BR} \cong \BR$ with $O(1, 1)$ corresponds to 1, and $\Lambda_{eff}$ corresponds to $[0, \infty)$.
    \item The canonical line bundle $K_V$ corresponds to $-3$ (since it is $O(-3)$ on $\BP^2$).
    \item $\CL$ corresponds to $k$ (since $\Sym^2 f$ has degree $k$)
\end{itemize}
Hence,
$$\alpha_V(\CL) = \inf\{\alpha \in \BR: \alpha \cdot k - 3 \ge 0\} = \frac{3}{k}.$$
Again since $\alpha_V(\CL)L + K_V = 0$, the minimal face of $\Lambda_{eff}$ containing $0$ is $0$ itself, and the co-dimension $\beta_V(\CL) = 1$.
\end{example}

\subsubsection{Tamagawa measure induced by metrized line bundle}
\label{Subsection_Batyrev_Manin_Tamgawa_measure}
The definition of $c_V(\CL)$ will involve the volume of (closure of) $V(F)$ in the adelic space $V(\BA_F)$. Hence we need to first define a measure on $V(\BA_F)$. In this section, we would define such Tamagawa measure. More details can be found in \cite[Section 3]{BatyrevTschinkel} and \cite[Section 3.3]{ChambertLoir}.

\begin{defn}[$\CL$-primitive variety]
Let $V$ be a smooth projective variety over number field $F$, and $\CL = (L, \|\cdot\|)$ be an ample metrized line bundle on $V$.

$V$ is called $\CL$-primitive if $\alpha_V(\CL) > 0$ is rational, and if there exists $k \in \BN$ such that 
$$D := k\left(\alpha_V(\CL) L + K_V\right)$$
is a \textit{rigid} effective divisor, i.e. an effective divisor satisfying $\dim H^0(V, O_V(vD)) = 1$ for all $v > 0$.
\end{defn}
\begin{example}
Let $V$ be a Fano variety over number field $F$, whose anticanonical bundle $-K_V$ is ample. Let $\CL$ be $-K_V$, equipped with the pullback metric from $\mathbf{O(1)}$ via the anticanonical embedding.

Now
$$\alpha_V(\CL) = \alpha_V(-K_V) = 1.$$ 
Moreover, $L + K_V = 0$, and for all $v > 0$,
$$\dim H^0(V, O_V(v \cdot 0)) = \dim H^0(V, O_V) = 1.$$
Hence $V$ is $\CL$-primitive.
\end{example}
\begin{example}
Consider a morphism $f: \BP^n \to \BP^n$ of degree $k$, and consider $\CL := f^*\mathbf{O(1)}$ in the source $V = \BP^n$. Note that the underlying bundle $L \cong O(k)$.

We showed in Example \ref{Example_Pn_pullback_alpha_beta} that $\alpha_V(\CL) = \frac{n+1}{k}$. Again, 
$$\alpha_V(\CL) L + K_V = 0.$$ 
For all $v > 0$,
$$\dim H^0(V, O_V(v \cdot 0)) = \dim H^0(V, O_V) = 1.$$
Hence $\BP^n$ is $\CL$-primitive.
\end{example}

\begin{example} 
Consider a morphism 
$$f: \BP^1 \to \BP^1$$ 
of degree $k$. This induces a map
$$\Sym^2 f: \Sym^2 \BP^1 \to \Sym^2 \BP^1,$$
also of degree $k$. Consider the metrized line bundle $\CL := (\Sym^2 f)^* \mathbf{O(1, 1)}$ on the source $V = \Sym^2 \BP^1$.

We showed in Example \ref{Example_Sym2P1_pullback_alpha_beta} that $\alpha_V(\CL) = \frac{3}{k}$. Again, 
$$\alpha_V(\CL) L + K_V = 0.$$ 
Again for all $v > 0$,
$$\dim H^0(V, O_V(v \cdot 0)) = \dim H^0(V, O_V) = 1.$$
Hence $\Sym^2 \BP^1$ is $\CL$-primitive.
\end{example}

\textit{Local Tamagawa measures.} Let $V$ be a $\CL$-primitive variety over $F$, and let $v \in M_F$ be a place. We now define a local measure on $V(F_v)$. 

By definition of $\CL$-primitivity, there is $k \in \BN$ such that
$$D = k(\alpha_V(\CL) L + K_V)$$
is a rigid effective divisor. By scaling up $k$, we can further assume that $k\alpha_V(\CL)$ is also an integer. Choose a $F$-rational section $g \in H^0(V, O_V(D))$; by rigidity, $g$ is unique up to a multiple of $F^{\times}$. Choose local analytic coordinates $dx_{1,v}, \cdots, dx_{n, v}$ in a neighborhood $U_x$ of $x \in V(F_v)$. We can then write $g$ as
$$g = c s^{k \alpha_V(L)} \left(dx_{1, v} \wedge \cdots \wedge dx_{n, v}\right)^k$$
where $s \in \Gamma(U_x, L)$ and $c \in \Gamma(U_x, O_V^{\times})$. This defines a local $v$-adic measure
$$|c|_v^{1/k} \|s\|_v^{\alpha_V(L)} |dx_{1, v} \cdots dx_{n,v}|$$
on $U_x$; here $|dx_{1,v} \cdots dx_{n,v}|$ is the additive Haar measure on $U_x$. One checks that this measure is well-defined under change of coordinates, hence it glues to a measure $\omega_{V, L, v, g}$ on the entire $V(F_v)$.

\begin{example}
\label{Example_P1_Local_Tamagawa_Measure}
Consider $\BP^1$ defined over $\BQ$, and consider the line bundle $\CL = \mathbf{O(1)}$ on $\BP^1$ with the standard metric. We will write down a local Tamagawa measure, and compute the total volume of $\BP^1(\BQ_v)$ for each place $v \in M_{\BQ}$.

We saw in Example \ref{Example_Pn_alpha_beta} that $\BP^1$ is $\CL$-primitive with $\alpha_V(\CL) = 2$. Since 
$$\alpha_V(\CL) L + K_V = 0,$$ 
we can take $k = 1$, $D = 0$, and $g \in H^0(V, O_V)$ be the constant section $g \equiv 1$. 

Let $x_0, x_1$ be the homogeneous coordinates of $\BP^1$, and let $U_i := \{ x_i \neq 0\}$. Let $w = \frac{x_1}{x_0}$ be the local coordinate on $U_0$ and $z = \frac{x_0}{x_1}$ be the coordinate on $U_1$. On $U_0 \cap U_1$, the differentials satisfy
$$dz = - \frac{dw}{w^2}.$$
Consider the section $s$ of $O(1)$ corresponding to $x_0$, i.e. $s(w) \equiv 1$ on $U_0$, and $s(z) = z$ on $U_1$. Then the constant global section $1$ can be written as 
$$1 = \begin{cases}
1 \otimes dw & \text{on $U_0$, } \\
z^2 \otimes \left(-\frac{1}{z^2} dz\right) = -\frac{1}{z^2} \left(z^2 \otimes dz\right) & \text{on $U_1$.}
\end{cases}$$
Hence in local coordinates, the local Tamagawa measure is 
\begin{align*}
w_{V, \CL, v, g} = & \begin{cases}
\|1\|_v^2 |dw| & \text{on $U_0$,} \\
\left|-\frac{1}{z^2}\right|_v \|z\|_v^2 |dz| = \|1\|_v^2 |dz|  & \text{on $U_1$.}
\end{cases} \\
= & \begin{cases}
\frac{1}{\max\{|w|_v, 1\}^2} |dw| & \text{on $U_0$,} \\
\frac{1}{\max\{|z|_v, 1\}^2} |dz| & \text{on $U_1$.}
\end{cases}
\end{align*}

For each place $v$, let us calculate the total volume of $\BP^1$ with this Tamagawa measure. Since the point $[0:1] \in \BP^1$ has measure 0, it suffices to calculate the volume of $U_0$.

For $v = \infty$, we have
$$
w_{V, \CL, \infty, g}(\BP^1(\BR)) = \int_{\BR} \frac{1}{\max\{|w|, 1\}^2} |dw| = 4.
$$
For each finite prime $p$, recall that the Haar measure is normalized such that $\vol(\BZ_p) = 1$. Now,
\begin{align*}
w_{V, \CL, p, g}(\BP^1(\BQ_p)) =&  \int_{\BQ_p} \frac{1}{\max\{|w|_p, 1\}^2} |dw| \\
=& \sum_{i=-\infty}^{\infty} \int_{|w|_p = p^i} \frac{1}{\max\{|w|_p, 1\}^2} |dw| \\
=& \sum_{i > 0} \vol(|w|_p = p^i) \cdot \frac{1}{p^{2i}} + \sum_{i \leq 0} \vol(|w|_p = p^i) \\ 
=& \vol(|w|_p = 1) \sum_{i > 0} \frac{1}{p^i} + \vol(\BZ_p) \\
=& \left(1 - \frac{1}{p}\right) \cdot \frac{1 / p}{1 - 1/p} + 1 \\
=& 1 + \frac{1}{p}.
\end{align*}
Note that the "adelic volume" of $\BP^1$, i.e. the infinite product 
$$\prod_{v} w_{V, \CL, v, g} (\BP^1(\BQ_v)) = 4 \prod_{p} \left(1 + \frac{1}{p}\right)$$
diverges. By rewriting it as
$$4 \prod_{p} \left(1 + \frac{1}{p}\right) = 4 \prod_p \frac{1 - \frac{1}{p^2}}{1 - \frac{1}{p}} = \lim_{s \to 1} 4 \frac{\zeta(s)}{\zeta(2s)},$$
and taking residue at $s = 1$, we can assign the value $\frac{4}{\zeta(2)}$ to the diverging product.

In a similar way, we will interpret (divergent) adelic volume when we discuss global Tamagawa measures.
\end{example}

\begin{example}
\label{Example_Pn_pullback_Local_Tamagawa_Measure}
Let $F$ be a number field. Consider a morphism $f: \BP^n \to \BP^n$ over $F$ and of degree $k$, and consider $\CL := f^*\mathbf{O(1)}$ in the source $V = \BP^n$.

If $x_0, x_1, \cdots, x_n$ are the homogeneous coordinates of $\BP^n$, let $U_i = \{ x_i \neq 0\}$. As an example, we calculate the measure on $U_0$.

Let $u_i = \frac{x_i}{x_0}$, and let $u_1, \cdots, u_n$ be the local coordinates on $U_0$. Write 
$$f(u_1, \cdots, u_n) = [G_0(1, u_1, \cdots, u_n), \cdots, G_n(1, u_1, \cdots, u_n)]$$ 
where $G_0, \cdots, G_n \in F[x_1, \cdots, x_n]$ are relatively prime, homogeneous polynomials of degree $k$, and that $G_0, \cdots, G_n$ has empty common vanishing locus. 

As in the last example, we can take $D = 0$, and take $g \in H^0(V, O_V)$ to be the constant section 1. A similar calculation shows that on $U_0$, 
$$w_{V, \CL, v, g} |_{U_0} = 
\max\{|G_0(1, u_1, \cdots, u_n)|_v, \cdots |G_n(1, u_1, \cdots, u_n)|_v\}^{-(n+1)/k} |du_1 \cdots du_n|
$$
For each place $v$, let us calculate the total volume of $\BP^n$ with this Tamagawa measure. One checks that the hyperplane $x_0 = 0$ has measure 0, so it suffices to compute volume of $U_0$. Hence,
$$w_{V, \CL, v, g}(\BP^n(F_v)) = \int_{F_v^n} \max\{|G_0(1, u_1, \cdots, u_n)|_v, \cdots |G_n(1, u_1, \cdots, u_n)|_v\}^{-(n+1)/k} |du_1 \cdots du_n|.$$

Note that for all but finitely many non-archimedean places $v$, we have 
$$\max\{|G_0(1, u_1, \cdots, u_n)|_v, \cdots,  |G_n(1, u_1, \cdots, u_n)|_v\}^{(n+1) / k} = \max\{1, |u_1|_v, \cdots, |u_n|_v\}^{n+1}.$$
It follows that for all but finitely many non-archimedean places $v = v_{\frak{p}}$ with corresponding prime $\frak{p}$, we have 
$$w_{V, \CL, v, g}(\BP^n(F_v)) = \int_{F_v^n} \max\{1, |u_1|_v, \cdots, |u_n|_v\}^{-(n+1)} |du_1 \cdots du_n| = \frac{1 - N_{F/\BQ}(\frak{p})^{-(n+1)}}{1 - N_{F/\BQ}(\frak{p})^{-1}}.$$
\end{example}

\begin{example} 
\label{Example_Sym2P1_pullback_Local_Tamagawa_Measure}
Let $F$ be a number field. Consider a morphism 
$$f: \BP^1 \to \BP^1$$ 
over $F$ of degree $k$. This induces a map
$$\Sym^2 f: \Sym^2 \BP^1 \to \Sym^2 \BP^1,$$
also of degree $k$. Consider the metrized line bundle $\CL := (\Sym^2 f)^* \mathbf{O(1, 1)}$ in the source $V = \Sym^2 \BP^1$.

As in Example \ref{Example_Standard_Height_Sym2P1}, we identify $\Sym^2 \BP^1 \cong \BP^2$. Let $x_0, x_1, x_2$ be the coordinates of $\BP^2$, and let $U_0 = \{x_0 \neq 0\}$. Suppose for the place $v$, $\Sym^2 f$ takes the form
$$\Sym^2 f([x_0, x_1, x_2]) = [F_0(x_0, x_1, x_2), F_1(x_0, x_1, x_2), F_2(x_0, x_1, x_2)]$$
where $F_0, F_1, F_2 \in O_{F_v}[x_0, x_1, x_2]$ are relatively prime, homogeneous polynomials of degree $k$, with empty common vanishing locus.

As in previous examples, we can take $D = 0$, and $g \in H^0(V, O_V)$ to be the constant section 1.
One checks that the local Tamagawa measure on $U_0$ is: 
$$w_{\Sym^2 \BP^1, \CL, v, g} |_{U_0} = |F_0(1, x_1, x_2)|_v^{-3/k} \prod_{\stackrel{a \in \overline{F_v}}{a^2 - x_1 a + x_2 = 0}} \max\{1, |f(a)|_v\}^{-3 / k} |dx_1dx_2|.$$
(See \cite[Section 6]{leRudulier} in the case of $f$ being identity.) The integrand can be seen as the $(-3/k)$-th power of the $v$-adic Mahler measure of the polynomial $F_0 z^2 - F_1 z + F_2$. We recall the $v$-adic Mahler measure (\cite[Section 1.6]{BombieriGruber}) as follows: if $f(x) \in \overline{F_v}[x]$ factorizes as 
$$f(x) = a (x - \alpha_1) \cdots (x - \alpha_n),$$
the $v$-adic Mahler measure of $f$ is defined as
$$M_{v}(f) := |a|_v \prod_{i=1}^n \max\{1, |\alpha_i|_v\}.$$

When $v$ is non-archimedean, one can write
$$w_{\Sym^2 \BP^1, \CL, v, g} |_{U_0} = \max\{|F_0(1, x_1, x_2)|_v, |F_1(1, x_1, x_2)|_v, |F_2(1, x_1, x_2)|_v\}^{-3/k} |dx_1dx_2|,$$
by Gauss lemma \cite[Lemma 1.6.3]{BombieriGruber}. One deduces that for all but finitely many non-archimedean places $v$, we have
$$w_{\Sym^2 \BP^1, \CL, v, g} |_{U_0} = \max\{1, |x_1|_v, |x_2|_v\}^{-3} |dx_1dx_2|.$$

We now compute the volume of $\Sym^2 \BP^1$ using the local Tamagawa measure. The Tamagawa measure on the hyperplane $x_0 = 0$ is 0, so it suffices to compute the volume of $U_0$, thus
$$w_{\Sym^2 \BP^1, \CL, v, g} (\Sym^2 \BP^1 (F_v)) = \int_{F_v^2} |F_0(1, x_1, x_2)|_v^{-3/k} \prod_{\stackrel{a \in \overline{F_v}}{a^2 - x_1 a + x_2 = 0}} \max\{1, |f(a)|_v\}^{-3 / k} |dx_1dx_2|.$$

Hence for all but finitely many non-archimedean places $v = v_{\frak{p}}$ with corresponding prime $\frak{p}$, we have
$$w_{\Sym^2 \BP^1, \CL, v, g} (\Sym^2 \BP^1 (F_v)) = \int_{F_v^2} \max\{1, |x_1|_v, |x_2|_v\}^{-3} |dx_1dx_2| = \frac{1 - N_{F/\BQ}(\frak{p})^{-3}}{1 - N_{F/\BQ}(\frak{p})^{-1}}.$$
\end{example}

\textit{Global Tamagawa measures.} Given a $\CL$-primitive variety $V$, we now want to define the global Tamagawa measure on $V(\BA_F)$. Ideally we can just take the product measure of $V(F_v)$ at all places $v \in M_F$. However, as we saw in Example \ref{Example_P1_Local_Tamagawa_Measure}, the infinite product diverges.

Nonetheless, we can make sense of the infinite divergent product similar to Example \ref{Example_P1_Local_Tamagawa_Measure}. Let $k \in \BN$ be such that $k\alpha_V(\CL)$ is also an integer and that 
$$D = k(\alpha_V(\CL) L + K_V)$$
is a rigid effective divisor. Think of $D$ as a Weil divisor, and let $D_1, \cdots, D_l$ be the irreducible components of $D$. Define
$$\Pic(V, \CL) := \Pic(V - \bigcup_{i=1}^l D_l)$$
The Galois group $Gal(\overline{F}/F)$ acts on $\Pic(V, \CL)$. Let $S \subset M_F$ be a finite set of places where the datum $(V, (D_1, \cdots, D_l))$ has bad reduction, including the archimedean places \cite[Definition 3.3.5]{BatyrevTschinkel}. Let $L_v(s, \Pic(X, \CL))$ be the local factor for the Artin $L$-function, and let
$$L_S^*(1, \Pic(X, \CL)) = \res_{s=1}\prod_{v \not\in S} L_v(s, \Pic(X, \CL)).$$
Finally let
$$\lambda_v = 
\begin{cases}
L_v(1, \Pic(X, \CL)) & \text{if $v \notin S$,} \\
1 & \text{if $v \in S$.}
\end{cases}.$$
We can now define the normalized measure $\omega_{V, \CL}$ on $V(\BA_F)$ by
$$\omega_{V, \CL} := |disc(F)|^{-n/2} L_S^*(1, \Pic(X, \CL)) \prod_v \lambda_v^{-1} \omega_{V, \CL, v},$$
where $n = \dim(V)$.

\begin{remark}
While the local Tamagawa measure depends on the choice of a $F$-rational section $g \in H^0(V, O_V(D))$, the global Tamagawa measure does not have this dependency. This follows from the product formula of $F$, as $g$ is unique up to $F^{\times}$. 

We may hence suppress mentioning $g$ in the following examples.
\end{remark}

\begin{example}
\label{Example_P1_Global_Tamagawa_Measure}
Consider $\BP^1 / \BQ$, and $\CL = \mathbf{O(1)}$ with the standard metric. We saw in Example  \ref{Example_Pn_alpha_beta} and Example \ref{Example_P1_Local_Tamagawa_Measure} that $\BP^1$ is $\CL$-primitive with $\alpha_V(\CL) = 2$, and we can choose $k = 1$, $D = 0$.

In this case, $\Pic(V, \CL) \cong \BZ$. All finite places have good reduction, i.e. $S = \{\infty\}$. For each prime $p$, the local $L$-factor at $p$ is exactly $1 + \frac{1}{p}$, and 
$$L_S^*(1, \BZ) = \res_{s=1} \prod_p \left(1 + \frac{1}{p}\right) = \res_{s=1} \frac{\zeta(s)}{\zeta(2s)} = \frac{1}{\zeta(2)}.$$
The normalized Tamagawa measure on $\BP^1(\BA_{\BQ})$ is hence
$$w_{V, \CL} = \frac{1}{\zeta(2)} w_{V, \CL, \infty} \prod_p \frac{w_{V, \CL, v_p}}{1 + 1/p}.$$
In particular, the total volume of $\BP^1(\BA_{\BQ})$ is $\frac{4}{\zeta(2)}$.
\end{example}

\begin{example}
\label{Example_Pn_pullback_Global_Tamagawa_Measure}
Let $F$ be a number field. Consider a morphism $f: \BP^n \to \BP^n$ over $F$ and of degree $k$, and consider $\CL := f^*\mathbf{O(1)}$ in the source $V = \BP^n$. We saw in Example  \ref{Example_Pn_pullback_alpha_beta} and Example \ref{Example_Pn_pullback_Local_Tamagawa_Measure} that $\BP^n$ is $\CL$-primitive with $\alpha_V(\CL) = \frac{n+1}{k}$. Again we can choose $D = 0$, and all finite places have good reduction.

It follows from Example \ref{Example_Pn_pullback_Local_Tamagawa_Measure} that
$$w_{V, \CL} = \frac{\res_{s=1}\zeta_F(s)}{\zeta_F(n+1)} \prod_{v \in M_F} \frac{w_{V, \CL, v}}{\lambda_v},$$
where 
$$\lambda_v = \begin{cases}
    \frac{1 - N_{F/\BQ}(\frak{p})^{-(n+1)}}{1 - N_{F/\BQ}(\frak{p})^{-1}} & \text{if $v < \infty$ corresponds to prime ideal $\frak{p}$,} \\
    1 & \text{if $v \in M_{F, \infty}$.}
\end{cases}$$
\end{example}

\begin{example} 
\label{Example_Sym2P1_pullback_Global_Tamagawa_Measure}
Let $F$ be a number field. Consider a morphism 
$$f: \BP^1 \to \BP^1$$ 
over $F$ of degree $k$. This induces a map
$$\Sym^2 f: \Sym^2 \BP^1 \to \Sym^2 \BP^1,$$
also of degree $k$. Consider the metrized line bundle $\CL := (\Sym^2 f)^* \mathbf{O(1, 1)}$ in the source $V = \Sym^2 \BP^1$.

It follows from Example \ref{Example_Sym2P1_pullback_Local_Tamagawa_Measure} that
$$w_{V, \CL} = \frac{\res_{s=1}\zeta_F(s)}{\zeta_F(3)} 
\prod_{v \in M_F} \frac{w_{V, \CL, v}}{\lambda_v},$$
where 
$$\lambda_v = \begin{cases}
    \frac{1 - N_{F/\BQ}(\frak{p})^{-3}}{1 - N_{F/\BQ}(\frak{p})^{-1}} & \text{if $v < \infty$ corresponds to prime ideal $\frak{p}$,} \\
    1 & \text{if $v \in M_{F, \infty}$.}
\end{cases}$$
\end{example}

\subsubsection{Constant $c_V(\CL)$.} 
\label{Subsection_Batyrev_Manin_c_V}
We will finally define $c_V(\CL)$ in this section. More details can be found in \cite[Section 3]{BatyrevTschinkel}.

Consider an $\CL$-primitive variety $V$ over number field $F$. We defined the relative Picard group $\Pic(V, \CL)$, and the global Tamagawa measure $\omega_{V, \CL}$ on $V(\BA_F)$. We further define a few constants:
\begin{itemize}
    \item $\gamma_V(\CL)$. Define 
    $$\Pic(V, \CL)_{\BR} := \Pic(V, \CL) \otimes_{\BZ} \BR.$$
    There is a natural restriction map
    $$\Pic(V) \to \Pic(V, \CL).$$
    Let $\Lambda_{eff}(V, \CL)$ be the image of the effective cone $\Lambda_{eff}(V)$ under this restriction map. We hence have a triple
    $$(\Pic(V, \CL), \Pic(V, \CL)_{\BR}, \Lambda_{eff}(V, \CL)).$$
    
    Consider the dual $\Pic(V, \CL)_{\BR}^* := \Hom_{\BR} (\Pic(V, \CL)_{\BR}, \BR)$, which comes with a natural pairing 
    $$\Pic(V, \CL)_{\BR} \times \Pic(V, \CL)_{\BR}^* \to \BR$$
    by $(v, y) \to y(v)$. Let 
    $$\Pic(V, \CL)^* := \Hom_{\BZ} (\Pic(V, \CL), \BZ)$$ 
    be the lattice dual to $\Pic(V, \CL)$ under this natural pairing, and let $dy$ be the Lebesgue measure on $\Pic(V, \CL)_{\BR}^*$, normalized by $\vol(\Pic(V, \CL)^*) = 1$.
    
   Consider the dual cone to $\Lambda_{eff}(V, \CL)$ under this natural pairing, i.e.
   $$\Lambda^* := \{f \in \Pic(V, \CL)_{\BR}^*: f(\Lambda_{eff}(V, \CL)) \subset \BZ\}.$$
   Then for $\mathbf{s} \in \Pic(V, \CL)_{\BC}$, where $\Re(\mathbf{s})$ in the interior of $\Lambda_{eff}(V, \CL)$, we can define
   $$\chi_{\Lambda_{eff}(V, \CL)} (\mathbf{s}) = \int_{\Lambda^*} e^{- \langle \mathbf{s}, y \rangle} dy.$$

   Finally we define \footnote{In \cite{BatyrevTschinkel}, $[-K_V]$ was used instead of $[\CL]$ when evaluating $\chi_{\Lambda_{eff}(V, \CL)}$, which was likely a typo. See also this Math Overflow question \cite{SawinMOAnswer}, which shows that using $[\CL]$ is necessary to make the conjecture compatible between $\CL$ and $\CL^{\otimes k}$.}
   $$\gamma_V(\CL) := \chi_{\Lambda_{eff}(V, \CL)} ([\CL]).$$
   \item $\delta_V(\CL) := |H^1(Gal(\overline{F}/F), Pic(V, \CL))|$.
   \item $\tau_V(\CL) := w_{V, \CL}(\overline{V(F)})$, where $\overline{V(F)}$ is the closure of $V(F)$ in $V(\BA_F)$. 
\end{itemize}
We are now ready to define
$$c_V(\CL) := \frac{\gamma_V(\CL)}{\alpha_V(\CL) (\beta_V(\CL) - 1)!}\delta_V(\CL) \tau_V(\CL).$$

\begin{example} 
\label{Example_P1_Peyre_Constant}
Consider $\BP^1$ defined over $\BQ$, and consider the line bundle $\CL = \mathbf{O(1)}$ on $\BP^1$ with the standard metric.
\begin{itemize}
    \item $\alpha_V(\CL) = 2$, $\beta_V(\CL) = 1$ from Example \ref{Example_Pn_alpha_beta}.
    \item $\gamma_V(\CL) = 1$.
    \begin{itemize}
        \item Since $D$ can be taken as 0 in Example \ref{Example_P1_Local_Tamagawa_Measure}, we have 
        $$Pic(V, \CL) = Pic(V) \cong \BZ,$$ 
        under which $L = O(1) \to 1$. Hence 
        $$i: Pic(V, \CL)_{\BR} \cong \BR,$$ 
        where we pick the isomorphism so that $i(O(1)) = 1$.
        \item $\Lambda_{eff}(V, \CL) = \Lambda_{eff}(V)$ maps to $[0, \infty)$ under $i$. The dual cone $\Lambda^*$ is still $[0, \infty)$, under the natural pairing $\langle -, - \rangle$ via multiplication map
        $$\BR \times \BR \to \BR$$
        \item Hence 
        $$\chi_{\Lambda_{eff}(V, \CL)} (s) = \int_0^{\infty} e^{-sy} dy$$
        and 
        $$\gamma_V(\CL) = \chi_{\Lambda_{eff}(V, \CL)} ([\CL]) = \int_0^{\infty} e^{-1 \cdot y} dy = 1.$$
    \end{itemize}
    \item $Pic(V, \CL) = Pic(V) \cong \BZ$ with trivial Galois action. Hence
    $$\delta_V(\CL) = |H^1(Gal(\overline{\BQ}/\BQ), \BZ)| = 1.$$
    \item Since $\BP^1$ satisfies strong approximation, we have $\overline{V(\BQ)} = V(\BA_{\BQ})$, hence 
    $$\tau_V(\CL) = w_{V, \CL}(V(\BA_{\BQ})) = \frac{4}{\zeta(2)}$$
    as calculated in Example \ref{Example_P1_Global_Tamagawa_Measure}.
\end{itemize}
Hence 
$$c_{V}(\CL) = \frac{1}{2 \cdot (1-1)!} \cdot 1 \cdot \frac{4}{\zeta(2)} = \frac{2}{\zeta(2)}.$$
\end{example}

\begin{example}
\label{Example_Pn_pullback_Peyre_Constant}
Let $F$ be a number field. Consider a morphism $f: \BP^n \to \BP^n$ over $F$ and of degree $k$, and consider $\CL := f^*\mathbf{O(1)}$ on the source $V = \BP^n$. 

\begin{itemize}
    \item $\alpha_V(\CL) = \frac{n+1}{k}, \beta_V(\CL) = 1$ from Example \ref{Example_Pn_pullback_alpha_beta}.
    \item $\gamma_V(\CL) = \frac{1}{k}$, from a similar calculation as in Example \ref{Example_P1_Peyre_Constant}, and that 
    $$\int_0^{\infty} e^{-k \cdot y} dy = \frac{1}{k}.$$
    \item $\delta_V(\CL) = 1$ as in Example \ref{Example_P1_Peyre_Constant}.
\end{itemize}
Hence
$$c_{V}(\CL) = \frac{\frac{1}{k}}{\frac{n+1}{k} \cdot (1-1)!} \cdot 1 \cdot \tau_V(\CL) = \frac{1}{n+1} \tau_V(\CL),$$
where the global Tamagawa measure $\tau_V(\CL)$ has the shape in Example \ref{Example_Pn_pullback_Global_Tamagawa_Measure}.
\end{example}

\begin{example} 
\label{Example_Sym2P1_pullback_Peyre_Constant}
Let $F$ be a number field. Consider a morphism 
$$f: \BP^1 \to \BP^1$$ 
over $F$ of degree $k$. This induces a map
$$\Sym^2 f: \Sym^2 \BP^1 \to \Sym^2 \BP^1,$$
also of degree $k$. Consider the metrized line bundle $\CL := (\Sym^2 f)^* \mathbf{O(1, 1)}$ on the source $V = \Sym^2 \BP^1$.
\begin{itemize}
    \item $\alpha_V(\CL) = \frac{3}{k}, \beta_V(\CL) = 1$ from Example \ref{Example_Sym2P1_pullback_alpha_beta}.
    \item $\gamma_V(\CL) = \frac{1}{k}$, from a similar calculation as in Example \ref{Example_P1_Peyre_Constant}, and that 
    $$\int_0^{\infty} e^{-k \cdot y} dy = \frac{1}{k}.$$
    \item $\delta_V(\CL) = 1$ as in Example \ref{Example_P1_Peyre_Constant}.
\end{itemize}
Hence
$$c_{V}(\CL) = \frac{\frac{1}{k}}{\frac{3}{k} \cdot (1-1)!} \cdot 1 \cdot \tau_V(\CL) = \frac{1}{3} \tau_V(\CL),$$
where the global Tamagawa measure $\tau_V(\CL)$ has the shape in Example \ref{Example_Sym2P1_pullback_Global_Tamagawa_Measure}.
\end{example}

\subsubsection{The Batyrev-Manin-Peyre conjectures}
\begin{defn}
Let $V$ be a $\CL$-primitive, smooth projective variety over number field $F$, and let $H_{\CL, F}$ be the induced height relative to $F$. For any Zariski open subset $U \subset V$, let
$$N(U, \CL, B) = \{x \in U(F): H_{\CL, F}(x) \leq B\}.$$
We say that $V$ is \textbf{strongly $\CL$-saturated} if $|N(V, \CL, B)| \to \infty$ as $B \to \infty$, and for any Zariski open subset $U$,
$$\lim_{B \to \infty} \frac{|N(U, \CL, B)|}{|N(V, \CL, B)|} = 1.$$
In other words, we don't have a closed subvariety of $V$ where many $F$-rational points accumulates.
\end{defn}
\begin{conjecture}\cite[Section 3.4 Step 4]{BatyrevTschinkel}
Let $F$ be a number field. Let $V$ be a strongly $\CL$-saturated, smooth projective variety over number field $F$. Let $H_{\CL, F}$ be the induced relative height on $V$.

then
$$
N(V, \CL, B) := \#\{x \in V(F): H_{\CL, F}(x) \leq B\} = c_V(\CL) B^{\alpha_V(\CL)} (\log B)^{\beta_V(\CL) - 1} (1 + o(1)),
$$
as $B \to \infty$.
\end{conjecture}
\begin{remark}
Suppose $[F: \BQ] = d$. Since the absolute height satisfies
$$H_{\CL}(x) = H_{\CL, F}(x)^{1/d}$$
for $x \in V(F)$, we see that 
$$\#\{x \in V(F): H_{\CL}(x) \leq B\} = c_V(\CL) B^{d\alpha_V(\CL)} (d\log B)^{\beta_V(\CL) - 1} (1 + o(1)).$$
\end{remark}

\subsection{Two results on counting rational points by height}
We document two known results of counting rational points by height: Franke-Manin-Tschinkel \cite{FrankeManinTschinkel} and N\'{e}ron \cite{Neron}. They will be the main tools for our theorems.

\subsubsection{Counting rational points on $\BP^n$.}
\begin{thm}
\label{Theorem_GeneralCountingResult_Pn}
Let $F/\BQ$ be a number field of degree $d$, and consider a metrized bundle $\CL' = (O(1), \|\cdot\|)$ on $\BP^n$ over $F$ (not necessarily the standard metric).

Let $f: \BP^n \to \BP^n$ be a morphism over $F$ of degree $k$, and consider the metrized line bundle $\CL := f^*(\CL')$ on the source $V = \BP^n$. Let $H_{\CL}: V(F) \to [1, \infty)$ be the induced (absolute, multiplicative) Weil height. Then
$$\#\{x \in \BP^n(F): H_{\CL}(x) \leq B\} = \frac{1}{n+1}\tau_{\BP^n} (\CL) B^{d(n+1)/k} + O_{\CL, F} (B^{(d(n+1)-1)/k} \log B),$$
as $B \to \infty$.
\end{thm}
\begin{proof}
For $O(1)$ with the standard metric, Franke-Manin-Tschinkel \cite{FrankeManinTschinkel} proved this theorem with $o(B^{d(n+1)/k})$ error, by observing that the associated height zeta function is an Eisenstein series on $GL_{n+1}$; the asymptotics then follows from analytic properties of the Eisenstein series. Peyre \cite[Corollary 6.2.16]{Peyre} verified that the leading constant takes the form $\frac{1}{n+1} \tau_{\BP^n}(\CL)$; in particular, the asymptotics formula is consistent with the predictions of Batyrev-Manin-Peyre (Example \ref{Example_Pn_pullback_Peyre_Constant}). The result for general adelic metric follows from the case of standard metric \cite[Proposition 5.0.1(c), Corollary 6.2.18]{Peyre}.

To get an explicit power-saving error term, one can modify Masser-Vaaler's work \cite{MasserVaaler} to allow more flexibility at finite places. This leads to an error term of shape $O_{\CL, F}(B^{(d(n+1)-1)/k} E)$, where
$$E = \begin{cases}
    \log B & \text{ when $(d, n) = (1, 1)$,} \\
    1 & \text{ otherwise.}
\end{cases}$$
Clearly, this error term is $O_{\CL, F} (B^{(d(n+1)-1)/k} \log B)$.
\end{proof}
\begin{remark}
To the best of author's knowledge, Widmer \cite[Theorem 3.1]{Widmer} is the closest to the generalization of Masser-Vaaler's result that we need. Unfortunately his notion of adelic Lipschitz system is too restrictive for our purpose. 

More precisely, suppose $f = [F_0, \cdots, F_n]$, where $F_0, \cdots, F_n$ are relatively prime homogeneous polynomials of degree $k$, with no common vanishing locus. For each place $v \in M_F$, we want to consider the norms
$$N_v([x_0, \cdots, x_n]) = \max\{|F_0(x_0, \cdots, x_n)|_v, \cdots, |F_n(x_0, \cdots, x_n)|_v\}^{1/k}$$
These norms satisfy condition (i) - (iii) of adelic Lipschitz system \cite[Definition 2.2]{Widmer}, but may not satisfy condition (iv). Nonetheless, a similar approach can still be used to prove the theorem.
\end{remark}
\begin{remark}
The error term can be strengthened to $O_{\CL, F}(B^{(d(n+1) - 1)/k}E')$, where 
$$E' = \begin{cases}
    \log B & \text{ when $(d, n, k) = (1, 1, 1),$} \\
    1 & \text{ otherwise.}
\end{cases}$$
When $(d, n) \neq (1, 1)$, Masser-Vaaler's argument gives the desired result.

When $d = n = 1$ but $k > 1$, we can use Huxley's version for principle of Lipschitz \cite{Huxley} (instead of Masser-Vaaler's version \cite{MasserVaaler}), to get the stronger error term.
\end{remark}

\subsubsection{Counting rational points on abelian varieties}
\begin{thm}
\label{Theorem_GeneralCountingResult_AV}
Let $F/\BQ$ be a number field of degree $d$. Let $A$ be an abelian variety of dimension $n$ over $F$. Let $\CL$ be a metrized line bundle on $A$ over $F$, and let $H_{\CL}: A(F) \to [1, \infty)$ be the induced (absolute, multiplicative) height. Let $r$ be rank of the Mordell-Weil group $A(F)$. 

Assume furthermore that $\CL$, considered as an element in $Pic(A)$, is divisible by 2.
\begin{enumerate}[label=(\alph*)]
    \item The (absolute, logarithmic) N\'{e}ron-Tate height $\widehat{h_{\CL}}: A(F) \to [0, \infty)$, defined by
$$\widehat{h_{\CL}}(x) := \lim_{n \to \infty} \frac{\log H_{\CL}(2^n x)}{4^n},$$
descends to a quadratic form on $A(F)/A(F)_{tor}$. By extension of scalars, it extends to a quadratic form on $\widehat{h_{\CL}}: A(F) \otimes \BR \to [0, \infty)$.
    \item Equip $A(F) \otimes \BR \cong \BR^r$ with a measure via the Lebesgue measure of $\BR$, such that $\covol(A(F)) = 1$. Then we have the asymptotics formula
\begin{align*}
& \#\{x \in A(F): H_{\CL}(x) \leq B\} \\
= & \, |A(F)_{tor}| \cdot \vol\left(\{\mathbf{x} \in A(F) \otimes \BR: \widehat{h_{\CL}}(\mathbf{x}) \leq 1\}\right) \cdot (\log B)^{r/2} + O_{\CL, F}( (\log B)^{(r-1)/2}).
\end{align*}
\end{enumerate}
\end{thm}
\begin{proof}
For (a), this follows from the theory of N\'{e}ron-Tate heights. See \cite[Chapter 5, Theorem 3.1]{Lang} for a reference.

For (b), this is due to N\'{e}ron \cite[Theorem 6]{Neron}. See also \cite[Chapter 5, Theorem 7.5]{Lang} for a reference.
\end{proof}

\section{Examples of main theorems for $8(2, 1, 1)$ over $\BQ$}
In this section, we illustrate Theorem \ref{Theorem_Main_Degree1} and Theorem \ref{Theorem_Main_Degree2} for the portrait $P = 8(2, 1, 1)$ over $\BQ$. We will also prove Theorem \ref{Theorem_Main_Degree1} for this portrait with bare hands, to illustrate the ideas of proof in the general case.

\begin{figure}[h!]
\includegraphics[scale=1]{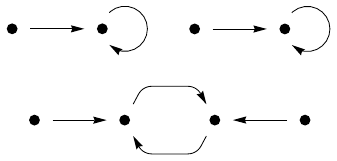}
\caption{The portrait $8(2, 1, 1)$.}
\end{figure}

We first collect some known facts about the dynamical modular curve $X_1(P)$ and the canonical forgetful map $\pi_P$. 

\begin{prop}[{\cite[Theorem 2.1]{Poonen}}]
\label{Proposition_Portrait_8_211_Properties}
Let $P$ be the portrait $8(2, 1, 1)$. 
\begin{enumerate}[label=(\alph*)]
    \item $X_1(P) \cong \BP^1$ over $\BQ$.
    \item The forgetful map $\pi_P: X_1(P) \cong \BP^1 \to \BP^1$ is of degree 4, and has a presentation
    $$\pi_P([x, y]) = [G_0(x, y), G_1(x,y)],$$
    where
    $$G_0(x,y) := -(3x^4 + 10x^2y^2 + 3y^4), \, G_1(x,y) := 4(x^2-y^2)^2 = 4(x-y)^2(x+y)^2.$$
\end{enumerate}
\end{prop}

Let $H: \BQbar \to [1, \infty)$ be the absolute multiplicative Weil height on $\BP^1$.

\subsection{Illustration of Theorem \ref{Theorem_Main_Degree1}}
\label{SectionProofMainThmDegree1Example}
Consider
$$S_{\BQ, 1} (P, B) := \#\{c \in \BQ: \preper(f_c, \BQ) \cong P, H(c) \leq B\}.$$
Our goal is an asymptotic formula for $S_{\BQ, 1} (P, B)$ as $B \to \infty$.
\begin{thm}[Special case of Theorem \ref{Theorem_Main_Degree1} for $8(2, 1, 1)$ over $\BQ$] 
\label{Theorem_Main_Q_Example_8_211_Degree1}
Assuming the Morton-Silverman conjecture for $\BQ$, we have
$$S_{\BQ, 1} (P, B) = \frac{1}{4 \zeta(2)} \cdot area(\CR(1)) \cdot B^{2/4} + O_P(B^{1/4}),$$
as $B \to \infty$. Here
$$\CR(1) := \{(x, y) \in \BR^2: \max\{|G_0(x, y)|, |G_1(x, y)|\} \leq 1\},$$
and the area is computed under the Lebesgue measure of $\BR^2$.
\end{thm}

The main workhorse is an asymptotic formula for 
$$N_{\BQ, 1}(P, B) := \#\{x \in X_1(P)(\BQ): H(\pi_P(x)) \leq B\},$$
as $B \to \infty$, provided by the next proposition.
\begin{prop}\label{RationalPointsX1G_Genus0_Q_Example_8_211}
We have
$$N_{\BQ, 1}(P, B) = \frac{1}{\zeta(2)} \cdot area(\CR(1)) \cdot B^{2 / 4} + O_P(B^{1 / 4}).$$
as $B \to \infty$. Here
$$\CR(1) := \{(x, y) \in \BR^2: \max\{|G_0(x, y)|, |G_1(x, y)|\} \leq 1\},$$
and the area is computed under the Lebesgue measure of $\BR^2$.
\end{prop}
\begin{proof}
We follow the general strategy of Harron-Snowden \cite{HarronSnowden} to get the asymptotic formula.

We want to compute
$$N_{\BQ, 1}(P, B) = \#\{[a, b] \in \BP^1(\BQ): H([G_0(a,b), G_1(a,b)]) \leq B\}.$$
This equals
$$N_{\BQ, 1}(P, B) = \frac{1}{2}\#\{(a, b) \in \BZ^2: gcd(a,b) = 1, H([G_0(a,b), G_1(a,b)]) \leq B\}.$$
where the factor of $2$ comes from the units $\pm 1$ of $\BZ$.

\textit{Step 1: Local conditions.} Since the calculation of height requires simplification of $[G_0(a,b), G_1(a,b)]$ to lowest terms, we first analyze $gcd(G_0(a, b), G_1(a, b))$ for specific values of $a, b$. 

\begin{lemma} \label{Lemma_Rational_Points_X1G_8_211_GCD}
If $(a,b) \in \BZ^2$ satisfies $gcd(a,b) = 1$, then $$gcd(G_0(a,b), G_1(a,b)) = 
\begin{cases}
1 & \text{ if $a \not \equiv b \mod 2$,} \\
16 & \text{ if $a,b$ are both odd.}
\end{cases}
$$
\end{lemma}
\begin{proof}
We repeatedly use the observations that for any integers $m, m', n$, 
$$gcd(mm', n) \mid gcd(m,n) \gcd(m', n)$$
and 
$$gcd(m + n, n) = gcd(m, n).$$
In our case,
$$gcd(G_0(a,b), G_1(a,b)) \mid gcd(G_0(a,b), 4) gcd(G_0(a,b), a-b)^2 gcd(G_0(a,b), a+b)^2.$$

When $a \not\equiv b \mod 2$, we now show that $gcd(G_0(a,b), G_1(a,b)) = 1$, by showing that each component on in right hand side is 1.
\begin{itemize}
    \item If $a,b$ have different parity, it is clear that $G_0(a,b)$ is odd. Hence $gcd(G_0(a,b), 4) = 1$.
    \item Since $G_0(a,b) = Q(a,b) (a-b) + G_0(b, b)$ for some $Q(x,y) \in \BZ[x,y]$, we see that
    $$gcd(G_0(a,b), a-b) = gcd(G_0(b, b), a-b) = gcd(-16b^4, a-b).$$
    If $a,b$ have different parity, $gcd(16, a-b) = 1$. If $a,b$ are relatively prime, $gcd(b^4, a-b) = 1$. Therefore $gcd(G_0(a,b), a-b) = 1$.

    By the same argument, $gcd(G_0(a,b), a + b) = 1$.
\end{itemize}
When $a, b$ are both odd, the same argument implies that $gcd(G_0(a,b), G_1(a,b))$ must be a power of 2. Now note that when $a,b$ are odd, we have
$$G_0(a,b) = -(3a^2 + b^2)(b^2 + 3a^2) \equiv 16 \mod 32,$$
$$G_1(a,b) = 4(a-b)^2(a+b)^2 \equiv 0 \mod 32.$$
Hence $gcd(G_0(a,b), G_1(a,b)) = 16$.
\end{proof}

Let
$$N_{\BQ, 1}^1(P, B) := \frac{1}{2}\#\{(a, b) \in \BZ^2: gcd(a,b) = 1, a \not\equiv b \mod 2, |G_0(a,b)| \leq B, |G_1(a,b)| \leq B\}, $$
$$N_{\BQ, 1}^2(P, B) := \frac{1}{2}\#\{(a, b) \in \BZ^2: gcd(a,b) = 1, a \equiv b \equiv 1 \mod 2, |G_0(a,b)| \leq 16B, |G_1(a,b)| \leq 16B\}. $$
By the lemma, we see that 
$$N_{\BQ, 1}(P, B) = N_{\BQ, 1}^1(P, B) + N_{\BQ, 1}^2(P, B).$$

\textit{Step 2: Principle of Lipschitz.}
We will evaluate $N_{\BQ, 1}^1(P, B)$ and $N_{\BQ, 1}^2(P, B)$ by Huxley's version of principle of Lipschitz \cite{Huxley}, which approximates the number of integral points in a compact region by its area.

We first show that the region
$$\CR(B) = \{(x, y) \in \BR^2: \max\{|G_0(x, y)|, |G_1(x, y)|\} \leq B\}$$
is bounded. Since $G_0(x, 1), G_1(x, 1) \in \BZ[x]$ are relatively prime polynomials of the same degree, we have 
$$\max\{|G_0(x, 1)|, |G_1(x, 1)|\} \gg_{G_0, G_1} 1 > 0$$
for all $x \in \BR$. As $G_0(x,y), G_1(x,y)$ are homogeneous of degree 4, we have
$$B \ge \max\{|G_0(x,y)|, |G_1(x,y)|\}  \gg_{G_0, G_1} |y|^4$$
so $|y| \ll_{G_0, G_1} B^{1/4}$. A symmetric argument shows that $|x| \ll_{G_0, G_1} B^{1/4}$ as well.

Being closed and bounded, $\CR(B)$ is compact. Moreover $\CR(B)$ has a rectifiable boundary defined by polynomials. By the Principle of Lipschitz (Huxley's version, \cite{Huxley}), the number of integral points in the region $\CR(B)$ is given by its area up to an error proportional to $(1 - \delta)$-th power of the length of its boundary for some small $\delta > 0$.

Conveniently, $\CR(B)$ is homogeneous in $B$: since $G_0, G_1$ are homogeneous polynomials of degree 4, we have
$$\CR(B) = B^{1 / 4} \CR(1).$$
Hence by principle of Lipschitz,
\begin{align*}
\#\CR(B) \cap \BZ^2 & = area(\CR(B)) + O(len(boundary(\CR(B)))^{1 - \delta}) \\
& = B^{2 / 4} area (\CR(1)) + O_P \left(B^{(1 - \delta)/4}\right)
\end{align*}
for some small $\delta > 0$.

One can get similar results for counting integral points with local conditions (via translating and rescaling $\CR$, so that the local conditions are gone). In our cases, we have
\begin{equation}
\label{Equation_Lipschitz_Principle_Example_1}
\#\{(a, b) \in \BZ^2: a \not\equiv b \mod 2, |G_0(a,b)| \leq B, |G_1(a,b)| \leq B\} = \frac{1}{2} B^{2/4} area(\CR(1)) + O_P \left(B^{(1 - \delta)/4}\right), 
\end{equation}
\begin{equation}
\label{Equation_Lipschitz_Principle_Example_2}
\#\{(a, b) \in \BZ^2: a \equiv b \equiv 1 \mod 2, |G_0(a,b)| \leq B, |G_1(a,b)| \leq B\} = \frac{1}{4} B^{2/4} area(\CR(1)) + O_P \left(B^{(1 - \delta)/4}\right).
\end{equation}

We now compute $N_{\BQ, 1}^1(P, B)$. By M\"obius inversion,
\begin{align*}
N_{\BQ, 1}^1(P, B) &:= \frac{1}{2}\#\{(a, b) \in \BZ^2: gcd(a,b) = 1, a \not\equiv b \mod 2, |G_0(a,b)| \leq B, |G_1(a,b)| \leq B\}\\
&= \frac{1}{2}\sum_{\stackrel{d=1}{gcd(d,2) = 1}}^{B^{1/4}} \mu(d) \#\{(a, b) \in \BZ^2: a \not\equiv b \mod 2, |G_0(a,b)| \leq d^{-4}B, |G_1(a,b)| \leq d^{-4}B\} \\
&= \frac{1}{2} \sum_{\stackrel{d=1}{gcd(d,2) = 1}}^{B^{1/4}} \mu(d) \left(\frac{1}{2} \left(d^{-4}B\right)^{2/4} area(\CR(1)) + O_P\left(\left(d^{-4}B\right)^{(1 - \delta)/4}\right)\right),
\end{align*}
by equation \ref{Equation_Lipschitz_Principle_Example_1}. Simplifying, we get
\begin{align*}
N_{\BQ, 1}^1(P, B) &= \frac{1}{2} \cdot \left(\frac{1}{2} \cdot \sum_{gcd(d,2) = 1} \frac{\mu(d)}{d^2}\right) \cdot area(\CR(1)) \cdot B^{2/4} + O_P(B^{1/4}) \\ 
&= \frac{1}{2} \cdot \left(\frac{1}{2} \cdot \frac{1}{1 - 2^{-2}} \cdot \frac{1}{\zeta(2)}\right) \cdot area(\CR(1)) \cdot B^{2/4} + O_P(B^{1/4}).
\end{align*}

Similarly we compute $N_{\BQ, 1}^2(P, B)$.
\begin{align*}
N_{\BQ, 1}^2(P, B) &:= \frac{1}{2}\#\{(a, b) \in \BZ^2: gcd(a,b) = 1, a \equiv b \equiv 1 \mod 2, |G_0(a,b)| \leq 16B, |G_1(a,b)| \leq 16B\}\\
&= \frac{1}{2}\sum_{\stackrel{d=1}{gcd(d,2) = 1}}^{B^{1/4}} \mu(d) \#\{(a, b) \in \BZ^2: a \equiv b \equiv 1 \mod 2, |G_0(a,b)| \leq 16d^{-4}B, |G_1(a,b)| \leq 16d^{-4}B\} \\
&= \frac{1}{2} \sum_{\stackrel{d=1}{gcd(d,2) = 1}}^{B^{1/4}} \mu(d) \left(\frac{1}{4} \left(16d^{-4}B\right)^{2/4} area(\CR(1)) + O_P\left(\left(16d^{-4}B\right)^{(1 - \delta)/4}\right)\right) \\
&= \frac{1}{2} \cdot \left(\sum_{gcd(d,2) = 1} \frac{\mu(d)}{d^2}\right) \cdot area(\CR(1)) \cdot B^{2/4} + O_P(B^{1/4}) \\ 
&= \frac{1}{2} \cdot \left(\frac{1}{1 - 2^{-2}} \cdot \frac{1}{\zeta(2)}\right) \cdot area(\CR(1)) \cdot B^{2/4} + O_P(B^{1/4}).
\end{align*}
Therefore,
\begin{align*}
N_{\BQ, 1}(P, B) &= N_{\BQ, 1}^1(P, B) + N_{\BQ, 1}^2(P, B) \\
&= \frac{1}{2} \cdot \left(\frac{3}{2} \cdot \frac{1}{1 - 2^{-2}} \cdot \frac{1}{\zeta(2)}\right) \cdot area(\CR(1)) \cdot B^{2/4} + O_P(B^{1/4}) \\
&= \frac{1}{\zeta(2)} \cdot area(\CR(1)) \cdot B^{2/4} + O_P(B^{1/4}).
\end{align*}
\end{proof}

\begin{proof}[Proof of Theorem \ref{Theorem_Main_Q_Example_8_211_Degree1}]
For any strongly admissible graph $G$, let $\pi_G: X_1(G) \to \BP^1$ be the forgetful map. The following are equivalent:
\begin{itemize}
    \item $c \in \BQ$ satisfies $\preper(f_c, \BQ) \supset P$,
    \item $c \in \pi_P (X_1(P)(\BQ))$, $c \neq \infty$.
\end{itemize}
Let $\{G_i\}_{i \in I}$ be the minimal strongly admissible graphs that strictly contains $P$, where $X_1(G_i)$ has a $\BQ$-rational point. Assuming Morton-Silverman conjecture over $\BQ$, this set of graphs is finite; enumerate them as $G_1, \cdots, G_m$.

By inclusion-exclusion principle, we see that
\begin{align*}
& \, \#\{c \in \pi_P(X_1(P) (\BQ)): H(c) \leq B\} - \sum_{i=1}^m \#\{c \in \pi_{G_i}(X_1(G_i) (\BQ)): H(c) \leq B\} \\
\leq & \, S_{\BQ, 1} (P, B) \\
\leq & \, \#\{c \in \pi_P(X_1(P) (\BQ)): H(c) \leq B\}
\end{align*}
From Proposition \ref{ClassificationForGenus01:AdmissibleGraphs}, among $G_i$'s only $G = 10(2, 1, 1)a$ or $G = 10(2, 1, 1)b$ satisfies $\genus(X_1(G)) = 1$; other dynamical modular curves have genus at least 2. It is also known that these two elliptic curves have rank $0$ over $\BQ$ \cite[17.a4, 15.a7]{LMFDB} and hence have finitely many $\BQ$-rational points. By Faltings' theorem, each of the genus $\ge 2$ dynamical modular curves has finitely $\BQ$-rational points as well.

Hence
$$\sum_{i=1}^m \#\{c \in \pi_{G_i}(X_1(G_i) (\BQ)): H(c) \leq B\} = O_P(1)$$
as $B \to \infty$, and 
$$S_{\BQ, 1} (P, B) = \#\{c \in \pi_P(X_1(P) (\BQ)): H(c) \leq B\} + O_P(1).$$
Finally, notice that for each $c \in \pi_P(X_1(P)(\BQ))$, each fiber $\pi_P^{-1}(c)$ has size 4 with finitely many exceptions. This is because each fiber has at most size 4 (as $\deg(\pi_P) = 4$), and whenever there is an $[x, y] \in X_1(P)(\BQ)$ in the fiber, so are $[x, -y], [y, x], [y, -x]$; these points are distinct except for finitely many $c$'s. Hence,
\begin{align*}
S_{\BQ, 1} (P, B) = & \, \#\{c \in \pi_P(X_1(P) (\BQ)): H(c) \leq B\} + O_P(1) \\
= & \, \frac{1}{4 \zeta(2)} \cdot area(\CR(1)) \cdot B^{2/4} + O_P(B^{1/4}).
\end{align*}
\end{proof}

\begin{remark}
The leading coefficient for $S_{\BQ, 1}(P, B)$ matches up with the expectation from Batyrev-Manin-Peyre conjectures calculated in Example \ref{Example_Pn_pullback_Peyre_Constant}.

More precisely, the leading coefficient equals 
$$\frac{1}{4} \cdot \frac{1}{2} \cdot \tau_V (\CL),$$
where $\dfrac{1}{4} = \dfrac{1}{|Aut(\pi_P)|}$, and $\dfrac{1}{2} = \dfrac{1}{\dim(\BP^1) + 1}$. The global Tamagawa measure 
$$\tau_V (\CL) = \frac{1}{\zeta(2)} \cdot area(\CR(1)) \cdot \prod_{p} area(\CR_p(1)),$$ 
can be further broken up into three pieces:
\begin{itemize}
    \item $\frac{1}{\zeta(2)}$ corresponds to the regularization factor $\dfrac{\res_{s=1} \zeta_{\BQ}(s)}{\zeta_{\BQ}(2)}$,
    \item $area(\CR(1))$ corresponds to local Tamagawa measure of $\BP^1$ at the infinite place.
    \item For each finite prime $p$, define
    $$\CR_p(1) := \{(x, y) \in \BQ_p^2: \max\{|G_0(x, y)|_p, |G_1(x, y)|_p\} \leq 1\}.$$
    Consider the Haar measure on $\BZ_p^2$ normalized so that $area\left(\BZ_p^2\right) = 1$. Then $area(\CR_p(1))$ corresponds to the local Tamagawa measure of $\BP^1$ at the finite place $p$.

    One can compute the local Tamagawa measure to find
    $$
    area(\CR_p(1)) = \begin{cases}
        2 & \text{if $p = 2$,} \\
        1 & \text{otherwise.}
    \end{cases}
    $$
    Hence the local Tamagawa measure at finite primes contribute 
    $$\prod_{p} area \left(\CR_p(1)\right) = 2.$$
\end{itemize} 
\end{remark}

\subsection{Illustration of Theorem \ref{Theorem_Main_Degree2}}
\label{SectionProofMainThmDegree2Example}
Consider
$$S_{\BQ, 2}(P, B) := \#\{c \in \BQbar: c \in K \text{ where $K/\BQ$ is of degree 2}, H(c) \leq B, PrePer(f_c, K) \cong P\}.$$
This counts all the $c$'s that has a preperiodic portrait isomorphic to $P$, when considered over some quadratic extension of $\BQ$. Our goal is an asymptotic formula for $S_{\BQ, 2}(P, B)$ as $B \to \infty$.

\begin{lemma}
Using the presentation of $\pi_P$ in Proposition \ref{Proposition_Portrait_8_211_Properties}, and the natural isomorphism $\Sym^2 \BP^1 \cong \BP^2$ in Lemma \ref{Lemma_Properties_Sym2_P1}, we have a presentation $\Sym^2 \pi_P: \Sym^2 X_1(P) \cong \BP^2 \to \BP^2$ of the form
$$\Sym^2 \pi_P ([x_0, x_1, x_2]) = [H_0(x_0, x_1, x_2), H_1(x_0, x_1, x_2), H_2(x_0, x_1, x_2)]$$
where
$$H_0(x_0,x_1,x_2) := 16 x_0^4 - 32 x_0^2 x_1^2 + 16 x_1^4 + 64 x_0^3 x_2 - 64x_0x_1^2x_2 + 96x_0^2x_2^2 - 32x_1^2x_2^2 + 64x_0x_2^3 + 16x_2^4$$
$$H_1(x_0,x_1,x_2) := 24 x_0^4 + 16 x_0^2 x_1^2 + 24 x_1^4 - 32x_0^3 x_2 - 96x_0x_1^2x_2 - 112 x_0^2 x_2^2 + 16x_1^2x_2^2 - 32x_0x_2^3 + 24x_2^4$$
$$H_2(x_0, x_1, x_2) := 9x_0^4 + 30x_0^2x_1^2 + 9x_1^4 - 60x_0^3x_2 - 36x_0x_1^2x_2 + 118x_0^2x_2^2 + 30x_1^2x_2^2 - 60x_0x_2^3 + 9x_2^4$$
are relatively prime, homogeneous polynomials of degree 4.
\end{lemma}
\begin{proof}
This follows from a direct computation in Sage.
\end{proof}

We also recall the Mahler measure, which will be used to state the theorem. If $f(x) \in \BQ[x]$ factorizes as 
$$f(x) = a (x - \alpha_1) \cdots (x - \alpha_n)$$
over $\BC$, the Mahler measure is defined as
$$M_{\infty}(f) := |a| \prod_{i=1}^n \max\{1, |\alpha_i|\}.$$
(Here $|\cdot|$ is the usual Euclidean norm on $\BC$.)

\begin{thm}[Special case of Theorem \ref{Theorem_Main_Degree2} for $8(2, 1, 1)$ over $\BQ$] 
\label{Theorem_Main_Q_Example_8_211_Degree2}
Assume the Morton-Silverman conjecture for $\BQ$ with $d = 2$. Then we have
$$S_{\BQ, 2}(P, B) = \frac{1}{8\zeta(3)} \vol\left(\CS(1)\right) B^{3/2} + O(B)$$
as $B \to \infty$. Here
$$\CS(1) := \{(x_0, x_1, x_2) \in \BR^3: M_{\infty}(H_0 z^2 - H_1 z + H_2) \leq 1\}$$
and the volume is computed under the Lebesgue measure of $\BR^3$.
\end{thm}
We will not prove this now, although it should be provable directly in a similar way as Theorem \ref{Theorem_Main_Q_Example_8_211_Degree1}.

\begin{remark}
The leading coefficient of $S_{\BQ, 2}(P, B)$ admits a local-global interpretation as in the previous theorem, which matches up with the expectation from Batyrev-Manin-Peyre conjectures calculated in Example \ref{Example_Sym2P1_pullback_Peyre_Constant}.

More precisely, the leading coefficient equals
$$\frac{1}{4} \cdot \frac{1}{3} \cdot \tau_V (\CL),$$
where $\dfrac{1}{4} = \dfrac{1}{|Aut(\pi_P)|}$, and $\frac{1}{3} = \dfrac{1}{\dim(\Sym^2 \BP^1) + 1}$. The global Tamagawa measure
$$\tau_V(\CL) = \frac{1}{\zeta(3)} \cdot \left(\frac{3}{2} \vol\left(\CS(1)\right)\right) \cdot \prod_p \vol\left(\CS_p(1)\right)$$
can be further broken up into three pieces where
\begin{itemize}
    \item $\frac{1}{\zeta(3)}$ corresponds to the regularization factor $\frac{\res_{s=1} \zeta_{\BQ}(s)}{\zeta_{\BQ}(3)}$,
    \item $\frac{3}{2}\vol\left(\CS(1)\right)$ corresponds to local Tamagawa measure of $\Sym^2 \BP^1$ at the infinite place. Here the $\frac{3}{2}$ factor comes from change of variables: when calculating $\vol(\CS(1))$, we can first integrate in $x_0$-axis, then relate the remaining integral to $w_{\Sym^2 \BP^1, \CL, \infty, g}$ in Example \ref{Example_Sym2P1_pullback_Local_Tamagawa_Measure}. 
    
    The factor $\frac{3}{2}$ should be interpreted as $\dfrac{\dim(\Sym^2 \BP^2) + 1}{|\BZ^{\times}|}$.
    \item Analogously for each finite prime $p$, we can define the $p$-adic Mahler measure $M_p$, and define
    $$\CS_{p}(1) := \{(x_0, x_1, x_2) \in \BQ_p^3: M_{p}(H_0 z^2 - H_1 z + H_2) \leq 1\}.$$
    Consider the Haar measure on $\BZ_p^3$ normalized so that $\vol\left(\BZ_p^3\right) = 1$. Then $\vol(\CS_p(1))$ corresponds to the local Tamagawa measure of $\Sym^2 \BP^1$ at the finite place $p$. In this specific example, one can compute the local Tamagawa measure to find that $\vol(\CS_p(1)) = 1$ for all finite primes.

    Again, when calculating $\vol(\CS_p(1))$ we can first integrate along the $x_0$-axis, then relate the remaining integral to $w_{\Sym^2 \BP^1, \CL, p, g}$ in Example \ref{Example_Sym2P1_pullback_Local_Tamagawa_Measure}. There will be a factor $\frac{1 - p^{-3}}{1 - p^{-1}}$ analogous to $\frac{3}{2}$, but this factor is cancelled out when we regularize the divergent infinite product (Example \ref{Example_Sym2P1_pullback_Global_Tamagawa_Measure}). Hence there is no extra factor at finite places.
\end{itemize}

Perhaps the leading coefficient of $S_{\BQ, 2}(P, B)$ can be interpreted more naturally as follows: it equals
$$\frac{1}{4} \cdot \frac{1}{2} \cdot \frac{1}{\zeta(3)} \cdot  \vol\left(\CS(1)\right) \cdot \prod_p \vol\left(\CS_p(1)\right),$$
where
\begin{itemize}
    \item $\frac{1}{4} = \frac{1}{|Aut(\pi_P)|}$,
    \item $\frac{1}{2} = \frac{1}{|\BZ^{\times}|}$,
    \item $\frac{1}{\zeta(3)}$ corresponds to the regularization factor $\frac{\res_{s=1} \zeta_{\BQ}(s)}{\zeta_{\BQ}(3)}$,
    \item $\vol\left(\CS(1)\right)$ and $\vol\left(\CS_p(1)\right)$ are the volumes of analogously defined $\CS(1) \subset \BR^3$ and $\CS_p(1) \subset \BQ_p^3$, 
    $$\CS(1) := \{(x_0, x_1, x_2) \in \BR^3: M_{\infty}(H_0 z^2 - H_1 z + H_2) \leq 1\},$$
    $$\CS_{p}(1) := \{(x_0, x_1, x_2) \in \BQ_p^3: M_{p}(H_0 z^2 - H_1 z + H_2) \leq 1\},$$
    under Haar measure.
\end{itemize}
\end{remark}

\section{Proof of Theorem \ref{Theorem_Main_Degree1}}
\subsection{Precise statement of Theorem \ref{Theorem_Main_Degree1}}
Let $F / \BQ$ be a number field of degree $d$, and let $H: \BQbar \to [1, \infty)$ be the absolute multiplicative Weil height. 

Let $\SP_1 = \Gamma_0 \cup \Gamma_1$ be a set of portraits, where
$$\Gamma_0 := \{\emptyset, 4(1, 1), 4(2), 6(1, 1), 6(2), 6(3), 8(2, 1, 1)\}, $$
$$\Gamma_1 := \{8(1, 1)a, 8(1, 1)b, 8(2)a, 8(2)b, 10(2, 1, 1)a, 10(2, 1, 1)b\}.$$
(See Appendix \ref{Appendix_Explicit_Presentation_piG} for pictures of these graphs.) By Proposition \ref{ClassificationForGenus01:AdmissibleGraphs}, $P \in \Gamma_g$ is equivalent to the dynamical modular curve $X_1(P)$ having genus $g$.

By Theorem \ref{Theorem_Base_Degree1}, if there are infinitely many $c \in F$'s such that $\preper(f_c, F) \cong P$, then $P \in \SP_1$. 

\begin{defn}[Constants $a_{F,1}, b_{F,1}, c_{F,1}$]
\label{Definition_Degree1_abc_constants}
Let $P \in \SP_1$ be a portrait. Let $X_1(P)$ be the dynamical modular curve associated to $P$, and let $\pi_P: X_1(P) \to \BP^1$ be the natural forgetful map of degree $k$.
\begin{itemize}
    \item If $P \in \Gamma_0$, then $X_1(P)$ is rational over $F$.
    
    We define $a_{F, 1}(P) = \frac{2d}{k}$; $b_{F, 1}(P) = 0$. For $c_{F, 1}$, suppose $\pi_P: X_1(P) \cong \BP^1 \to \BP^1$ has a presentation
    $$\pi_P(x) = \frac{g_0(x)}{g_1(x)}$$
    for relatively prime $g_0(x), g_1(x) \in F[x]$ (see examples in Appendix \ref{Appendix_Explicit_Presentation_piG}, Table \ref{TableModelOfForgetfulMapGenus0}), then we define
    $$c_{F, 1}(P) = \frac{1}{2 |Aut(\pi_P)|}\frac{\res_{s=1} \zeta_F(s)}{\zeta_F(2)} \prod_{v \in M_F} \frac{w_{\BP^1, \pi_P^* O(1), v}(\BP^1(F_v))}{\lambda_v},$$
    where
    \begin{itemize}
        \item $M_F$ is the set of places of $F$.
        \item $\zeta_F(s)$ is the Dedekind zeta function of $F$.
        \item For each place $v$, $w_{\BP^1, \pi_P^*O(1), v}$ is a local Tamagawa measure of $\BP^1(F_v)$,
        $$w_{\BP^1, \pi_P^*O(1), v}(\BP^1(F_v)) = \int_{F_v} \max\{|g_0(u)|_v, |g_1(u)|_v\}^{-2 / k} du.$$
        \item For each place $v$, $\lambda_v$ is a regularization factor to make sure the infinite product converges,
        $$\lambda_v = \begin{cases}
        1 + \frac{1}{N_{F/\BQ}(\frak{p})} & \text{ if $v < \infty$ corresponds to prime ideal $\frak{p}$,} \\
        1 & \text{if $v \in M_{F, \infty}$.}
        \end{cases}$$
    \end{itemize}
    \item If $P \in \Gamma_1$, then $X_1(P)$ is an elliptic curve over $F$. 
    
    We define $a_{F, 1}(P) = 0$; $b_{F, 1}(P) = \frac{r}{2}$, where $r$ is the rank of the Mordell-Weil Group $X_1(P)(F)$; and
    $$c_{F, 1}(P) = \frac{1}{|Aut(\pi_P)|} \cdot \frac{|X_1(P)(F)_{tor}|V_{r}}{k^{r/2}Reg(X_1(P)/F)^{1/2}}$$
    where $V_{r} := \frac{\pi^{r/2}}{\Gamma\left(\frac{r}{2} + 1\right)}$ is the volume of unit ball in $\BR^r$, and $Reg(X_1(P)/F)$ is the regulator of $X_1(P)$ over $F$.
\end{itemize}
\end{defn}
With $a_{F, 1}, b_{F, 1}, c_{F, 1}$ defined, we now restate Theorem \ref{Theorem_Main_Degree1} with error terms.

\begin{thm}
\label{Theorem_Main_Degree1_Details}
Let $F/\BQ$ be a number field of degree $d$. For each $P \in \SP_1$, define
$$S_{F,1}(P, B) := \#\{c \in F: \preper(f_c, F) \cong P, H(c) \leq B\}.$$

Assume the Morton-Silverman conjecture for the number field $F$. Then for each $P \in \SP_1$, We have the asymptotic formula
$$S_{F,1}(P, B) = c_{F, 1}(P) B^{a_{F, 1}(P)} (\log B)^{b_{F, 1}(P)} + O_{P, F}(Error_{P, F, 1}(B))$$
as $B \to \infty$. Here
$$
Error_{P, F, 1}(B) = \begin{cases}
B^{(2d-1)/k} \log B & \text{when $P \in \Gamma_0$,} \\
(\log B)^{(r-1)/2} + 1 & \text{when $P \in \Gamma_1$.}
\end{cases}
$$
\end{thm}
\subsection{Asymptotic formula for $N_{F, 1}(P, B)$}
The main input to Theorem \ref{Theorem_Main_Degree1}/Theorem \ref{Theorem_Main_Degree1_Details} is an asymptotic formula for
$$N_{F, 1}(P, B) := \#\{x \in X_1(P)(F): H(\pi_P(x)) \leq B\},$$
as $B \to \infty$. It is our goal to obtain such formula in this section (Theorem \ref{Theorem_Main_Degree1_CountingResult_Summary}).

\begin{prop}[Counting rational points on $\BP^1$]
\label{Theorem_Main_Degree1_CountingResult_P1}
Let $F/\BQ$ be a number field of degree $d$, and consider $\mathbf{O(1)}$ on $\BP^1$ over $F$ with the standard metric. 

Let $f: \BP^1 \to \BP^1$ be a morphism over $F$ of degree $k$, and consider the metrized line bundle $\CL := f^*\mathbf{O(1)}$ on the source $V = \BP^1$. Let $H_{\CL}: V(F) \to [1, \infty)$ be the induced (absolute, multiplicative) Weil height. Then
$$\#\{x \in \BP^1(F): H_{\CL}(x) \leq B\} = \frac{1}{2}\tau_{\BP^1} (\CL) B^{2d/k} + O_{\CL, F} (B^{(2d-1)/k} \log B),$$
as $B \to \infty$. Here $\tau_{\BP^1}(\CL)$ is the volume of $\BP^1$ with respect to a global Tamagawa measure, defined in Section \ref{Subsection_Batyrev_Manin_c_V} and calculated in Example \ref{Example_Pn_pullback_Local_Tamagawa_Measure},  \ref{Example_Pn_pullback_Global_Tamagawa_Measure} and \ref{Example_Pn_pullback_Peyre_Constant}.
\end{prop}
\begin{proof}
This is special case of Theorem \ref{Theorem_GeneralCountingResult_Pn} when $n = 1$.
\end{proof}

\begin{prop}[Counting rational points on elliptic curve]
\label{Theorem_Main_Degree1_CountingResult_EC}
Let $F/\BQ$ be a number field of degree $d$, and consider $\mathbf{O(1)}$ on $\BP^1$ over $F$ with the standard metric. 

Let $E$ be an elliptic curve over $F$. Let $f: E \to \BP^1$ be a morphism over $F$ of degree $k$ such that $f(x) = f(-x)$; in particular, $k$ is even.

Consider the metrized line bundle $\CL = f^*\mathbf{O(1)}$ on $E$, and let $H_{\CL}: E(F) \to [1, \infty)$ be the induced (absolute, multiplicative) height. Then
$$\#\{x \in E(F): H_{\CL}(x) \leq B\} = \frac{|E(F)_{tor}|V_{r}}{k^{r/2}Reg(E/F)^{1/2}} (\log B)^{r/2} + O_{\CL, F}((\log B)^{(r-1)/2})$$
as $B \to \infty$. Here 
\begin{itemize}
    \item $r = r(E/F)$ is the rank of the Mordell-Weil group $E(F)$,
    \item $V_{r} := \frac{\pi^{r/2}}{\Gamma\left(\frac{r}{2} + 1\right)}$ is the volume of unit ball in $\BR^r$,
    \item $Reg(E/F)$ is the regulator of $E$ over $F$.
\end{itemize}
\end{prop}
\begin{proof}
From Theorem \ref{Theorem_GeneralCountingResult_AV}, we see that
\begin{align*}
& \#\{x \in E(F): H_{\CL}(x) \leq B\} \\
= & \, |E(F)_{tor}| \cdot \vol\left(\{\mathbf{x} \in E(F) \otimes \BR: \widehat{h_{\CL}}(\mathbf{x}) \leq 1\}\right) \cdot (\log B)^{r/2} + O_{\CL, F}( (\log B)^{(r-1)/2}).
\end{align*}
Here $r = r(E/F)$ is the rank of the Mordell-Weil group $E(F)$, and the volume of $E(F) \otimes \BR$ is induced from Lebesgue measure on $\BR$, normalized such that $\covol(E(F)) = 1$. We fix this measure by taking an integral basis $P_1, \cdots, P_r$ of $E(F) / E(F)_{tor}$, and identify $\BR^r \cong E(F) \otimes \BR$ via
$$(a_1, \cdots, a_r) \to a_1P_1 + \cdots + a_r P_r.$$

Let $\widehat{h_E}$ be the canonical height of $E/F$. For $x \in E(F)$, 
$$\widehat{h_{\CL}}(x) = \lim_{n \to \infty} \frac{\log H(f(2^n x))}{4^n} = k \cdot \widehat{h_E}(x)$$
by the defining property of canonical height \cite[Proposition VIII.9.1]{SilvermanEllipticCurve}. The discriminant of the quadratic form $\widehat{h_E}$,
$$\disc(\widehat{h_E}) = \det\left(\langle P_i, P_j \rangle\right)_{i,j},$$
is exactly the regulator $Reg(E/F)$ by definition. Hence
\begin{align*}
\vol\left(\{\mathbf{x} \in E(F) \otimes \BR: \widehat{h_{\CL}}(\mathbf{x}) \leq 1\}\right) =& \, \vol\left(\left\{\mathbf{x} \in E(F) \otimes \BR: \widehat{h_E}(\mathbf{x}) \leq \frac{1}{k}\right\}\right) \\
=& \, \frac{1}{k^{r/2}} \vol\left(\left\{\mathbf{x} \in E(F) \otimes \BR: \widehat{h_E}(\mathbf{x}) \leq 1\right\}\right) \\
=& \, \frac{V_{r}}{k^{r/2} Reg(E/F)^{1/2}},
\end{align*}
where $V_{r}$ is the volume of unit ball in $\BR^r$.
\end{proof}

With the last two results at hand, we can compute the asymptotics for $N_{F, 1}(P, B)$.

\begin{thm}
\label{Theorem_Main_Degree1_CountingResult_Summary}
Let $P$ be a portrait. Let $X_1(P)$ be the dynamical modular curve associated to $P$, and let $\pi_P: X_1(P) \to \BP^1$ be the natural forgetful map of degree $k$.

Let $H: \BP^1(\BQbar) \to [1, \infty)$ be the absolute, multiplicative Weil height. For a fixed number field $F/\BQ$ of degree $d$, denote
$$N_{F, 1}(P, B) := \#\{x \in X_1(P)(F): H(\pi_P(x)) \leq B\},$$
\begin{enumerate}[label=(\alph*)]
    \item Suppose $P \in \Gamma_0$, then $X_1(P)$ is rational over $F$. We have
    $$N_{F, 1}(P, B) = |Aut(\pi_P)| \cdot c_{F, 1}(P) B^{2d/k} + O_{P, F}(B^{(2d-1)/k} \log B)$$
    as $B \to \infty$. 
    \item Suppose $P \in \Gamma_1$, then $X_1(P)$ is an elliptic curve over $F$. We have
    $$N_{F, 1}(P, B) = |Aut(\pi_P)| \cdot c_{F, 1}(P) (\log B)^{r/2} + O_{P, F}((\log B)^{(r-1)/2})$$
    as $B \to \infty$. Here $r = r(X_1(P)/F)$ is the rank of the Mordell-Weil Group $X_1(P)(F)$.
    \item Suppose $P \notin \Gamma_0 \cup \Gamma_1$, then $X_1(P)$ is a curve of genus $\ge 2$, and 
    $$N_{F, 1}(P, B) = O_{P, F}(1).$$
\end{enumerate}
\end{thm}
\begin{proof}
For part (a), we use Proposition \ref{Theorem_Main_Degree1_CountingResult_P1}. The constant $\tau_{\BP^1}(\CL)$ was calculated in Example \ref{Example_Pn_pullback_Local_Tamagawa_Measure},  \ref{Example_Pn_pullback_Global_Tamagawa_Measure} and \ref{Example_Pn_pullback_Peyre_Constant}. We defined $c_{F, 1}(P)$ such that 
$$\frac{1}{2} \tau_{\BP^1}(\CL) = |Aut(\pi_P)| \cdot c_{F, 1}(P),$$
so part (a) follows from Proposition \ref{Theorem_Main_Degree1_CountingResult_P1} directly.

For part (b), we use Proposition \ref{Theorem_Main_Degree1_CountingResult_EC}. We defined $c_{F, 1}(P)$ such that 
$$\frac{|X_1(P)(F)_{tor}|V_{r}}{k^{r/2}Reg(X_1(P)/F)^{1/2}} = |Aut(\pi_P)| \cdot c_{F, 1}(P),$$
so part (b) follows from Proposition \ref{Theorem_Main_Degree1_CountingResult_EC} directly.

Part (c) follows from Faltings' theorem.
\end{proof}

\subsection{Proof of Theorem \ref{Theorem_Main_Degree1}/Theorem \ref{Theorem_Main_Degree1_Details}}
\begin{lemma}
\label{Lemma_Main_Degree1_Fiber_Count}
With the notations in Theorem \ref{Theorem_Main_Degree1_CountingResult_Summary}, we have
$$\#\{c \in \pi_P(X_1(P)(F)): H(c) \leq B\} = \frac{1}{|Aut(\pi_P)|} N_{F, 1}(P, B) + O_{P, F} (Error_{P, F, 1}(B)).$$
Here 
$$
Error_{P, F, 1}(B) = \begin{cases}
B^{(2d-1)/k} \log B & \text{when $P \in \Gamma_0$,} \\
(\log B)^{(r-1)/2} + 1 & \text{when $P \in \Gamma_1$.}
\end{cases}
$$
\end{lemma}
\begin{proof}
Suppose $c = \pi_P(x)$ for some $x \in X_1(P)(F)$. We want to show that for generic $c \in \pi_P(X_1(P)(F))$,
$$\pi_P^{-1}(c) \cap X_1(P)(F) = \{\sigma(x): \sigma \in Aut(\pi_P)\}.$$
By the definition of $Aut(\pi_P)$, we see that right hand side is a subset of the left. 

First, note that right hand side has size $|Aut(\pi_P)|$ with $O_P(1)$ exceptions of $c$. This is because if right hand side is smaller than expected, it must mean that an unexpected coincidence
$$\sigma(x) = \sigma'(x)$$ 
occured for some $\sigma \neq \sigma' \in Aut(\pi_P)$. Since $\sigma = \sigma'$ is a proper Zariski-closed subset of the curve $X_1(P)$, it is of dimension 0; so there are finitely many exceptional $x$'s, hence finitely many exceptional $c$'s as well. To recap, with $O_P(1)$ exceptions of $c \in \pi_P(X_1(P)(F))$, we have
$$|\pi_P^{-1}(c) \cap X_1(P)(F)| \ge |Aut(\pi_P)|.$$

Finally, we need to show that 
$$|\pi_P^{-1}(c) \cap X_1(P)(F)| > |Aut(\pi_P)|$$
happens rarely. Suppose we have such an exceptional $c \in \pi_P(X_1(P)(F))$, where there exists $x, z \in X_1(P)(F)$ not in the same $Aut(\pi_P)$-orbit, such that
$$\pi_P(x) = \pi_P(z) = c.$$
Then $(x, z)$ is a $F$-rational point on some irreducible component $C$ of
$$X_1(P) \times_{\BP^1} X_1(P) := \{(x, z) \in X_1(P) \times X_1(P): \pi_P(x) = \pi_P(z)\},$$
which is not the graph of any $\sigma \in Aut(\pi_P)$; we call such curve $C$ exceptional. By computing these exceptional curves for each portrait $P \in \Gamma_0 \cup \Gamma_1$ in Sage (see Appendix \ref{Appendix_Explicit_Presentation_piG}), we find that 
\begin{itemize}
    \item If $P \in \Gamma_0$, any exceptional $C$ is smooth and has genus $\ge 1$.
    \item If $P \in \Gamma_1$, any exceptional $C$ is smooth and has genus $\ge 4$.
\end{itemize}
The number of $F$-rational points on these exceptional curves are few; by Theorem \ref{Theorem_Main_Degree1_CountingResult_Summary} applied to genus 1 and genus $\ge 2$ curves, we see that
\begin{align*}
& \, \#\{c \in \pi_P(X_1(P)(F)): |\pi_P^{-1}(c) \cap X_1(P)(F)| > |Aut(\pi_P)|, H(c) \leq B\} \\
\ll_P & \, \#\left\{(x, z) \in (X_1(P) \times_{\BP^1} X_1(P))(F): 
    \begin{aligned}
    & (x, z) \text{ lies in an exceptional curve } \\
    & \text{of $X_1(P)$, $H(\pi_P(x)) \leq B$}
    \end{aligned}
\right\} \\
= & \, O_P(Error_{P, F, 1}(B))
\end{align*}

This shows the rarity of $c \in \pi_P(X_1(P)(F))$ satisfying 
$$|\pi_P^{-1}(c) \cap X_1(P)(F)| > |Aut(\pi_P)|,$$
and finish the proof of lemma.
\end{proof}

\begin{proof}[Proof of Theorem \ref{Theorem_Main_Degree1}/Theorem \ref{Theorem_Main_Degree1_Details}]
Recall the definition of $S_{F, 1}(P, B)$: for each $P \in \SP_1$,
$$S_{F,1}(P, B) := \#\{c \in F: \preper(f_c, F) \cong P, H(c) \leq B\}.$$
Our goal is to get an asymptotic formula for $S_{F, 1}(P, B)$ as $B \to \infty$.

By definition of $X_1(P)$, the following are equivalent:
\begin{itemize}
    \item $c \in F$ satisfies $\preper(f_c, F) \supset P$,
    \item $c \in \pi_P (X_1(P)(F))$, $c \neq \infty$.
\end{itemize}
Let $\{G_i\}_{i \in I}$ be the minimal strongly admissible graphs that strictly contains $P$, where $X_1(G_i)$ has a $F$-rational point. Assuming Morton-Silverman conjecture over $F$, this set of graphs is finite; enumerate them as $G_1, \cdots, G_m$.

By inclusion-exclusion principle, we see that
\begin{align*}
& \, \#\{c \in F: \preper(f_c, F) \supset P, H(c) \leq B\} - \sum_{i=1}^m \#\{c \in F: \preper(f_c, F) \supset G_i, H(c) \leq B\} \\
\leq & \, S_{F, 1} (P, B) \\
\leq & \, \#\{c \in F: \preper(f_c, F) \supset P, H(c) \leq B\}.
\end{align*}
Using the above equivalence, we have
\begin{align*}
& \, \#\{c \in \pi_P(X_1(P) (F)): H(c) \leq B\} - \sum_{i=1}^m \#\{c \in \pi_{G_i}(X_1(G_i) (F)): H(c) \leq B\} \\
\leq & \, S_{F, 1} (P, B) \\
\leq & \, \#\{c \in \pi_P(X_1(P) (F)): H(c) \leq B\}.
\end{align*}
Now note that
\begin{align*}
\#\{c \in \pi_P(X_1(P)(F)): H(c) \leq B\} = & \, \frac{1}{|Aut(\pi_P)|} N_{F, 1}(P, B) + O_{P, F} (Error_{P, F, 1}(B))   \\
= & \, c_{F, 1}(P) B^{a_{F, 1}(P)} (\log B)^{b_{F, 1}(P)} + O_{P, F} (Error_{P, F, 1}(B))   
\end{align*}
by Theorem \ref{Theorem_Main_Degree1_CountingResult_Summary} and Lemma \ref{Lemma_Main_Degree1_Fiber_Count}.

Finally, we show that
$$\sum_{i=1}^m \#\{c \in \pi_{G_i}(X_1(G_i) (F)): H(c) \leq B\}$$
only contributes to the error term. For each $G_i \supset P$,
\begin{itemize}
    \item If $a_{F, 1}(P) > 0$, by Proposition \ref{Proposition_Properties_Of_X1G}(b), $\deg \pi_{G_i} \ge 2 \deg \pi_P$, so either $X_1(G_i)$ has genus 0 and
    $$a_{F, 1}(G_i) = \frac{2d}{\deg \pi_{G_i}} \leq \frac{d}{\deg \pi_P} \leq \frac{2d-1}{\deg \pi_P}, $$
    or $X_1(G_i)$ has genus $\ge 1$ and $a_{F, 1}(G_i) = 0$. In either case, 
    $$\#\{c \in \pi_{G_i}(X_1(G_i) (F)): H(c) \leq B\} = O\left(B^{(2d-1) / \deg \pi_P}\right) = O (Error_{P, F, 1}(B)).$$
    \item If $a_{F, 1}(P) = 0$, then $X_1(P)$ has genus 1. Since the 6 such possible $P$'s (Table \ref{TableModelOfForgetfulMapGenus1}, Appendix \ref{Appendix_Explicit_Presentation_piG}) has no containment relationship, we see that $X_1(G_i)$ has genus $\ge 2$. Faltings' theorem thus implies that 
    $$\#\{c \in \pi_{G_i}(X_1(G_i) (F)): H(c) \leq B\} = O_{G_i, F}(1) = O(Error_{P, F, 1}(B)).$$
\end{itemize}
Hence
\begin{align*}
S_{F, 1}(P, B) = & \, \#\{c \in \pi_P(X_1(P)(F)): H(c) \leq B\} + O_{P, F}(Error_{P, F, 1}(B)) \\
= & \,  c_{F, 1}(P) B^{a_{F, 1}(P)} (\log B)^{b_{F, 1}(P)} + O_{P, F} (Error_{P, F, 1}(B)).
\end{align*}
\end{proof}

\section{Proof of Theorem \ref{Theorem_Main_Degree2}}
\subsection{Precise statement of Theorem \ref{Theorem_Main_Degree2}}
Let $F / \BQ$ be a number field of degree $d$, and let $H: \BQbar \to [1, \infty)$ be the absolute multiplicative Weil height. 

Let $\SP_2 = \Gamma_0 \cup \Gamma_1 \cup \Gamma_2$ be a set of portraits, where
$$\Gamma_0 := \{\emptyset, 4(1, 1), 4(2), 6(1, 1), 6(2), 6(3), 8(2, 1, 1)\}, $$
$$\Gamma_1 := \{8(1, 1)a, 8(1, 1)b, 8(2)a, 8(2)b, 10(2, 1, 1)a, 10(2, 1, 1)b\},$$
$$\Gamma_2 := \{8(3), 8(4), 10(3, 1, 1), 10(3, 2)\}.$$
(See Appendix \ref{Appendix_Explicit_Presentation_piG} for pictures of these graphs.) By Proposition \ref{ClassificationForGenus01:AdmissibleGraphs}, $P \in \Gamma_g$ is equivalent to the dynamical modular curve $X_1(P)$ having genus $g$.

By Theorem \ref{Theorem_Base_Degree2}, if there are infinitely many $c \in \overline{F}$ such that $[F(c): F] \leq 2$ and $\preper(f_c, K) \cong P$ for some $K/F$ of degree 2, then $P \in \SP_2$. 

\begin{defn}[Constants $a_{F,2}, b_{F,2}, c_{F,2}$]
\label{Definition_Degree2_abc_constants}
Let $P \in \SP_2$ be a portrait. Let $X_1(P)$ be the dynamical modular curve associated to $P$, and let $\pi_P: X_1(P) \to \BP^1$ be the natural forgetful map of degree $k$.

\begin{itemize}
    \item If $P \in \Gamma_0$, then $X_1(P)$ is rational over $F$.
    
    We define $a_{F, 2}(P) = \frac{6d}{k}$; $b_{F, 2}(P) = 0$; and
    $$c_{F, 2}(P) = \frac{2}{3|Aut(\pi_P)|}\frac{\res_{s=1} \zeta_F(s)}{\zeta_F(3)} \prod_{v \in M_F} \frac{w_{\Sym^2 \BP^1, \pi_P^* O(1), v}(\Sym^2 \BP^1(F_v))}{\lambda_v},$$
    where
    \begin{itemize}
        \item $M_F$ is the set of places of $F$.
        \item $\zeta_F(s)$ is the Dedekind zeta function of $F$.
        \item For each place $v$, $w_{\Sym^2 \BP^1, \pi_P^*O(1), v}$ is a local Tamagawa measure of $\Sym^2 \BP^1(F_v)$. Suppose $\Sym^2 \pi_P$ takes the form
        $$\Sym^2 \pi_P ([x_0, x_1, x_2]) = [F_0(x_0, x_1, x_2), F_1(x_0, x_1, x_2), F_2(x_0, x_1, x_2)]$$
        where $F_0, F_1, F_2 \in O_{F_v}[x_0, x_1, x_2]$ are relatively prime, homogeneous polynomials of degree $k$. Then
        \begin{align*}
        & w_{\Sym^2 \BP^1, \pi_P^*O(1), v}(\Sym^2 \BP^1(F_v)) \\
        = & \int_{F_v^2} |F_0(1, x_1, x_2)|^{-3/k} \prod_{\stackrel{a \in \overline{F_v}}{a^2 - x_1 a + x_2 = 0}} \max\{1, |\pi_P(a)|_v\}^{-3 / k} |dx_1dx_2|.
        \end{align*}
        By Gauss lemma \cite[Lemma 1.6.3]{BombieriGruber}, if $v$ is non-archimedean, this is the same as
        \begin{align*}
        & w_{\Sym^2 \BP^1, \pi_P^*O(1), v}(\Sym^2 \BP^1(F_v)) \\
        = & \int_{F_v^2} \max\{|F_0(1, x_1, x_2)|, |F_1(1, x_1, x_2)|, |F_2(1, x_1, x_2)|\}^{-3/k} |dx_1dx_2|.
        \end{align*}
        \item For each place $v$, $\lambda_v$ is a regularization factor to make sure the infinite product converges,
        $$\lambda_v = \begin{cases}
        \frac{1 - N_{F/\BQ}(\frak{p})^{-3}}{1 - N_{F/\BQ}(\frak{p})^{-1}} & \text{ if $v < \infty$ corresponds to prime ideal $\frak{p}$,} \\
        1 & \text{if $v \in M_{F, \infty}$.}
        \end{cases}$$
    \end{itemize}
    \item If $P \in \Gamma_1$, then $X_1(P)$ is an elliptic curve over $F$.
    
    We define $a_{F, 2}(P) = \frac{4d}{k}$; $b_{F, 2}(P) = 0$. We will show the existence of the constant $c_{F, 2} > 0$, but it does not have a nice form; see the remark after Proposition \ref{Theorem_Main_Degree2_CountingResult_Sym2_EC} for its value in a special case, where $c_{F, 2} = \frac{2}{|Aut(\pi_P)|}C_{E, f, F}$ with the constant $C_{E, f, F}$ defined in the remark.
    
    \item If $P \in \Gamma_2$, then $X_1(P)$ is a genus 2 curve over $F$.

    We define $a_{F, 2}(P) = \frac{4d}{k}$; $b_{F, 2}(P) = 0$. For $c_{F, 2}$, suppose $X_1(P)$ has a model $y^2 = h(x)$ for some polynomial $h(x) \in F[x]$, and $\pi_P$ has a presentation
    $$\pi_P(x) = \frac{g_0(x)}{g_1(x)}$$
    for relatively prime $g_0(x), g_1(x) \in F[x]$ (see examples in Appendix \ref{Appendix_Explicit_Presentation_piG},Table \ref{TableModelOfForgetfulMapGenus0}). In particular, $\pi_P$ is invariant under the hyperelliptic involution. Then we define
    $$c_{F, 2}(P) = \frac{1}{|Aut(\pi_P)|} \frac{\res_{s=1} \zeta_F(s)}{\zeta_F(2)} \prod_{v \in M_F} \frac{w_{\BP^1, \pi_P^*\mathbf{O(1)}, v}(\BP^1(F_v))}{\lambda_v},$$
    where
    \begin{itemize}
        \item $M_F$ is the set of places of $F$.
        \item $\zeta_F(s)$ is the Dedekind zeta function of $F$.
        \item For each place $v$, $w_{\BP^1, \pi_P^*\mathbf{O(1)}, v}$ is a local Tamagawa measure of $\BP^1(F_v)$, and
        $$w_{\BP^1, \pi_P^*\mathbf{O(1)}, v}(\BP^1(F_v)) = \int_{F_v} \max\{|g_0(u)|_v, |g_1(u)|_v\}^{-4 / k} du.$$
        \item For each place $v$, $\lambda_v$ is a regularization factor to make sure the infinite product converges,
        $$\lambda_v = \begin{cases}
        1 + \frac{1}{N_{F/\BQ}(\frak{p})} & \text{ if $v < \infty$ corresponds to prime ideal $\frak{p}$,} \\
        1 & \text{if $v \in M_{F, \infty}$.}
        \end{cases}$$
    \end{itemize}
\end{itemize}
\end{defn}
With $a_{F, 2}, b_{F, 2}, c_{F, 2}$ defined, we now restate Theorem \ref{Theorem_Main_Degree2} with error terms.

\begin{thm}
\label{Theorem_Main_Degree2_Details}
Let $F/\BQ$ be a number field of degree $d$. For each $P \in \SP_2$, define 
$$S_{F,2}(P, B) := \#\{c \in \overline{F}: \preper(f_c, K) \cong P \text{ for some $K/F$ of degree $2$ containing $c$}, H(c) \leq B\}.$$

Assume the Morton-Silverman conjecture for quadratic extensions of $F$. Then for each $P \in \SP_2$, we have the asymptotic formula
$$S_{F,2}(P, B) = c_{F, 2}(P) B^{a_{F, 2}(P)} (\log B)^{b_{F, 2}(P)} + O_{P, F}(Error_{P, F, 2}(B))$$
as $B \to \infty$. Here
$$
Error_{P, F, 2}(B) = \begin{cases}
B^{2(3d-1)/k} \log B & \text{when $P \in \Gamma_0$,} \\
B^{2(2d-1)/k} (\log B)^{r/2 + 1} & \text{when $P \in \Gamma_1$,} \\
B^{2(2d-1)/k} \log B & \text{when $P \in \Gamma_2$.} \\
\end{cases}
$$
Here when $P \in \Gamma_1$, $X_1(P)$ is an elliptic curve over $F$, and $r = r(X_1(P)/F)$ is the rank of the Mordell-Weil group $X_1(P)(F)$.
\end{thm}

\subsection{Counting rational points on $\Sym^2 \BP^1$}
\begin{prop}
\label{Theorem_Main_Degree2_CountingResult_Sym2_P1}
Let $F/\BQ$ be a number field of degree $d$. Consider a morphism 
$$f: \BP^1 \to \BP^1$$ 
over $F$ of degree $k$. This induces a map
$$\Sym^2 f: \Sym^2 \BP^1 \to \Sym^2 \BP^1,$$
also of degree $k$. Consider the metrized line bundle $\CL := (\Sym^2 f)^* \mathbf{O(1, 1)}$ on the source $V = \Sym^2 \BP^1$.

Let $H_{\CL}: \Sym^2 \BP^1(F) \to [1, \infty)$ be the induced (absolute, multiplicative) Weil height. Then
$$\#\{x \in \Sym^2 \BP^1(F): H_{\CL}(x) \leq B\} = \frac{1}{3}\tau_{\Sym^2 \BP^1} (\CL) B^{3d/k} + O_{f, F} (B^{(3d-1)/k}),$$
as $B \to \infty$. Here $\tau_{\Sym^2 \BP^1}(\CL)$ is the volume of $\Sym^2 \BP^1$ with respect to a global Tamagawa measure, defined in Section \ref{Subsection_Batyrev_Manin_c_V} and calculated in Example \ref{Example_Sym2P1_pullback_Local_Tamagawa_Measure},  \ref{Example_Sym2P1_pullback_Global_Tamagawa_Measure} and \ref{Example_Sym2P1_pullback_Peyre_Constant}.
\end{prop}
\begin{proof}
This is special case of Theorem \ref{Theorem_GeneralCountingResult_Pn} when $n = 2$, for the product metric on $\BP^2$. The savings of $\log B$ in error term comes from $\Sym^2 \BP^1 \cong \BP^2$ having dimension $> 1$.
\end{proof}

\subsection{Counting rational points on symmetric square of elliptic curves}
\label{Section_Sym2_EC}
\begin{prop}[Counting rational points on symmetric square of elliptic curve]
\label{Theorem_Main_Degree2_CountingResult_Sym2_EC}
Let $F/\BQ$ be a number field of degree $d$, and consider $\mathbf{O(1, 1)}$ on $\Sym^2 \BP^1$ over $F$ with the standard metric. 

Let $E$ be an elliptic curve over $F$. Let $f: E \to \BP^1$ be a morphism over $F$ of degree $k$ such that $f(x) = f(-x)$; in particular, $k$ is even. This induces
$$\Sym^2 f: \Sym^2 E \to \Sym^2 \BP^1,$$
also of degree $k$. 

Consider the metrized line bundle $\CL = (\Sym^2 f)^*\mathbf{O(1, 1)}$ on the source $\Sym^2 E$, and let $H_{\CL}: (\Sym^2 E)(F) \to [1, \infty)$ be the induced (absolute, multiplicative) height. Then
$$\#\{x \in (\Sym^2 E)(F): H_{\CL}(x) \leq B\} = C_{f, E, F} B^{2d/k} + O_{f, E, F}(B^{(2d-1)/k} (\log B)^{r/2 + 1}),$$
as $B \to \infty$, for some constant $C_{f, E, F} > 0$. Here $r = r(E/F)$ is the rank of the Mordell-Weil group $E(F)$.
\end{prop}

\begin{remark}
This calculation was done in Arakelov without details \cite[p. 408, Constant T]{Arakelov}, where the celebrated Faltings-Riemann-Roch theorem was also conjectured (and later proved in Faltings \cite{FaltingsRR}). 

As pointed out in Faltings \cite[p. 406, Remark]{FaltingsRR}, Arakelov's normalization for volume does not lead to a Riemann-Roch theorem, so modifications are probably needed for Arakelov's result to hold. See also \cite[Section 4.3]{BatyrevTschinkel}, where Bost pointed out that Arakelov's calculations may have some $O(1)$ error. 

Since we are unable to find another reference for this calculation, we would sketch it here, and carry it out in details in Appendix \ref{Appendix_Proof_Sym2_EC}. Comparing to Faltings' calculation \cite[Section 8]{FaltingsRR}, we compute the leading constant of main term explicitly (which was not calculated there), and we prove an explicit zero-free region of width $\frac{1}{k} - \epsilon$ for the associated height zeta function.
\end{remark}

\begin{remark}
\label{Remark_Sym2_E_Constant}
The constant $C_{f, E, F}$ does not have a nice form. To illustrate how the constant may look, we compute $C_{f, E, F}$ in Appendix \ref{Appendix_Proof_Sym2_EC} with the following assumptions:
\begin{itemize}
    \item $E$ is semi-stable over $F$,
    \item If $X / \Spec(O_F)$ is the minimal regular model of $E/F$, and $v$ is a non-archimedean place, we assume that the vertical fiber $X_v := X \times_{\Spec(O_F)} \Spec (k(v))$ is irreducible.
\end{itemize}
In this case, the constant $C_{f, E, F}$ takes the form
\begin{align*} & C_{f, E, F} \\
= & 2^{r_1 + r_2 - 1} \cdot \frac{|Cl(F)| Reg_F}{|\mu_F| \zeta_F(2)} \cdot \frac{1}{|\disc(O_F)|} \cdot \frac{\exp\left(d \cdot h_{Faltings}(E)\right)}{|N_{F/\BQ} (\Delta_{E/F})|^{1/4} H(f(O))^{2d/k}} \\
& \sum_{R \in E(F)} \left(\frac{H_{NT}(R)}{H(f(R))^{2/k}}\right)^d  \\
& \,\,\,\,\, \left(\prod_{v \in M_{F, \infty}} \vol_{Fal}\left(g_v \in H^0(X, O_X(D_R + D_O)) \otimes_{O_F} F_v: \int_{X(\overline{F_v})} \log |g_v|_v d\mu_v\leq 0\right)\right).
\end{align*}
Here,
\begin{itemize}
    \item $r_1, r_2$ are the number of real embeddings/pairs of complex embeddings of $F$.
    \item $|Cl(F)|$, $Reg_F$, $|\mu_F|$, $disc(O_F)$, $\zeta_F(s)$ are the class number, regulator, number of roots of unity, discriminant and Dedekind zeta function of $F$.
    \item $h_{Faltings}(E)$ is the (logarithmic) Faltings height of $E$ (defined at Remark \ref{Remark_Definition_Faltings_height}).
    \item $\Delta_{E/F}$ is the minimal discriminant of $E/F$. 
    \item $E(F)$ are the $F$-rational points of the elliptic curve $E$. $O$ is the identity of the elliptic curve $E/F$. 
    \item $X \to Spec(O_F)$ is the minimal regular model of $E/F$. 
    \item $D_R$ is the closure of the $F$-rational point $R$ in $X$. It is an irreducible horizontal Weil divisor. 
    \item $O_X(D_R + D_O)$ is the line bundle on $X$ corresponding to the horizontal Weil divisor $D_R + D_O$.
    \item $H_{NT}$ is the (multiplicative) N\'{e}ron-Tate height on $E$.
    \item For each archimedean place $v$, let
    $$d\mu_v := \frac{1}{k} c_1(f^*\mathbf{O(1)})$$ 
    be the first Chern form of the metrized line bundle $f^*\mathbf{O(1)}$, normalized by $1/k$. The $1/k$ factor ensures that 
    $$\int_{X(\overline{F_v})} d\mu_v = 1.$$
    \item The space $H^0(X, O_X(D_R + D_O)) \otimes_{O_F} F_v$ has a Haar measure, called the Faltings volume \cite{FaltingsRR}, and is denoted by $\vol_{Fal}$ (see Definition \ref{Definition_Faltings_Volume}).
\end{itemize}
It is worth remarking that
$$2^{r_1 + r_2 - 1} \cdot \frac{|Cl(F)| Reg_F}{|\mu_F| \zeta_F(2)} \cdot \frac{1}{|\disc(O_F)|}$$
is basically Schanuel's constant for $n = 2$ (up to a constant factor), and is closely related to $$\frac{res_{s=1}\zeta_F(s)}{\zeta_F(2)}$$
via class number formula. 

\end{remark}
\begin{proof}
We will sketch the proof here, while leaving the details in Appendix \ref{Appendix_Proof_Sym2_EC}.

The addition map $\Sym^2 E \to E$ gives $\Sym^2 E$ a $\BP^1$-bundle structure over $E$. We can then count rational points on $\Sym^2 E$ fiber by fiber. We summarize the key points in this argument:
\begin{itemize}
    \item If $x \in (\Sym^2 E)(F)$ lies in the fiber over $R \in E(F)$ such that $H_{\CL}(x) \leq B$, we will show that 
    $$H(R) \ll_{E, f, F} B^2.$$ 
    By Theorem \ref{Theorem_GeneralCountingResult_AV}, there are $\asymp_{E, f, F} (\log B)^{r/2}$ such $R \in E(F)$, so we only need to handle $\asymp_{E, f, F} (\log B)^{r/2}$ fibers.
    \item Let $R \in E(F)$. We will count the number of rational points $x$ in $\BP^1$-fiber over $R$ satisfying $H_{\CL}(x) \leq B$. The height function on the $\BP^1$-fiber comes from the pullback of $\mathbf{O(1)}$ via the map
    $$\BP^1 \into \Sym^2 E \stackrel{\Sym^2 f}{\longrightarrow} \Sym^2 \BP^1,$$
    and is a metrized line bundle of degree $k$. By Proposition \ref{Theorem_Main_Degree1_CountingResult_P1}, we hence expect the number of rational points on each $\BP^1$-fiber to have main term $\asymp B^{2d/k}$, and error term $O\left(B^{(2d-1)/k} \log B\right)$. However, a priori the big $O$ constant of the error term depends on the fiber; equivalently it depends on $R \in E(F)$.

    We will show that:
    \begin{itemize}
        \item The sum over $R \in E(F)$ of the leading constants in the main term converges.
        \item The big-$O$ constant in the error term can be taken uniformly over $R \in E(F)$.

        (As in Masser-Vaaler \cite{MasserVaaler}, power savings in error term comes from the Lipschitz parametrizability of certain regions, which vary with $R \in E(F)$. We will show that the number of Lipschitz parametrizations/Lipschitz constants of these regions are uniform over $R \in E(F)$ (Proposition \ref{Proposition_Uniformity_R}), from which the uniformity of big-O constant follows.)
    \end{itemize}
    Hence after summing over all fibers, the main term would still have size $\asymp B^{2d/k}$, and the error term would have size $O\left(B^{(2d-1)/k} (\log B)^{r/2 + 1}\right)$.
\end{itemize}
We would need to work explicitly with the $\BP^1$-bundle structure of $\Sym^2 E$ and height on $\Sym^2 E$ restricted to fibers over $E$. This naturally brings us to Arakelov geometry, and is why we ran the argument in the Arakelov setup.
\end{proof}

\subsection{Counting rational points on symmetric square of genus 2 curves}
\begin{prop}
\label{Theorem_Main_Degree2_CountingResult_Sym2_Genus2}
Let $F/\BQ$ be a number field of degree $d$, and consider $\mathbf{O(1, 1)}$ on $\Sym^2 \BP^1$ over $F$ with the standard metric. 

Let $C$ be a genus 2 curve over $F$ with a $F$-rational point $Q_0$, and let $-: C \to C$ be the hyperelliptic involution. Let $x: C \to \BP^1$ be the unique even degree 2 map such that
$$x(Q_0) = x(-Q_0) = \infty.$$
Consider an even morphism $f: C \to \BP^1$ over $F$ of degree $k$, factorized as $f = G \circ x$ for some $G: \BP^1 \to \BP^1$ of degree $\frac{k}{2}$. This induces 
$$\Sym^2 f: \Sym^2 C \to \Sym^2 \BP^1,$$ 
also of degree $k$. Consider the metrized line bundle 
$$\CL := (\Sym^2 f)^* \mathbf{O(1, 1)}$$ 
on the source $\Sym^2 C$. Let $H_{\CL}: (\Sym^2 C)(F) \to [1, \infty)$ be the induced (absolute, multiplicative) Weil height. Then
$$\#\{x \in (\Sym^2 C)(F): H_{\CL}(x) \leq B\} = \frac{1}{2} \tau_{\BP^1} (G^*\mathbf{O(1)}) B^{2d / k} + O_{\CL, F}\left(B^{(2d-1) / k} \log B \right),
$$
as $B \to \infty$. Here $\tau_{\BP^1} (G^*\mathbf{O(1)})$ is the volume of $\BP^1$ with respect to a global Tamagawa measure, defined in Section \ref{Subsection_Batyrev_Manin_c_V} and calculated in Example \ref{Example_Pn_pullback_Local_Tamagawa_Measure},  \ref{Example_Pn_pullback_Global_Tamagawa_Measure} and \ref{Example_Pn_pullback_Peyre_Constant}.
\end{prop}
\begin{proof}
Consider the birational map
$$p_{Q_0}: \Sym^2 C \to Jac(C)$$
over $F$, defined by sending $\{Q, Q'\} \to cl\left([Q] + [Q'] - 2[Q_0]\right)$. This is the blow up of $Jac(C)$ at one point, with exceptional divisor of the form
$$\nabla = \{\{Q, -Q\}: Q \in C(\overline{F})\}.$$

Consider the theta divisor on $Jac(C),$
$$\Theta_{Q_0} := \{[Q] - [Q_0]: Q \in C(\overline{F})\},$$
with respect to $Q_0$. Similarly, consider $p_{-Q_0}: \Sym^2 C \to Jac(C)$ and the theta divisor $\Theta_{-Q_0}$.

\begin{lemma}
In $\Pic(\Sym^2 C)$,
$$[\CL] = \frac{k}{2} [p_{Q_0}^* \Theta_{Q_0}] + \frac{k}{2} [p_{-Q_0}^* \Theta_{-Q_0}] - k\nabla.$$
\end{lemma}
\begin{proof}[Proof of lemma]
In $\Pic(\Sym^2 \BP^1)$, $O(1, 1)$ corresponds to the hyperplane class
$$[D] = \{\{R, \infty\}: R \in \BP^1(\overline{F})\}.$$
Pulling back, we see that in $\Pic(Sym^2 C)$,
\begin{align*}
[\CL] = & \sum_{\stackrel{P}{f(P) = \infty}} \{\{Q, P\}: Q \in C(\overline{F})\} \\ 
= & \sum_{\stackrel{P / \{\pm 1\} }{f(P) = \infty}} \left(\{\{Q, P\}: Q \in C(\overline{F})\} + \{\{Q, -P\}: Q \in C(\overline{F})\}\right)
\end{align*}

For any $P \in C(\overline{F})$, consider the morphism $prod \circ \Sym^2 (x - x(P)): \Sym^2 C \to \BP^1$, defined by
$$prod \circ \Sym^2 (x - x(P))(\{Q, Q'\}) = (x(Q) - x(P))(x(Q') - x(P)).$$
Note that the associated principal divisor is
\begin{align*}
& \div\left(prod \circ \Sym^2 (x - x(P))\right) \\
= &\, \{\{Q, P\}: Q \in C(\overline{F})\} + \{\{Q, -P\}: Q \in C(\overline{F})\} \\
&\, - \{\{Q, Q_0\}: Q \in C(\overline{F})\} - \{\{Q, -Q_0\}: Q \in C(\overline{F})\}.
\end{align*}
Hence in $\Pic(\Sym^2 C)$,
$$[\CL] = \frac{k}{2} \left(\{\{Q, Q_0\}: Q \in C(\overline{F})\} + \{\{Q, -Q_0\}: Q \in C(\overline{F})\}\right)$$
By \cite[Proposition V.3.6]{Hartshorne},
$$p_{Q_0}^* \Theta_{Q_0} = \{\{Q, Q_0\}: Q \in C(\overline{F})\} + \nabla$$
and
$$p_{-Q_0}^* \Theta_{-Q_0} = \{\{Q, -Q_0\}: Q \in C(\overline{F})\} + \nabla$$
Hence
$$[\CL] = \frac{k}{2} [p_{Q_0}^* \Theta_{Q_0}] + \frac{k}{2} [p_{-Q_0}^* \Theta_{-Q_0}] - k\nabla$$
as desired.
\end{proof}

We now count $\Sym^2 C(F)$ in two cases: those on the exceptional divisor $\nabla$, and those not on $\nabla$.

\textit{Case 1: Counting $\nabla(F)$.}

Note that $\nabla$ is a line; it is parametrized by 
$$x^{-1}: \BP^1 \to \nabla \subset \Sym^2 C$$
defined over $F$. Moreover, if $R \in \BP^1(F)$ and $x(Q) = x(-Q) = R$, we see that 
\begin{align*}
H_{\CL} (x^{-1}(R)) & = H(f(Q)) H(f(-Q))\\
& = H(f(Q))^2 \\
& = H(G(R))^2
\end{align*}
Hence by Proposition \ref{Theorem_Main_Degree1_CountingResult_P1},
\begin{align*}
& \, \#\{R \in \BP^1(F): H_{\CL}(x^{-1}(R)) \leq B\} \\
= & \, \#\{R \in \BP^1(F): H(G(R)) \leq B^{1/2}\} \\
= & \, \frac{1}{2} \tau_{\BP^1} (G^* \mathbf{O(1)}) B^{2d / k} + O_{\CL, F}\left(B^{(2d-1) / k} \log B\right),
\end{align*}
since $\deg G = \frac{k}{2}$.

\textit{Case 2: Counting $(\Sym^2 C - \nabla)(F)$.}

Note that $\Sym^2 C - \nabla$ is isomorphic to $Jac(C) - (point)$, and that the divisor $\CL |_{\Sym^2 C - \nabla}$ comes from pullback of a degree $k$ effective divisor on $Jac(C)$. Hence by Theorem \ref{Theorem_GeneralCountingResult_AV} applied to $Jac(C)$, we see that 
$$\#\{x \in (\Sym^2 C - \nabla)(F): H_{\CL}(x) \leq B\} = O_{\CL, F}((\log B)^{r/2}),$$
where $r$ is the rank of Mordell-Weil group $Jac(C)(F)$. In particular, this is smaller than the error term in Case 1.

Finally, putting both cases together we see that
$$\#\{x \in (\Sym^2 C)(F): H_{\CL}(x) \leq B\} = \frac{1}{2} \tau_{\BP^1} (G^*\mathbf{O(1)}) B^{2d / k} + O_{\CL, F}\left(B^{(2d-1) / k} \log B\right),
$$
as desired.
\end{proof}

\subsection{Asymptotics formula of $N_{F, 2}(P, B)$}
Putting the last three propositions together, we can now compute the asymptotics for $N_{F, 2}(P, B)$.

\begin{thm}
\label{Theorem_Main_Degree2_CountingResult_Summary}
Let $P$ be a portrait. Let $X_1(P)$ be the dynamical modular curve associated to $P$, and let $\pi_P: X_1(P) \to \BP^1$ be the natural forgetful map of degree $k$.

Let $H: \BP^1(\BQbar) \to [1, \infty)$ be the absolute, multiplicative Weil height. For a fixed number field $F/\BQ$ of degree $d$, define
$$N_{F, 2}(P, B) := \#\{x \in X_1(P)(\overline{F}): [F(x): F] \leq 2, H(\pi_P(x)) \leq B\},$$
\begin{enumerate}[label=(\alph*)]
    \item Suppose $P \in \Gamma_0$, then $X_1(P)$ is rational over $F$. 
    We have
    $$N_{F, 2}(P, B) = |Aut(\pi_P)| \cdot c_{F, 2}(P) B^{6d/k} + O_{P, F}(B^{2(3d-1)/k} \log B)$$
    as $B \to \infty$. 
    \item Suppose $P \in \Gamma_1$, then $X_1(P)$ is an elliptic curve over $F$. We have
    $$N_{F, 2}(P, B) = |Aut(\pi_P)| \cdot c_{F, 2}(P) B^{4d/k} + O_{P, F}(B^{2(2d-1)/k}(\log B)^{r/2 + 1})$$
    as $B \to \infty$. 
    \item Suppose $P \in \Gamma_2$, then $X_1(P)$ is a genus 2 curve over $F$. We have
    $$N_{F, 2}(P, B) = |Aut(\pi_P)| \cdot c_{F, 2}(P) B^{4d/k} + O_{P, F}(B^{2(2d-1)/k}\log B)$$
    as $B \to \infty$. 
    \item Suppose $P \not\in \Gamma_0 \cup \Gamma_1 \cup \Gamma_2$, then $X_1(P)$ is a curve of genus $\ge 3$, and 
    $$N_{F, 2}(P, B) = O_{P, F}(1).$$
\end{enumerate}
\end{thm}
\begin{proof}
$N_{F, 2}(P, B)$ is the cardinality of the set
$$S := \{x \in X_1(P)(\overline{F}): [F(x): F] \leq 2, H(\pi_P(x)) \leq B\}$$
Note that $S$ can be partitioned into $S_1 \cup S_2$, where $S_1$ is the degree 1 points,
$$S_1 := \{x \in X_1(P)(F): H(\pi_P(x)) \leq B\},$$
and $S_2$ is the degree 2 points,
$$S_2 := \{x \in X_1(P)(\overline{F}): [F(x): F] = 2, H(\pi_P(x)) \leq B\}.$$

Also consider
$$T := \{\{Q_1, Q_2\} \in \Sym^2 X_1(P)(F): H(\pi_P(Q_1))H(\pi_P(Q_2)) \leq B^2\}.$$
which can be partitioned into $T_1 \cup T_2$. Here $T_1$ is
$$
T_1 := \{\{Q_1, Q_2\} \in \Sym^2 X_1(P)(F): [F(Q_1, Q_2): F] = 1, H(\pi_P(Q_1))H(\pi_P(Q_2)) \leq B^2\}.
$$
Since $[F(Q_1, Q_2):F] = 1$ forces $Q_1, Q_2 \in X_1(P)(F)$, $T_1$ tracks unordered pairs of points in $X_1(P)(F)^2$ with bounded height,
$$T_1 = \{\{Q_1, Q_2\} \in X_1(P)(F)^2 / (\BZ/2\BZ): H(\pi_P(Q_1))H(\pi_P(Q_2)) \leq B^2\}.
$$
On the other hand,
$$T_2 := \{\{Q_1, Q_2\} \in \Sym^2 X_1(P)(F): [F(Q_1, Q_2): F] = 2, H(\pi_P(Q_1))H(\pi_P(Q_2)) \leq B^2\}.$$

We have the obvious lemma.
\begin{lemma}
Let $\overline{x} \in X_1(P)(\overline{F})$ be the Galois conjugate of $x$, when $[F(x): F] = 2$. 
Then the map $x \to \{x, \overline{x}\}$ is a 2-to-1 map from $S_2$ to $T_2$.
\end{lemma}
Hence $|S_2| = 2|T_2| = 2(|T| - |T_1|)$, and
$$N_{F, 2}(P, B) = |S_1| + |S_2| = |S_1| + 2(|T| - |T_1|).$$
We now use the last three propositions, and previous results on $N_{F, 1}(P, B)$ to calculate $N_{F, 2}(P, B)$. The main term would come from $|T|$, while the $|S_1|, |T_1|$ would contribute to error terms only.

\textit{Case (a): $P \in \Gamma_0$}.
\begin{itemize}
    \item By Theorem \ref{Theorem_Main_Degree1_CountingResult_Summary}, $|S_1| = O_{P, F}(B^{2d/k})$.
    \item By Proposition \ref{Theorem_Main_Degree2_CountingResult_Sym2_P1}, 
    $$|T| = \frac{1}{3} \tau_{\Sym^2 \BP^1} (\CL) B^{6d/k} + O_{P, F} (B^{2(3d-1)/k}).$$
    The constant $\tau_{\Sym^2 \BP^1}(\CL)$ was calculated in Example \ref{Example_Sym2P1_pullback_Local_Tamagawa_Measure}, \ref{Example_Sym2P1_pullback_Global_Tamagawa_Measure} and \ref{Example_Sym2P1_pullback_Peyre_Constant}. We defined $c_{F, 2}(P)$ such that
    $$2 \cdot \frac{1}{3} \tau_{\Sym^2 \BP^1}(\CL) = |Aut(\pi_P)| \cdot c_{F, 2}(P),$$
    hence
    $$2|T| = |Aut(\pi_P)| \cdot c_{F, 2}(P) B^{6d/k} + O_{P, F} (B^{2(3d-1)/k}).$$
    
    \item For $|T_1|$, focus on $H(\pi_P(Q_1))$, and partition $[0, B^2]$ into $\asymp \log B$ dyadic intervals. By Theorem \ref{Theorem_Main_Degree1_CountingResult_Summary}, for each dyadic interval $[U, 2U]$ we have
    \begin{align*}
    & \, \#\left\{\{Q_1, Q_2\} \in X_1(P)(F)^2: H(\pi_P(Q_1)) \asymp U, H(\pi_P(Q_2)) \ll \frac{B^2}{U}\right\} \\
    = & \, O_{P, F}\left(U^{2d/k} \cdot \left(\frac{B^2}{U}\right)^{2d/k}\right) \\
    = & \, O_{P, F}(B^{4d/k}).
    \end{align*}
    Summing over all the intervals and since $d \ge 1$, we see that $$|T_1| = O_{P, F}(B^{4d/k} \log B) = O_{P, F}(B^{2(3d-1)/k} \log B).$$
    \item Hence 
    $$N_{F, 2}(P, B) = |Aut(\pi_P)| \cdot c_{F, 2}(P) B^{6d/k} + O_{P, F} (B^{2(3d-1)/k} \log B).$$
\end{itemize}
\textit{Case (b): $P \in \Gamma_1$.} Let $r$ be the rank of Mordell-Weil group $X_1(P)(F)$.
\begin{itemize}
    \item By Theorem \ref{Theorem_Main_Degree1_CountingResult_Summary}, $|S_1| = O_{P, F}\left((\log B)^{r/2}\right)$.
    \item By Proposition \ref{Theorem_Main_Degree2_CountingResult_Sym2_EC},
    $$|T| = C_{\pi_P, E, F} B^{4d/k} + O_{P, F}(B^{2(2d-1)/k} (\log B)^{r/2 + 1}).$$
    We defined $c_{F, 2}(P)$ such that
    $$2 \cdot C_{\pi_P, E, F} = |Aut(\pi_P)| \cdot c_{F, 2}(P),$$
    hence 
    $$2|T| = |Aut(\pi_P)| \cdot c_{F, 2}(P) B^{4d/k} + O_{P, F}(B^{2(2d-1)/k} (\log B)^{r/2 + 1}).$$
    \item For $|T_1|$, note that 
    \begin{align*}
    & \, \#\left\{\{Q_1, Q_2\} \in X_1(P)(F)^2: H(\pi_P(Q_1)) H(\pi_P(Q_2)) \leq B^2\right\} \\
    \leq & \, \#\left\{Q_1 \in X_1(P)(F): H(\pi_P(Q_1)) \leq B^2\right\}^2.
    \end{align*}
    By Theorem \ref{Theorem_Main_Degree1_CountingResult_Summary},
    $$\#\left\{Q_1 \in X_1(P)(F): H(\pi_P(Q_1)) \leq B^2\right\} = O_{P, F}\left((\log B)^{r/2}\right).$$
    Therefore $|T_1| = O_{P, F}\left((\log B)^{r}\right)$.
    \item Hence
    $$N_{F, 2}(P, B) = |Aut(\pi_P)| \cdot c_{F, 2}(P) B^{4d/k} + O_{P, F} (B^{2(2d-1)/k} (\log B)^{r/2 + 1}).$$
\end{itemize}
\textit{Case (c): $P \in \Gamma_2$.}
\begin{itemize}
    \item By Faltings' theorem, $|S_1| = O_{P, F}(1)$.
    \item By Proposition \ref{Theorem_Main_Degree2_CountingResult_Sym2_Genus2},
    $$|T| = \frac{1}{2}\tau_{\BP^1}{G^*\mathbf{O(1)}} B^{4d/k} + O_{P, F}(B^{2(2d-1)/k} \log B).$$
    The constant $\frac{1}{2}\tau_{\BP^1}{G^*\mathbf{O(1)}}$ was calculated in Example \ref{Example_Pn_pullback_Local_Tamagawa_Measure}, \ref{Example_Pn_pullback_Global_Tamagawa_Measure} and \ref{Example_Pn_pullback_Peyre_Constant}. We defined $c_{F, 2}(P)$ such that
    $$2 \cdot \frac{1}{2}\tau_{\BP^1}{G^*\mathbf{O(1)}} = |Aut(\pi_P)| \cdot c_{F, 2}(P),$$
    hence
    $$2|T| = |Aut(\pi_P)| \cdot c_{F, 2}(P) B^{4d/k} + O_{P, F}(B^{2(2d-1)/k} \log B).$$
    \item By Faltings' theorem, $|T_1| = O_{P, F}(1)$.
    \item Hence
    $$N_{F, 2}(P, B) = |Aut(\pi_P)| \cdot c_{F, 2}(P) B^{4d/k} + O_{P, F} (B^{2(2d-1)/k} \log B).$$
\end{itemize}
\textit{Case (d): $P \notin \Gamma_0 \cup \Gamma_1 \cup \Gamma_2$.}
\begin{itemize}
    \item By Faltings' theorem, $|S_1| = O_{P, F}(1)$.
    \item By Theorem \ref{Theorem_Base_Degree2} and the remark that follows, $|T| = O_{P, F}(1)$.
    \item By Faltings' theorem, $|T_1| = O_{P, F}(1)$.
    \item Hence
    $$N_{F, 2}(P, B) = O_{P, F}(1).$$
\end{itemize}
\end{proof}

\subsection{Proof of Theorem \ref{Theorem_Main_Degree2}/Theorem \ref{Theorem_Main_Degree2_Details}}

\begin{lemma}
\label{Lemma_Main_Degree2_Fiber_Count}
With the notations in Theorem \ref{Theorem_Main_Degree2_CountingResult_Summary}, we have
\begin{align*}
& \#\{c \in \overline{F}: \preper(f_c, K) \supset P \text{ for some $K/\BQ$ of degree $2$ containing $c$}, [F(c):F] \leq 2, H(c) \leq B\} \\
= & \frac{1}{|Aut(\pi_P)|} N_{F, 2}(P, B) + Error_{P, F, 2}(B).
\end{align*}
Here 
$$
Error_{P, F, 2}(B) = \begin{cases}
B^{2(3d-1)/k} \log B & \text{when $P \in \Gamma_0$,} \\
B^{2(2d-1)/k} (\log B)^{r/2 + 1} & \text{when $P \in \Gamma_1$,} \\
B^{2(2d-1)/k} \log B & \text{when $P \in \Gamma_2$.} \\
\end{cases}
$$
\end{lemma}
\begin{proof}
The forgetful map $\pi_P: X_1(P) \to \BP^1$ induces
$$\Sym^2 \pi_P: \Sym^2 X_1(P) \to \Sym^2 \BP^1.$$
Note that the following are equivalent:
\begin{itemize}
    \item $c \in \overline{F}$ satisfies $[F(c): F] \leq 2$ and $\preper(f_c, K) \supset P$ for some $K/F$ of degree 2,
    \item $c = \pi_P(x)$ for some $x \in X_1(P) (\overline{F})$ such that $[F(x) : F] \leq 2$, and $c \neq \infty$.
\end{itemize}
Hence
\begin{align*}
& \#\{c \in \overline{F}: \preper(f_c, K) \supset P \text{ for some $K/\BQ$ of degree $2$ containing $c$}, [F(c):F] \leq 2, H(c) \leq B\} \\
= & \#\{c \in \overline{F}: c = \pi_P(x) \text{ for some $x \in X_1(P) (\overline{F})$, with } [F(x) : F] \leq 2\} + O_P(1).
\end{align*}

Suppose $c = \pi_P(x)$ for some $x \in X_1(P)(\overline{F})$ with $[F(x):F] \leq 2$. We want to show that for generic such $c$,
$$\pi_P^{-1}(c) \cap \{z \in X_1(P)(\overline{F}): [F(z):F] \leq 2 \} = \{\sigma(x): \sigma \in Aut(\pi_P)\},$$
which would prove the lemma. By the definition of $Aut(\pi_P)$, we see that right hand side is a subset of the left. 

First, note that right hand side has size $|Aut(\pi_P)|$ with $O_P(1)$ exceptions of $c$. This is because if right hand side is smaller than expected, it must mean that an unexpected coincidence
$$\sigma(x) = \sigma'(x)$$ 
occured for some $\sigma \neq \sigma' \in Aut(\pi_P)$. Since $\sigma = \sigma'$ is a proper Zariski-closed subset of the curve $X_1(P)$, it is of dimension 0; so there are finitely many exceptional $x$'s, hence finitely many exceptional $c$'s as well. To recap, with $O_P(1)$ exceptions of $c \in \pi_P(X_1(P)(\overline{F}))$, we have
$$\left|\pi_P^{-1}(c) \cap \{z \in X_1(P)(\overline{F}): [F(z):F] \leq 2 \}\right| \ge |Aut(\pi_P)|.$$

Finally, we need to show that 
$$\left|\pi_P^{-1}(c) \cap \{z \in X_1(P)(\overline{F}): [F(z):F] \leq 2 \}\right| > |Aut(\pi_P)|$$
happens rarely. Suppose we have such an exceptional $c$, where there exists $x, z \in X_1(P)(\overline{F})$ not in the same $Aut(\pi_P)$-orbit, with $[F(x): F] \leq 2$ and $[F(z): F] \leq 2$, such that
$$\pi_P(x) = \pi_P(z) = c.$$
\begin{itemize}
    \item If $[F(c):F] = 2$, then $F(x) = F(c)$, since $F(c) \subset F(x)$, and
    $$2 = [F(c) : F] \leq [F(x): F] \leq 2.$$
    Similarly, $F(z) = F(c) = F(x)$. Hence $(x, z)$ is a quadratic point on some irreducible component $C$ of
    $$X_1(P) \times_{\BP^1} X_1(P) := \{(x, z) \in X_1(P) \times X_1(P): \pi_P(x) = \pi_P(z)\},$$
    which is not the graph of any $\sigma \in Aut(\pi_P)$; we call such curve exceptional. We find that
    \begin{itemize}
    \item If $P \in \Gamma_0$, then any exceptional $C$ is smooth and has genus $\ge 1$ (Explicit computation in Sage, see Appendix \ref{Appendix_Explicit_Presentation_piG}).
    \item If $P \in \Gamma_1$, then any exceptional $C$ is smooth and has genus $\ge 4$ (Explicit computation in Sage, see Appendix \ref{Appendix_Explicit_Presentation_piG}).
    \item If $P \in \Gamma_2$, then any exceptional $C$ has geometric genus $\ge 3$.

    \textit{Reason:} Consider the projection map $\pi: C \to X_1(P)$ to first coordinate 
    $$C \into X_1(P) \times_{\BP^1} X_1(P) \stackrel{\pi_1}{\longrightarrow} X_1(P),$$
    and extend it to the smooth projective model $\widetilde{\pi}: \widetilde{C} \to X_1(P)$. Note that $\widetilde{\pi}$ has degree $> 1$; otherwise $C$ is the graph of automorphism $\sigma := \pi_2 \circ \widetilde{\pi}^{-1} \in Aut(\pi_P)$, contradicting the assumption that $C$ is exceptional.

    Since $X_1(P)$ has genus 2, by applying Riemann-Hurwitz theorem to $\widetilde{\pi}: \widetilde{C} \to X_1(P)$ we see that $genus(\widetilde{C}) \ge 3$.
\end{itemize}
The number of quadratic points on these exceptional curves are few due to the lower bound on genus: by Theorem \ref{Theorem_Main_Degree2_CountingResult_Summary}, we see that
\begin{align*}
& \, \#\left\{c \in \pi_P(X_1(P)(\overline{F})): 
    \begin{aligned}
    & \left|\pi_P^{-1}(c) \cap \{z \in X_1(P)(\overline{F}): [F(z):F] \leq 2 \}\right| > |Aut(\pi_P)|, \\
    & [F(c): F] = 2, H(c) \leq B
    \end{aligned}
  \right\} \\
\ll_P & \, \#\left\{(x, z) \in (X_1(P) \times_{\BP^1} X_1(P))(\overline{F}): 
    \begin{aligned}
    & (x, z) \text{ lies in an exceptional curve $C$} \\
    & \text{of $X_1(P)$, $[F(x, z): F] = 2$,} \\
    & \pi_P(x) = \pi_P(z), H(\pi_P(x)) \leq B
    \end{aligned}
\right\} \\
\ll_P & \, \sum_{\text{exceptional curves $C$}} \#\left\{ Q \in (\Sym^2 C)(F): H(\Sym^2 \pi_P(Q)) \leq B^2
\right\} \\
= & \, O_P(Error_{P, F, 2}(B))
\end{align*}

\item Now suppose $c \in F$, and suppose there exists $x, z \in X_1(P)(\overline{F})$ not in the same $Aut(\pi_P)$-orbit, with $[F(x): F] \leq 2$ and $[F(z): F] \leq 2$, such that
$$\pi_P(x) = \pi_P(z) = c.$$

If $x \in X_1(P)(F)$, then $(x, z)$ is a quadratic point on $X_1(P) \times_{\BP^1} X_1(P)$. The number of such $c$'s can be handled the same way as in the last case.
\item Hence suppose $c \in F$ and $[F(x): F] = 2$ as well. Let $\overline{x}$ be the Galois conjugate of $x$. Consider the diagonal closed subset $\Delta: \BP^1 \to \Sym^2 \BP^1$. That $c \in F$ implies that $\{c, c\} \in \Delta(F) \subset \Sym^2 \BP^1(F)$, so 
$$\{x, \overline{x}\} \in (\Sym^2 \pi_P)^{-1}(\Delta) (F) \subset \Sym^2 X_1(P)(F).$$
The number of rational points on the curve $(\Sym^2 \pi_P)^{-1}(\Delta)$ is few: the degree of the restriction
$$\Sym^2 \pi_P: (\Sym^2 \pi_P)^{-1}(\Delta) \to \BP^1$$
is still $k$, there are $O_P(1)$ irreducible components in $(\Sym^2 \pi_P)^{-1}(\Delta)$, and each irreducible component $C$ has at most $O\left(B^{2d/k}\right)$ rational points $Q \in C(F)$ with $H(\Sym^2 \pi_P(Q)) \leq B$ (Theorem \ref{Theorem_Main_Degree1_CountingResult_Summary}), which is small compared to $Error_{P, F, 2}(B)$.

In other words,
\begin{align*}
& \, \#\left\{c \in \pi_P(X_1(P)(\overline{F})): 
    \begin{aligned}
    & \left|\pi_P^{-1}(c) \cap \{z \in X_1(P)(\overline{F}): [F(z):F] \leq 2 \}\right| > |Aut(\pi_P)|, \\
    & c \in F, c = \pi_P(x) \text{ for $[F(x): F] = 2$}, H(c) \leq B
    \end{aligned}
  \right\} \\
\le & \, \#\left\{c \in \pi_P(X_1(P)(\overline{F})): 
    c \in F, c = \pi_P(x) \text{ for $[F(x): F] = 2$}, H(c) \leq B
  \right\} \\
\ll & \, \#\left\{Q \in \left(\Sym^2 \pi_P\right)^{-1}(\Delta)(F): 
    H(\Sym^2 \pi_P(Q)) \leq B
\right\} \\
= & \, O_P\left(B^{2d/k}\right) \\
= & \, O_P(Error_{P, F, 2}(B)).
\end{align*}
\end{itemize}
This shows the rarity of $c$ satisfying 
$$\left|\pi_P^{-1}(c) \cap \{z \in X_1(P)(\overline{F}): [F(z):F] \leq 2 \}\right| > |Aut(\pi_P)|,$$
and finish the proof of lemma.
\end{proof}

\begin{proof}[Proof of Theorem \ref{Theorem_Main_Degree2}/Theorem \ref{Theorem_Main_Degree2_Details}]
Recall the definition of $S_{F, 2}(P, B)$: for each $P \in \SP_2$, 
$$S_{F,2}(P, B) := \#\{c \in \overline{F}: \preper(f_c, K) \cong P \text{ for some $K/F$ of degree $2$ containing $c$}, H(c) \leq B\}.$$
Our goal is to get an asymptotic formula for $S_{F, 2}(P, B)$ as $B \to \infty$.

Let $\{G_i\}_{i \in I}$ be the minimal strongly admissible graphs that strictly contains $P$, where $X_1(G_i)$ has a quadratic point over $F$. Assuming Morton-Silverman conjecture for quadratic extensions of $F$, this set of graphs is finite; enumerate them as $G_1, \cdots, G_m$.

By inclusion-exclusion principle, we see that
\begin{align*}
& \, \#\{c \in \overline{F}: \preper(f_c, K) \supset P \text{ for some $K/F$ of degree $2$ containing $c$}, H(c) \leq B\} \\
& \, \,- \sum_{i=1}^m \#\{c \in \overline{F}: \preper(f_c, K) \supset G_i \text{ for some $K/F$ of degree $2$ containing $c$}, H(c) \leq B\} \\
\leq & \, S_{F, 2} (P, B) \\
\leq & \, \#\{c \in \overline{F}: \preper(f_c, K) \supset P \text{ for some $K/F$ of degree $2$ containing $c$}, H(c) \leq B\}.
\end{align*}

The main term would be
\begin{align*}
& \, \#\{c \in \overline{F}: \preper(f_c, K) \supset P \text{ for some $K/F$ of degree $2$ containing $c$}, H(c) \leq B\} \\
= & \, \frac{1}{|Aut(\pi_P)|} N_{F, 2}(P, B) + O_{P, F}(Error_{P, F, 2}(B))) \\
= & \, c_{F, 2}(P) B^{a_{F, 2}(P)} (\log B)^{b_{F, 2}(P)} + O_{P, F}(Error_{P, F, 2}(B))
\end{align*}
by Lemma \ref{Lemma_Main_Degree2_Fiber_Count} and Theorem \ref{Theorem_Main_Degree2_CountingResult_Summary}.

It now suffices to show that 
$$\sum_{i=1}^m \#\{c \in \overline{F}: \preper(f_c, K) \supset G_i \text{ for some $K/F$ of degree $2$ containing $c$}, H(c) \leq B\}$$
only contributes to the error term. 

For each $G_i \supset P$,
\begin{itemize}
    \item If $P \in \Gamma_0$, by Proposition \ref{Proposition_Properties_Of_X1G}(b), $\deg \pi_{G_i} \ge 2 \cdot \deg \pi_P = 2k$, so either $X_1(G_i)$ has genus $\leq 2$ and
    $$a_{F, 1}(G_i) \leq \frac{6d}{\deg \pi_{G_i}} \leq \frac{3d}{k} < \frac{2(3d-1)}{k}, $$
    or $X_1(G_i)$ has genus $\ge 3$ and $a_{F, 1}(G_i) = 0$. In either case, Theorem \ref{Theorem_Main_Degree2_CountingResult_Summary} implies that
    \begin{align*}
    & \, \#\{c \in \overline{F}: \preper(f_c, K) \supset G_i \text{ for some $K/F$ of degree $2$ containing $c$}, H(c) \leq B\} \\
    = & \, O_{P, F} \left(B^{2(3d-1)/k}\right) \\
    = & \, O_{P, F} (Error_{P, F, 2}(B)).
    \end{align*}
    \item If $P \in \Gamma_1 \cup \Gamma_2$, then $X_1(P)$ has genus 1 or 2. Since the 10 such possible $P$'s (Table \ref{TableModelOfForgetfulMapGenus1} and Table \ref{TableModelOfForgetfulMapGenus2}, Appendix \ref{Appendix_Explicit_Presentation_piG}) has no containment relationship, we see that $X_1(G_i)$ has genus $\ge 3$. Theorem \ref{Theorem_Main_Degree2_CountingResult_Summary}(d) implies that 
    \begin{align*}
    & \, \#\{c \in \overline{F}: \preper(f_c, K) \supset G_i \text{ for some $K/F$ of degree $2$ containing $c$}, H(c) \leq B\} \\
    = & \, O_{P, F} (1) \\
    = & \, O_{P, F} (Error_{P, F, 2}(B)).
    \end{align*}
\end{itemize}
The result hence follows.
\end{proof}

\appendix
\newpage

\section{Explicit presentation of $\pi_G: X_1(G) \to \BP^1$}
\label{Appendix_Explicit_Presentation_piG}
For strongly admissible graphs $G$, we document a model of $X_1(G)$, and a presentation of the forgetful map $\pi_G: X_1(G) \to \BP^1$
when genus of $X_1(G)$ is 0, 1 or 2. 

Model and presentation of $\pi_G$ were already known. For genus 0, it was calculated in \cite{Poonen} but scattered across several theorems there; we summarize it in the table below. For genus 1 and 2, it was calculated in \cite{Poonen} \cite{Morton4Cycle} and was summarized nicely in \cite[Appendix C, Table 8]{DoyleCyclotomic}. We replicate this table here, and add the Cremona labels \cite{Cremona}/LMFDB labels \cite{LMFDB} whenever appropriate, for reader's convenience.

We also document the automorphism group $Aut(\pi_G)$ for each $G$. When genus of $X_1(G)$ is 0 or 1, we also record some information about exceptional curves associated to $G$, used in Lemma \ref{Lemma_Main_Degree1_Fiber_Count} and Lemma \ref{Lemma_Main_Degree2_Fiber_Count}. These were calculated manually in Sage \cite{Sage}.

All pictures from below are taken from \cite[Appendix C, Table 8]{DoyleCyclotomic}.

\newpage
\begin{longtable}[h!]{|l|c|l|c|l|}
	\caption{Models for dynamical modular curves of genus $0$}
	\label{TableModelOfForgetfulMapGenus0}
	\\\hline
	Label & Graph $G$ & Model of $\pi_G: X_1(G) \to \BP^1$ & $|Aut(\pi_G)|$ & Reference\\\hline
            &&&& \\
        & $\emptyset$ & $x$ & 1 & \\
            &&&& \\\hline
            &\multicolumn{4}{|l|}{\makecell[l]{
            $Aut(\pi_G) = \{id\}$ \\
            Exceptional curves: None
            }} \\\hline
            &&&& \\
        $4(1,1)$ & \includegraphics[scale=.6]{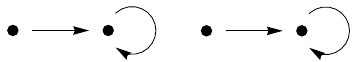} & $\dfrac{1}{4} - x^2$ & 2 & \cite[Theorem 1.1]{Poonen} \\
            &&&& \\\hline
            &\multicolumn{4}{|l|}{\makecell[l]{
            $Aut(\pi_G) = \{\sigma(x) = \pm x\}$ \\
            Exceptional curves: None
            }} \\\hline
            &&&& \\
        $4(2)$ & \includegraphics[scale=.6]{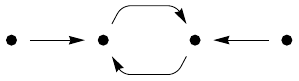} & $-\dfrac{3}{4} - x^2$ & 2 & \cite[Theorem 1.2]{Poonen} \\
            &&&& \\\hline
            &\multicolumn{4}{|l|}{\makecell[l]{
            $Aut(\pi_G) = \{\sigma(x) = \pm x\}$ \\
            Exceptional curves: None
            }} \\\hline
            &&&& \\
        $6(1,1)$ & \includegraphics[scale=.6]{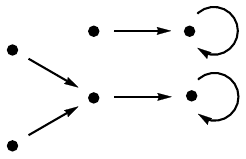} & $-\dfrac{2(x^2+1)}{(x^2-1)^2}$ & 2 & \cite[Theorem 3.2]{Poonen} \\
            &&&& \\\hline
            &\multicolumn{4}{|l|}{\makecell[l]{
            $Aut(\pi_G) = \{\sigma(x) = \pm x\}$ \\
            Exceptional curves: 1 smooth curve of genus 1
            }} \\\hline
            &&&& \\
        $6(2)$ & \includegraphics[scale=.6]{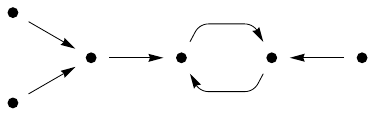} & $\dfrac{-x^4 - 2x^3 - 2x^2 + 2x - 1}{(x^2-1)^2}$ & 2 & \cite[Theorem 3.3]{Poonen} \\
            &&&& \\\hline
            &\multicolumn{4}{|l|}{\makecell[l]{
            $Aut(\pi_G) = \left\{\sigma(x) = x, -\frac{1}{x}\right\}$ \\
            Exceptional curves: 1 smooth curve of genus 1
            }} \\\hline
            &&&& \\
        $6(3)$ & \includegraphics[scale=.6]{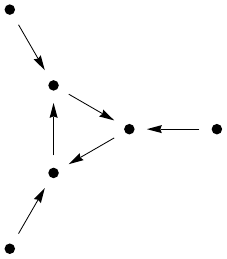} & $-\frac{x^6 + 2x^5 + 4x^4 + 8x^3 + 9x^2 + 4x + 1}{4x^2(x+1)^2}$ & 3 & \cite[Theorem 1.3]{Poonen} \\
            &&&& \\\hline
            &\multicolumn{4}{|l|}{\makecell[l]{
            $Aut(\pi_G) = \left\{\sigma(x) = x, -\frac{1}{x + 1}, - \frac{x + 1}{x}\right\}$ \\
            Exceptional curves: 1 smooth curve of genus 4
            }} \\\hline
            &&&& \\
        $8(2, 1, 1)$ & \includegraphics[scale=.6]{assets/graph8_211} & $-\dfrac{3x^4 + 10x^2 + 3}{4(x^2-1)^2}$ & 4 & \cite[Theorem 2.1]{Poonen} \\
            &&&& \\\hline
            &\multicolumn{4}{|l|}{\makecell[l]{
            $Aut(\pi_G) = \left\{\sigma(x) = \pm x, \pm \frac{1}{x}\right\}$ \\
            Exceptional curves: None
            }} \\\hline
\end{longtable}

\newpage
\begin{longtable}[h!]{|l|c|p{0.25\textwidth}|l|c|}
	\caption{Models for dynamical modular curves of genus $1$ \cite[Appendix C, Table 8]{DoyleCyclotomic}}
	\label{TableModelOfForgetfulMapGenus1}
	\\\hline
	Label & Graph $G$ & Model for $X_1(G)$ & Model of $\pi_G: X_1(G) \to \BP^1$ & $|Aut(\pi_G)|$\\\hline
		&&&& \\
	8(1,1)a & \includegraphics[scale=.5]{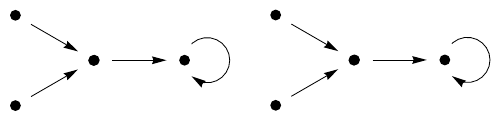} & {$y^2 = -(x^2 - 3)(x^2 + 1)$ \newline Cremona label: 24A4 \newline LMFDB label: 24.a5} & $-\dfrac{2(x^2 + 1)}{(x+1)^2(x - 1)^2}$ & 8 \\
		&&&& \\\hline
            &\multicolumn{4}{|l|}{\makecell[l]{
            $Aut(\pi_G) = \left\{
            \sigma(x,y) = (\pm x, \pm y), \left(\pm \frac{x^2 - 3}{y}, \pm \frac{4x}{x^2 + 1}\right)
            \right\}$ \\
            Exceptional curves: None
            }} \\\hline
            &&&& \\
	8(1,1)b & \includegraphics[scale=.5]{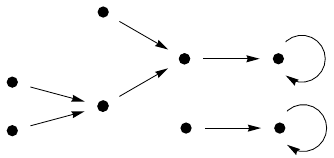} & {$y^2 = 2(x^3 + x^2 - x + 1)$ \newline Cremona label: 11A3 \newline LMFDB label: 11.a3} & $-\dfrac{2(x^2 + 1)}{(x+1)^2(x - 1)^2}$ & 2\\
		&&&& \\\hline
            &\multicolumn{4}{|l|}{\makecell[l]{
            $Aut(\pi_G) = \left\{
            \sigma(x,y) = (x, \pm y)
            \right\}$ \\
            Exceptional curves: 1 smooth curve of genus 4, 1 smooth curve of genus 17
            }} \\\hline
            &&&& \\
	8(2)a & \includegraphics[scale=.5]{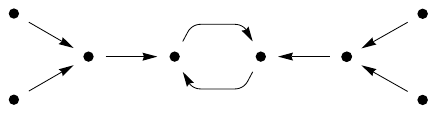} & {$y^2 = 2(x^4 + 2x^3 - 2x + 1)$ \newline Cremona label: 40A3 \newline LMFDB label: 40.a3} & $-\dfrac{x^4 + 2x^3 + 2x^2 - 2x + 1}{(x+1)^2(x - 1)^2}$ & 8\\
		&&&& \\\hline
            &\multicolumn{4}{|l|}{\makecell[l]{
            $Aut(\pi_G) = \left\{
            \sigma(x,y) = (x, \pm y), \left(-\frac{1}{x}, \pm \frac{y}{x^2}\right), \left(\frac{-y + 2x^2 + 2x - 2}{y - 2x}, \pm \frac{-4yx^2 - 2y^2 + 4y - 16x + 8}{(y - 2x)^2}\right), \right.$\\
            $\left. \left(- \frac{y - 2x}{-y + 2x^2 + 2x - 2}, \pm \frac{-4yx^2 - 2y^2 + 4y - 16x + 8}{(-y + 2x^2 + 2x - 2)^2}\right)
            \right\}$ \\
            Exceptional curves: None
            }} \\\hline
            &&&& \\
	8(2)b & \includegraphics[scale=.5]{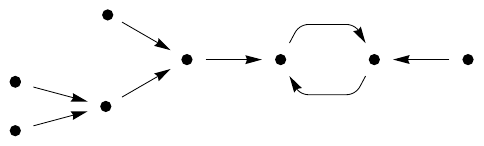} & {$y^2 = 2(x^3 + x^2 - x + 1)$ \newline Cremona label: 11A3 \newline LMFDB label: 11.a3} & $-\dfrac{x^4 + 2x^3 + 2x^2 - 2x + 1}{(x+1)^2(x - 1)^2}$ & 2\\
		&&&& \\\hline
            &\multicolumn{4}{|l|}{\makecell[l]{
            $Aut(\pi_G) = \left\{
            \sigma(x,y) = (x, \pm y)
            \right\}$ \\
            Exceptional curves: 1 smooth curve of genus 5, 1 smooth curve of genus 17
            }} \\\hline
            &&&& \\
	10(2,1,1)a & \includegraphics[scale=.5]{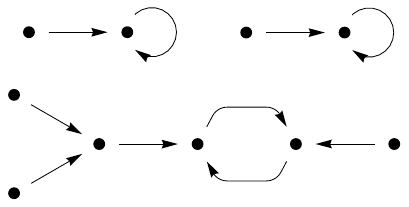} & {$y^2 = 5x^4 - 8x^3 + 6x^2 + 8x + 5$ \newline Cremona label: 17A4 \newline LMFDB label: 17.a4} & $-\dfrac{(3x^2 + 1)(x^2 + 3)}{4(x+1)^2(x - 1)^2}$ & 4\\
		&&&& \\\hline
            &\multicolumn{4}{|l|}{\makecell[l]{
            $Aut(\pi_G) = \left\{
            \sigma(x,y) = (x, \pm y), \left(-\frac{1}{x}, \pm \frac{y}{x^2}\right)
            \right\}$ \\
            Exceptional curves: 2 smooth curves of genus 5
            }} \\\hline
            &&&& \\
	10(2,1,1)b & \includegraphics[scale=.5]{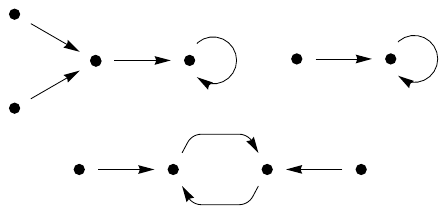} & {$y^2 = (5x^2 - 1)(x^2 + 3)$ \newline Cremona label: 15A8 \newline LMFDB label: 15.a7} & $-\dfrac{(3x^2 + 1)(x^2 + 3)}{4(x+1)^2(x - 1)^2}$ & 4\\
            &&&& \\\hline
            &\multicolumn{4}{|l|}{\makecell[l]{
            $Aut(\pi_G) = \left\{
            \sigma(x,y) = (\pm x, \pm y)
            \right\}$ \\
            Exceptional curves: 2 smooth curves of genus 5
            }} \\\hline
\end{longtable}

\newpage
\begin{longtable}[h!]{|l|c|p{0.35\textwidth}|l|c|}
	\caption{Models for dynamical modular curves of genus $2$ \cite[Appendix C, Table 8]{DoyleCyclotomic}}
	\label{TableModelOfForgetfulMapGenus2}
	\\\hline
	Label & Graph $G$ & Model for $X_1(G)$ & Model of $\pi_G: X_1(G) \to \BP^1$ & $|Aut(\pi_G)|$ \\\hline
		&&&& \\
            &&&& \\
	8(3) & \includegraphics[scale=.5]{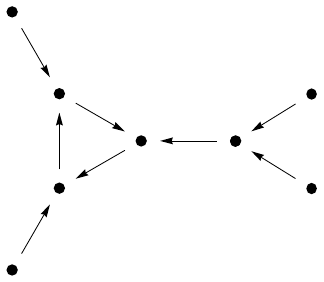} & {$y^2 = x^6 - 2x^4 + 2x^3 + 5x^2 + 2x + 1$ \newline LMFDB label: 743.a.743.1} & $-\frac{x^6 + 2x^5 + 4x^4 + 8x^3 + 9x^2 + 4x + 1}{4x^2(x+1)^2}$ & 2 \\
		&&&& \\\hline
            &\multicolumn{4}{|l|}{\makecell[l]{
            $Aut(\pi_G) = \left\{
            \sigma(x,y) = (x, \pm y)
            \right\}$ \\
            Exceptional curves: not computed
            }} \\\hline
            &&&& \\
	8(4) & \includegraphics[scale=.55]{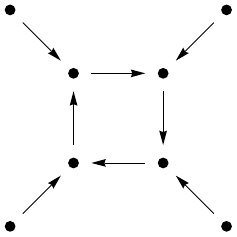} & {$y^2 = -x(x^2 + 1)(x^2 - 2x - 1)$ \newline LMFDB label: 256.a.512.1} & $\frac{(x^2 - 4x - 1)(x^4 + x^3 + 2x^2 - x + 1)}{4x(x+1)^2(x-1)^2}$ & 4 \\
		&&&& \\\hline
            &\multicolumn{4}{|l|}{\makecell[l]{
            $Aut(\pi_G) = \left\{
            \sigma(x,y) = (x, \pm y), \left(-\frac{1}{x}, \pm \frac{y}{x^3}\right)
            \right\}$ \\
            Exceptional curves: not computed
            }} \\\hline
            &&&& \\
	10(3,1,1) & \includegraphics[scale=.6]{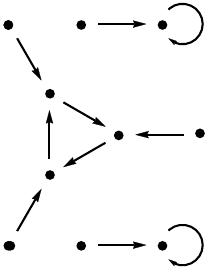} & {$y^2 = x^6 + 2x^5 + 5x^4 + 10x^3 + 10x^2 + 4x + 1$ \newline LMFDB label: 324.a.648.1} & $-\frac{x^6 + 2x^5 + 4x^4 + 8x^3 + 9x^2 + 4x + 1}{4x^2(x+1)^2}$ & 6 \\
		&&&& \\\hline
            &\multicolumn{4}{|l|}{\makecell[l]{
            $Aut(\pi_G) = \left\{
            \sigma(x,y) = (x, \pm y), \left(-\frac{1}{x + 1}, \pm \frac{y}{(x+1)^3}\right), \left(-\frac{x+1}{x}, \pm \frac{y}{x^3}\right)
            \right\}$ \\
            Exceptional curves: not computed
            }} \\\hline
            &&&& \\
	10(3,2) & \includegraphics[scale=.5]{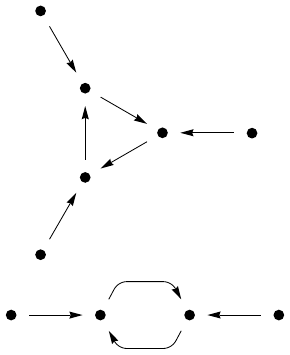} & {$y^2 = x^6 + 2x^5 + x^4 + 2x^3 + 6x^2 + 4x + 1$ \newline LMFDB label: 169.a.169.1} & $-\frac{x^6 + 2x^5 + 4x^4 + 8x^3 + 9x^2 + 4x + 1}{4x^2(x+1)^2}$ & 6 \\
		&&&& \\\hline
            &\multicolumn{4}{|l|}{\makecell[l]{
            $Aut(\pi_G) = \left\{
            \sigma(x,y) = (x, \pm y), \left(-\frac{1}{x + 1}, \pm \frac{y}{(x+1)^3}\right), \left(-\frac{x+1}{x}, \pm \frac{y}{x^3}\right)
            \right\}$ \\
            Exceptional curves: not computed
            }} \\\hline
\end{longtable}

\clearpage

\section{Proof of Proposition \ref{Theorem_Main_Degree2_CountingResult_Sym2_EC}}
\label{Appendix_Proof_Sym2_EC}
This appendix proves Proposition \ref{Theorem_Main_Degree2_CountingResult_Sym2_EC}, restated here. 
\begin{prop*}
Let $F/\BQ$ be a number field of degree $d$, and consider $\mathbf{O(1, 1)}$ on $\Sym^2 \BP^1$ over $F$ with the standard metric. 

Let $E$ be an elliptic curve over $F$. Let $f: E \to \BP^1$ be a morphism over $F$ of degree $k$ such that $f(x) = f(-x)$; in particular, $k$ is even. This induces
$$\Sym^2 f: \Sym^2 E \to \Sym^2 \BP^1,$$
also of degree $k$. 

Consider the metrized line bundle $\CL = (\Sym^2 f)^*\mathbf{O(1, 1)}$ on the source $\Sym^2 E$, and let $H_{\CL}: (\Sym^2 E)(F) \to [1, \infty)$ be the induced (absolute, multiplicative) height. Then
$$\#\{x \in (\Sym^2 E)(F): H_{\CL}(x) \leq B\} = C_{f, E, F} B^{2d/k} + O_{f, E, F}(B^{(2d-1)/k} (\log B)^{r/2 + 1}),$$
as $B \to \infty$, for some constant $C_{f, E, F} > 0$. Here $r = r(E/F)$ is the rank of the Mordell-Weil group $E(F)$.
\end{prop*}

\begin{assumption}
\textbf{For ease of presentation, we will assume in the following that:}
\label{Assumption_Sym2_EC}
\begin{itemize}
    \item Class number of $F$ is 1.
    \item $E$ is semi-stable over $F$.
    \item If $X / Spec(O_F)$ is the minimal regular model of $E/F$, and $v$ is a non-archimedean place, we assume that the vertical fiber $X_v := X \times_{\Spec(O_F)} \Spec(k(v))$ is irreducible.
\end{itemize}
\end{assumption}

Our basic strategy is to utilize the addition map $\Sym^2 E \to E$, which gives $\Sym^2 E$ a $\BP^1$-bundle structure. We will count the rational points fiber by fiber.

In Section 1-5, we review some Arakelov geometry and calculus on arithmetic surfaces. Section 6-8 recasts Proposition \ref{Theorem_Main_Degree2_CountingResult_Sym2_EC} into a lattice point counting problem in a region; both the lattice and the region vary with $R \in E(F)$. The geometry of numbers set up is very similar to Masser-Vaaler \cite{MasserVaaler}. In Section 9-10, we develop basic properties of the first successive minimum and the covolume of the involved lattices, and we track the dependence on $R \in E(F)$ carefully. In Section 11-12, we prove the Lipschitz parametrizability of boundary of the regions in which we count lattice points. These results come together in Section 13 and gives us an asymptotic count of rational point on each fiber. Finally in section 14, we finish the proof of Proposition \ref{Theorem_Main_Degree2_CountingResult_Sym2_EC}, and discuss modifications needed wihout Assumption \ref{Assumption_Sym2_EC}.

\subsection{First Chern form of a Hermitian metric}
For this subsection, let $M$ be a compact connected Riemann surface, and $\CL = (L, \|\cdot\|)$ be a line bundle $L$ on $M$ equipped with a continuous Hermitian metric $\|\cdot\|$.

Let $s$ be a nonzero, meromorphic section of $\CL$. The first Chern form of $\CL$ is a distribution on the space of smooth functions $C^{\infty}(M)$ defined by 
$$c_1(\CL) := \frac{\partial \overline{\partial}}{\pi i} \log \|s\| + \delta_{\div(s)}.$$
This is independent of the choice of section $s$.
\begin{remark}
When the metric is smooth, the first Chern form can be represented by a $(1, 1)$-form on $M$.
\end{remark}

\begin{prop}[Poincare-Lelong formula, {\cite[Theorem 13.4]{DemaillyComplexGeometry}}]
\label{Proposition_Poincare_Lelong}
$$\int_{M} c_1(\CL) = \deg(L).$$
\end{prop}

\begin{prop}
\label{Proposition_Metric_Determined_By_Curvature}
Let $\omega$ be a positive measure on $M$ with $\omega(M) = 1$. Suppose $\omega$ is a scalar multiple of the first Chern form of some smoothly (resp. continuously) metrized line bundle.

Then for any line bundle $L$ on $M$ with a nonzero meromorphic section $s$, there is a unique smooth (resp. continuous) Hermitian metric $\|\cdot\|$ on $L$ such that the first Chern form 
$$c_1(\|\cdot\|, L) = \deg (L) \cdot \omega,$$
and
$$\int_M \log \|s\| \omega = 0.$$
\end{prop}
\begin{proof}
Suppose $\omega = a \cdot c_1(L')$ for some line bundle $L'$ with smooth (resp. continuous) metric. By integrating over $M$ and Proposition \ref{Proposition_Poincare_Lelong}, we must have $a = \frac{1}{\deg(L')}$. 

Take an arbitrary smooth metric $\|\cdot\|_{L'}$ on $L'$. Note that $\deg (L') \omega$ and $c_1(\|\cdot\|_{L'}, L')$ are both Chern forms of $L'$ with possibly different metrics. By definition, one sees that their difference is $\partial \overline{\partial}$-exact. In particular,
$$\omega - \frac{1}{\deg(L')} \cdot c_1(\|\cdot\|_{L'}, L') = \frac{\partial \overline{\partial}}{\pi i} \varphi'$$
for some smooth (resp. continuous) function $\varphi'$ on $M$.

Consider a line bundle $L$ on $M$ with a nonzero meromorphic section $s$. Take an arbitrary smooth metric $\|\cdot\|_L$ on $L$. Then
$$\frac{\deg(L)}{\deg(L')} c_1(\|\cdot\|_{L'}, L') - c_1(\|\cdot\|_{L}, L)$$
is a $(1,1)$-form with integral 0 over $M$, hence $d$-exact (from de Rham's theorem), and hence $\partial \overline{\partial}$-exact \cite[Lemma VI.8.6]{DemaillyComplexGeometry}. Hence
$$\frac{\deg(L)}{\deg(L')} c_1(\|\cdot\|_{L'}, L') - c_1(\|\cdot\|_{L}, L) = \frac{\partial \overline{\partial}}{\pi i} \varphi$$
for some smooth function $\varphi$ on $M$.

Therefore,
\begin{align*}
& \deg(L) \cdot \omega \\
= & \deg(L)\left(\omega - \frac{1}{\deg(L')} \cdot c_1(\|\cdot\|_{L'}, L')\right) + \left(\frac{\deg(L)}{\deg(L')} c_1(\|\cdot\|_{L'}, L') - c_1(\|\cdot\|_{L}, L)\right) + c_1(\|\cdot\|_{L}, L) \\
= & \frac{\partial \overline{\partial}}{\pi i} \left( \deg (L) \cdot \varphi' +  \varphi\right) + c_1(\|\cdot\|_L, L).
\end{align*}
So the metric 
$$\|\cdot\| = \|\cdot\|_L \cdot \exp\left( \deg (L) \cdot \varphi' +  \varphi\right)$$
is smooth (resp. continuous), and satisfies $c_1(\|\cdot\|, L) = \deg(L) \cdot w$. By taking a unique scalar multiple of this metric, we can make sure  that
$$\int_M \log \|s\| \omega = 0.$$
\end{proof}

For this paper, all metrics we consider would be in class (S), as defined below. These metrics were considered in \cite{Zhang}.
\begin{defn}
Let $L$ be a line bundle on compact, connected Riemann surface $M$. A metric $\|\cdot\|$ on $L$ is of class (S) if
\begin{itemize}
    \item The first Chern form $c_1(\|\cdot\|, L)$ can always be extended to a continuous functional on the space of continuous functions $C(M)$ with supremum norm.
    \item There is a sequence of smooth metrics $\|\cdot\|_n$ on $M$ such that
    \begin{itemize}
        \item $\lim_{n \to \infty} \frac{\|\cdot\|_n}{\|\cdot\|} = 1$ uniformly on $M$.
        \item For continuous function $g$ on $M$, we have 
        $$\lim_{n \to \infty} \int_M g c_1(L, \|\cdot\|_n) = \int_M g c_1(L, \|\cdot\|).$$
    \end{itemize}
\end{itemize}
\end{defn}
Proposition \ref{Proposition_Metric_Determined_By_Curvature} also works for metrics of class (S).

\begin{example}
Consider $\BP^1(\BC)$ with coordinates $[x_0, x_1]$, and consider the line bundle $\CL = \mathbf{O(1)}$ on $\BP^1(\BC)$ with the standard metric. The standard metric is of class (S), since it is well-approximated by the $L^p$-metrics on $O(1)$ as $p \to \infty$.

One computes that the first Chern form $c_1(\mathbf{O(1)})$ to be the functional: for continuous function $g$ on $\BP^1(\BC)$,
$$\int_{\BP^1(\BC)} g c_1(\mathbf{O(1)}) := \frac{1}{2\pi} \int_0^{2\pi} g([1, e^{i\theta}])d\theta.$$
(See \cite[Section 6]{Zhang} for a reference.)
\end{example}

\begin{example}
\label{Example_Curvature_EC}
Let $E(\BC)$ be a complex elliptic curve. Let $f: E(\BC) \to \BP^1(\BC)$ be a morphism of degree $k$. Consider $\mathbf{O(1)}$ on $\BP^1(\BC)$ with the standard metric, and consider the metrized line bundle $\CL = f^*\mathbf{O(1)}$ on $E(\BC)$.

For $x \in \BP^1(\BC)$, let $A_f(g)$ be the average of $g$ over fiber of $f$ counted with multiplicity, i.e.
$$A_f(g)(x) := \sum_{y \in f^{-1}(x)} e_f(y) g(y),$$
with $e_f(y)$ the ramification index of $f$ at $y \in E(\BC)$.

Similar to the last example, one can then compute the first Chern form $c_1(f^*\mathbf{O(1)})$ to be the functional: for continuous function $g$ on $E(\BC)$,
$$\int_{E(\BC)} g c_1(f^*\mathbf{O(1)}) := \frac{1}{2\pi} \int_0^{2\pi} A_f(g)([1, e^{i\theta}]) d\theta.$$
Since $\deg(f^*O(1)) = k$, Poincare-Lelong formula implies that 
$$\int_{E(\BC)} c_1(f^*\mathbf{O(1)}) = k,$$
so $\frac{1}{k}c_1(f^*\mathbf{O(1)})$ is a probability measure on $E(\BC)$.

More generally, let $F$ be a number field. Let $E/F$ be an elliptic curve and $f: E \to \BP^1$ be a morphism of degree $k$ over $F$. For each archimedean place $v$, let $\overline{f^*\mathbf{O(1)}_v}$ be the metrized line bundle on $E(\overline{F_v})$ via base change. As before, we can define the probability measure 
$$d\mu_v := \frac{1}{k} c_1 (\overline{f^* \mathbf{O(1)}_v})$$ 
on $E(\overline{F_v})$ for each $v$.
\end{example}

\subsection{Admissible metrics and Faltings volume on elliptic curves}
\label{Subsection_Admissible_Metric_EC}
Following Faltings \cite{FaltingsRR}, we define the admissible metrics on line bundles of elliptic curves, and will use them to define the Faltings volume on global sections of a line bundle. These notions will be used when we discuss the Faltings-Riemann-Roch theorem in later sections. Finally, we will prove an upper bound on Faltings volume.

Let $E(\BC)$ be a complex elliptic curve. 
\begin{defn}[Canonical volume form]
Take a global holomorphic 1-form $\omega \in H^0(E, \Omega_E^1)$ such that its norm
$$\frac{i}{2}\int_{E(\BC)} \omega \wedge \overline{\omega} = 1.$$
The form $\nu_{can} := \frac{i}{2} \omega \wedge \overline{\omega}$ is called the canonical volume form on $E(\BC)$.
\end{defn} 

\begin{defn}[Admissible metric]
Let $L$ be a line bundle on $E(\BC)$. By \cite[Theorem V.13.9(b)]{DemaillyComplexGeometry} and Proposition \ref{Proposition_Metric_Determined_By_Curvature}, there exists a smooth metric $\|\cdot\|$ on $L$ such that
$$c_1(\|\cdot\|, L) = \deg(L) \cdot \nu_{can}.$$
We will call this an \textbf{admissible metric} on $L$. When $\CL = (L, \|\cdot\|)$ is metrized with an admissible metric, we will call $\CL$ an admissible line bundle.

There is a unique metric on $L$ that further satisfies
$$\int_{E(\BC)} \log \|s\| \cdot \nu_{can} = 0.$$
This metric will be called the \textbf{canonical admissible metric} on $L$. 
\end{defn}

The canonical admissible metric can be expressed in terms of the Arakelov-Green function (normalized as in Faltings \cite[p.393]{FaltingsRR}). Recall that the Arakelov-Green function $G: E(\BC) \times E(\BC) \to \BR_{\ge 0}$ is the unique function such that the following three properties hold:
\begin{itemize}
    \item $G(P, Q)^2$ is smooth on $E(\BC) \times E(\BC)$, and $G(P, Q)$ vanishes only on the diagonal with multiplicity 1. In particular, $G(P, Q)$ is bounded above on $E(\BC) \times E(\BC)$.
    \item For any $P \neq Q \in E(\BC)$, we have
    $$\frac{\partial_Q \overline{\partial_Q}}{\pi i} \log G(P, Q) = \nu_{can}(Q).$$
    \item For any $P \in E(\BC)$, 
    $$\int_{E(\BC)} \log G(P, Q) \nu_{can}(Q) = 0.$$
\end{itemize}
For $Q \in E(\BC)$, consider the line bundle $O_E(Q)$ with the canonical holomorphic section $1_Q$ that vanishes at $Q$. The canonical admissible metric on $O_E(Q)$ then satisfies
$$\|1_Q\|(P) = G(P, Q),$$
for $P \neq Q \in E(\BC)$. 

We record two properties (\cite[p.394]{FaltingsRR}) of admissible metrics and Arakelov-Green's function as follows:
\begin{itemize}
    \item Tensor product/inverse of admissible line bundles, with the natural metric, is admissible. Hence for a Weil divisor $D = n_1 Q_1 + \cdots + n_r Q_r$, with corresponding line bundle $O_E(D)$ and canonical section $1_D$, we have
    $$\|1_D\|(P) = G(P, Q_1)^{n_1} \cdots G(P, Q_r)^{n_r}$$
    for $P \neq Q_1, \cdots, Q_r \in E(\BC)$.
    \item Arakelov-Green function is symmetric, i.e. $G(P, Q) = G(Q, P)$.
\end{itemize}

We next work towards the Faltings metric on determinant of cohomology. For a complex vector space $V$, define
$$\det(V) := \wedge^{top} V$$
to be the top exterior product of $V$. For a line bundle $L$ on $E(\BC)$, define the determinant of its cohomology to be
$$\det H(E, L) := \det H^0 (E, L) \otimes \det H^1(E, L)^{-1}.$$
Note that when $H^1(E, L) = 0$, we have $\det H(E, L) = \det H^0(E, L)$.
\begin{thm}[Faltings metric, {\cite[Theorem 1]{FaltingsRR}}]
For each admissible line bundle $\CL$ on $E(\BC)$, one can define a unique metric on $\det H(E, \CL)$ such that the following properties hold:
\begin{itemize}
    \item Isometric isomorphism of line bundles $\CL_1 \cong \CL_2$ induces an isometry on 
    $$\det H(E, \CL_1) \cong \det H(E, \CL_2).$$
    \item If the metric on $\CL$ is changed by a factor $\alpha > 0$, then the metric on $\det H(E, \CL)$ is changed by $\alpha^{\deg \CL}$.
    \item For a Weil divisor $D$ on $E(\BC)$ and a point $P \in E(\BC)$, let $D_1 = D - P$, and consider the line bundles $O_E(D), O_E(D_1)$ with canonical admissible metrics. The fiber of $O_E(D)$ over $P$, denoted $O_E(D)[P]$, has a metric by restriction. The natural isomorphism
    $$\det H(E, O_E(D)) \cong \det H(E, O_E(D_1)) \otimes_{\BC} O_E(D)[P]$$
    is an isometry.
    \item For $\CL = \Omega_E^1$, we require $\det H(E, \Omega_E^1) \cong H^0(E, \Omega_E^1)$ to be metrized by the metric:
    $$\langle \omega_1, \omega_2 \rangle := \frac{i}{2} \int_{E(\BC)} \omega_1 \wedge \overline{\omega_2}.$$
\end{itemize}
This unique metric is called the Faltings metric on $\det H(E, \CL)$.
\end{thm}

\begin{defn}[Faltings volume]
\label{Definition_Faltings_Volume}
If $H^1(E, \CL) = 0$, then $\det H(E, \CL) = \det H^0(E, \CL)$. In that case, the Faltings metric is a volume form on $H^0(E, \CL)$, called the Faltings volume on $H^0(E, \CL)$.
\end{defn}

Let $O \in E(\BC)$ be the identity, and $R \neq O \in E(\BC)$. Note that $H^1(E, O_E(R + O)) = 0$ by Riemann-Roch, so we have Faltings volume on $H^0(E, O_E(R + O))$. 

On the other hand, equip $O_E(R + O)$ with the canonical admissible metric. This induces an $L^2$-norm on $g \in H^0(E, O_E(R + O))$, defined as
$$\|g\|_2^2 := \int_{E(\BC)} \|g\|^2 \nu_{can}.$$

\begin{prop}
\label{Proposition_Faltings_Volume_Unit_ball_Uniformly_Bounded}
The Faltings volume of the unit ball
$$\left\{g \in H^0(E, O_E(R + O)): \|g\|_2 \leq 1\right\}$$
is bounded above, uniformly over $R \in E(\BC)$.
\end{prop}
\begin{proof}
We follow the approach of Faltings \cite[Theorem 2]{FaltingsRR}. 

Start with a fixed $R \in E(\BC)$. For $P, Q \in E(\BC)$, the lines $\det H(E, O_E(R + O - P - Q))$ varies nicely with $P, Q$: they are the fibers of a line bundle $\CN$ on $E(\BC)^2$ (see \cite[p.297]{Chinburg} for a construction of $\CN$). By fixing a polarization $\Pic^0(E) \cong E$ via $O_E(R - S) \to S$, we see that the map $\varphi: E(\BC)^2 \to \Pic^0(E) \cong E$,
$$\varphi(P, Q) = O_E(R + O - P - Q) \cong O_E(R - (P+Q)) \cong P + Q$$
is the addition map. Moreover, from the proof of Faltings \cite[Theorem 1]{FaltingsRR}, 
$$i: \CN \cong \varphi^*O_E(-R),$$
where the Faltings metric on $\CN$ is the pullback of some admissible metric on $O_E(-R)$. Hence the Faltings metric is a constant multiple of the canonical admissible metric on $O_E(-R)$, with the constant $c_E > 0$ depending only on $E$.

For general points $P, Q \in E(\BC)$ where $P + Q \neq R$, the line bundle $O_E(R + O - P - Q)$ is of degree 0 and has no global sections. So $\det H(E, O_E(R + O - P - Q)) \cong \BC$. Moreover at such points, under the isomorphism $i$, the image of $1$ in $\det H(E, O_E(R + O - P - Q))$ corresponds to the section $1_{-R}$ of $O_E(-R)$ at $P+Q$. This correspondence implies that
$$\|1\|^2_{Fal} = c_E^2 \|1_{-R} (P+Q)\|^2 = \frac{c_E^2}{G(P + Q, R)^2}.$$

The restriction map at general points $P, Q \in E(\BC)$ yields an exact sequence
$$0 \to O_E(R + O - P - Q) \to O_E(R + O) \to O_E(R+O)[P] \oplus O_E(R + O)[Q] \to 0$$
which yields a natural isomorphism
$$\det H(E, O_E(R+O)) \cong \det H(E, O_E(R + O - P - Q)) \otimes O_E(R + O)[P] \otimes O_E(R + O)[Q].$$
One computes the norm of this map to be $G(P, Q)$.

Let $1_{R+O}$ be the canonical section for $O_E(R+O)$. Let $f_1 1_{R+O}, f_2 1_{R+O}$ be an orthonormal basis of $H^0(E, O_E(R + O))$, under the $L^2$-norm induced from admissible metric; here $f_i$ are meromorphic functions on $E$ with at worst a pole at $R$ and $O$. Under the isomorphism, 
$$f_1 1_{R+O} \wedge f_2 1_{R+O} \to (f_1(P)f_2(Q) - f_2(P)f_1(Q)) 1 \otimes 1_{R+O}(P) \otimes 1_{R+O}(Q).$$
We hence have
\begin{align*}
&\vol_{Fal} \left(\left\{g \in H^0(E, O_E(R + O)): \int_{E(\BC)}\|g\|^2 \nu_{can} \leq 1\right\}\right) \\
= & \frac{\pi^2}{2} \|f_1 1_{R+O} \wedge f_2 1_{R+O}\|_{Fal}^2 \\
= & \frac{\pi^2}{2} \frac{1}{G(P, Q)^2} \cdot \left|f_1(P)f_2(Q) - f_2(P)f_1(Q)\right|^2 \|1\|_{Fal}^2 \|1_{R+O} (P)\|^2\|1_{R+O} (Q)\|^2 \\
= & \frac{\pi^2 c_E^2}{2} \frac{1}{G(P+Q, R)^2 G(P, Q)^2} \left|f_1(P)f_2(Q) - f_2(P)f_1(Q)\right|^2 \|1_{R+O} (P)\|^2\|1_{R+O} (Q)\|^2
\end{align*}
We integrate over 
$$S_{\epsilon} = \{(P, Q) \in E(\BC)^2: G(P, Q) > \epsilon, \, G(P+Q, R) > \epsilon\}.$$
for a sufficiently small $\epsilon > 0$, such that $\vol(S_{\epsilon}) > 0$. Then,
\begin{align*}
& \vol_{Fal} \left(\{g \in H^0(E, O_E(R + O)): \int_{E(\BC)}\|g\|^2 \nu_{can} \leq 1\}\right) \\
& = \frac{\pi^2 c_E^2}{2 \vol(S_{\epsilon})} \int_{(P, Q) \in S_{\epsilon}} \frac{\left|f_1(P)f_2(Q) - f_2(P)f_1(Q)\right|^2}{G(P+Q, R)^2 G(P, Q)^2} \|1_{R+O} (P)\|^2\|1_{R+O} (Q)\|^2 \nu_{can}^2 \\
& \ll_{E, \epsilon} \int_{(P, Q) \in S_{\epsilon}} \left|f_1(P)f_2(Q) - f_2(P)f_1(Q)\right|^2 \|1_{R+O} (P)\|^2\|1_{R+O} (Q)\|^2 \nu_{can}^2.
\end{align*}
Since $f_1 1_{R+O}, f_2 1_{R+O} \in H^0(E, O_E(R+O))$ are orthonormal, we have
\begin{align*}
& \, \int_{(P, Q) \in S_{\epsilon}} \left|f_1(P)f_2(Q) - f_2(P)f_1(Q)\right|^2 \|1_{R+O} (P)\|^2\|1_{R+O} (Q)\|^2 \nu_{can}^2 \\
\leq & \, \int_{E(\BC)^2}\left|f_1(P)f_2(Q) - f_2(P)f_1(Q)\right|^2 \|1_{R+O} (P)\|^2\|1_{R+O} (Q)\|^2 \nu_{can}^2 \\
= & \, 2 \left(\int_{E(\BC)} \|f_1(P) 1_{R+O}(P)\|^2 \nu_{can}\right)^2 \\
= & \, 2.
\end{align*}
Thus we have
$$\vol_{Fal} \left(\left\{g \in H^0(E, O_E(R + O)): \int_{E(\BC)}\|g\|^2 \nu_{can} \leq 1\right\}\right) \ll_{E, \epsilon} 1,$$
uniformly over $R \in E(\BC)$ as desired.
\end{proof}

\subsection{Regular model of $E$}
Let $E$ be an elliptic curve over $F$. 
\begin{defn}[Regular model of $E$]
A regular model $X \to \Spec(O_F)$ is an integral, regular, proper, flat scheme of dimension 2, whose generic fiber is isomorphic to $E \to \Spec(F)$.   
\end{defn}

For elliptic curves, there is a particularly nice regular model: the minimal regular model \cite[Definition 9.3.14]{Liu}. Minimal regular model exists and is unique for elliptic curves \cite[Theorem 9.3.21]{Liu}. We will eventually work with minimal regular models so that Proposition \ref{Proposition_RobindeJong} applies, but will stay at the level of regular models for now.

\subsection{Weil divisors on regular model of $E$}
\label{Section_Weil_Divisors_On_Regular_Model}

Let $X$ be a regular model of $E$. There are two types of irreducible Weil divisors on $X$:
\begin{itemize}
    \item The horizontal divisors: those whose image under $X \to \Spec(O_F)$ is the entire $\Spec(O_F)$. For $Q \in E(\overline{F})$, we will denote its closure in $X$ by $D_Q$; this is a horizontal divisor, and all irreducible horizontal divisor on $X$ arises this way.
    \item The vertical divisors: those whose image under $X \to \Spec(O_F)$ is a closed point. 
    
    For a non-archimedean place $v$, let $k(v)$ be the residue field, and $X_v$ be the fiber of $X$ over $\Spec(k(v))$. By Assumption \ref{Assumption_Sym2_EC}, $X_v$ is irreducible.
    
    Since $X$ is regular and $X_v$ is irreducible of codimension 1, each $X_v$ corresponds to a discrete valuation $\val_{X_v}$ on $K(X)$, or equivalently a discrete absolute value $|\cdot|_{X_v}$, normalized such that for $a \in F$,
    $$|a|_{X_v} = |a|_v.$$
\end{itemize}
We will also consider principal divisors. A rational function $g$ on $E$ uniquely extends to a rational function $g_X$ on $X$. We note that if 
$$\div (g) = \sum_i a_i Q_i$$
for $Q_i \in E(\overline{F})$, then
$$\div(g_X) = \sum_i a_i D_{Q_i} + \sum_{v \in M_{F, finite}} \val_{X_v}(g_X) X_v.$$

\subsection{Intersection theory on regular model of $E$}
Let $E/F$ be an elliptic curve, and $X / \Spec(O_F)$ be a regular model of $E$. For two smoothly metrized line bundles $\CL$, $\CM$ on $X$, Deligne \cite{Deligne} defined an intersection number $\langle \CL, \CM \rangle$ which is bilinear and includes contributions from archimedean places. The same definition works for metrized line bundle of class (S) as well (\cite[Section 1]{Zhang}). From here, we can define
\begin{itemize}
    \item The intersection number between a metrized line bundle $\CL$ of class (S), and a Weil divisor $D$. 
    
    This can be done by first equipping $O_X(D)$ with the metric whose first Chern form equals that of $\CL$, via Proposition \ref{Proposition_Metric_Determined_By_Curvature}; we then take the intersection number of $\CL$ and $O_X(D)$ under this metric.

    For horizontal Weil divisor, this definition is equivalent to the degree of metrized line bundle pulled back to $\Spec(O_F)$ \cite[Proposition 4.20]{Moriwaki}.
    \item The intersection number between two Weil divisors $D_1, D_2$.
    
    By Proposition \ref{Proposition_Metric_Determined_By_Curvature}, we can equip $O_X(D_1)$ and $O_X(D_2)$ with metric that has the same first Chern form, then take the intersection number of such metrized line bundles. This definition is compatible with Arakelov \cite{Arakelov}, where the canonical admissible metric is always used.
\end{itemize}
We will use $\langle -, - \rangle$ to denote this intersection pairing. The precise definition would not matter to us; we will only use the formal properties of intersection number in the following propositions.

The intersection number is important for us because height can be expressed via this intersection number.
\begin{prop}[{\cite[Proposition 4.20]{Moriwaki}}]
\label{Proposition_Height_As_Intersection}
Let $X$ be a regular model of $E/F$. Let $\CL$ be a metrized line bundle on $X$, and $H_{\CL}$ be the absolute, multiplicative Weil height on $X$ induced by $\CL$.

If $Q \in E(\overline{F})$ and $D_Q$ is its closure in $X$, then
$$\log H_{\CL}(Q) = \frac{1}{[F(Q):\BQ]} \langle \CL, D_Q \rangle$$
\end{prop}

We document other properties of the intersection number as follows.

\begin{prop}[Intersection with a vertical divisor, {\cite[Lemma 4.2]{Moriwaki}}]
\label{Proposition_Intersection_Vertical_Divisor}
Let $X$ be a regular model of $E/F$. Let $\CL$ be a metrized line bundle on $X$.

If $v$ is a non-archimedean place of $F$, then
$$\langle \CL, X_v \rangle = \deg(\CL) \log |k(v)|.$$
\end{prop}

\begin{prop}[Intersection with a principal divisor, {\cite[Proposition 4.17(2)]{Moriwaki}}]
\label{Proposition_Intersection_Principal_Divisor}
Let $X$ be a regular model of $E/F$, and $\CL$ a metrized line bundle on $X$. 

For each archimedean place $v$, let $\overline{\CL_v}$ be the metrized line bundle over $X(\overline{F_v})$, induced from $\CL$; let $c_1(\overline{\CL_v})$ be the first Chern form of $\overline{\CL_v}$.

For a rational function $g_X$ on $X$, we have
$$\langle \CL, \div(g_X) \rangle = \sum_{v \in M_{F, \infty}} \int_{X(\overline{F_v})} \log |g_X|_v c_1(\overline{\CL_v}).$$
\end{prop}

\begin{prop}[Adjunction formula, {\cite[Lemma 4.26]{Moriwaki}}]
\label{Proposition_Adjunction_Formula}
Let $X / \Spec(O_F)$ be a regular model of $E/F$. The relative dualizing sheaf $\omega_X$ exists \cite[Theorem 6.4.32]{Liu}. For each archimedean place, we metrize $\omega_X$ with the canonical admissible metric (Section \ref{Subsection_Admissible_Metric_EC}).

Let $Q \in E(\overline{F})$, and $D_Q$ be the irreducible horizontal divisor on $X$. Metrize $O_X(D_Q)$ with the canonical admissible metric. Then
$$\langle O_X(D_Q), O_X(D_Q) + \omega_X \rangle = 0.$$
\end{prop}

\begin{prop}[Faltings-Hriljac theorem]
\label{Proposition_Faltings_Hriljac}
Let $X$ be a regular model of $E/F$. 

Let $O \in E(F)$ be the identity, and $Q \in E(F)$. Metrize $O_X(D_Q - D_O)$ with the canonical admissible metric. If $H_{NT}$ is the (multiplicative) N\'{e}ron-Tate height of $E$, then
$$\langle O_X(D_Q - D_O), O_X(D_Q - D_O) \rangle = - 2d \log H_{NT}(Q)$$
\end{prop}
\begin{proof}
We assumed that all vertical fibers are irreducible (Assumption \ref{Assumption_Sym2_EC}). By Proposition \ref{Proposition_Intersection_Vertical_Divisor}, $D_Q - D_O$ is perpendicular to all vertical divisors. Faltings-Hriljac theorem \cite[Theorem 4]{FaltingsRR} thus applies to get the result.
\end{proof}

The next result is the reason we work with minimal regular model, instead of any regular model.
\begin{prop}[{\cite[Theorem 6.1, Proposition 6.2]{RobindeJong}}]
\label{Proposition_RobindeJong}
Let $X$ be the minimal regular model of $E/F$. 

Let $O \in E(F)$ be the identity, and $Q \in E(F)$. Then,
$$\langle O_X(D_Q), O_X(D_Q) \rangle = \langle O_X(D_O), O_X(D_O) \rangle.$$
If furthermore $E$ is semi-stable over $F$, then
$$\langle O_X(D_O), O_X(D_O) \rangle = - \frac{1}{12} \log |N_{F/\BQ}(\Delta_{E/F})|,$$
where $\Delta_{E/F}$ is the minimal discriminant of the elliptic curve $E/F$.
\end{prop}

\subsection{Reinterpreting counting points on $\Sym^2 E$ as counting rational functions on $E$}
We first recall the set up of Proposition \ref{Theorem_Main_Degree2_CountingResult_Sym2_EC}. 

Let $F/\BQ$ be a number field of degree $d$, and consider $\mathbf{O(1, 1)}$ on $\Sym^2 \BP^1$ over $F$ with the standard metric. 

Let $E$ be an elliptic curve over $F$, and let $X$ over $\Spec(O_F)$ be the minimal regular model. Let $f: E \to \BP^1$ be a morphism over $F$ of degree $k$ such that $f(x) = f(-x)$; in particular, $k$ is even. This induces
$$\Sym^2 f: \Sym^2 E \to \Sym^2 \BP^1,$$
also of degree $k$. 

Consider the metrized line bundle $\CL = (\Sym^2 f)^*\mathbf{O(1, 1)}$ on the source $\Sym^2 E$, and let $H_{\CL}: (\Sym^2 E)(F) \to [1, \infty)$ be the induced (absolute, multiplicative) height.

Our eventual goal is to prove Proposition \ref{Theorem_Main_Degree2_CountingResult_Sym2_EC}, i.e. getting an asymptotic formula for
$$\#\{x \in (\Sym^2 E)(F): H_{\CL}(x) \leq B\}$$
as $B \to \infty$. As a first step, we want to reinterpret this point counting problem in terms of counting rational functions on $E$ instead.

If $D$ is a Weil divisor on $E$, we let $O_E(D)$ be the corresponding line bundle, and $H^0(E, O_E(D))$ be its global sections; this can also be identified with the rational functions $g$ on $E$ satisfying $\div(g) + D \ge 0$. Any such $g$ extends uniquely to a rational function $g_X$ on $X$.

\begin{prop}
\label{Proposition_Step_1_Rewrite_Counting_Rational_Functions}
\begin{align*}
& \#\{x \in (\Sym^2 E)(F): H_{\CL}(x) \leq B\} \\
= & \sum_{R \in E(F)} \#\left\{g_X \in (H^0(E, O_E(R + O)) - 0) / F^{\times}: \, \right.\\
& \left.\,\,\,\sum_{v \in M_{F, finite}} \log |g_X|_{X_v} + \sum_{v \in M_{F, \infty}} \int_{X(\overline{F_v})} \log |g_X|_v d \mu_v
\leq 
\frac{d}{k}\left(\log B - \log H(f(R)) - \log H(f(O))\right)\right\}
\end{align*}
\end{prop}
\begin{proof}
Our goal is to count
\begin{align*}
& \, \#\{x \in (\Sym^2 E)(F): H_{\CL}(x) \leq B\} \\
= & \, \#\{\{Q_1, Q_2\} \in (\Sym^2 E)(F): H(f(Q_1))H(f(Q_2)) \leq B\} \\
= & \, \#\{\{Q_1, Q_2\} \in (\Sym^2 E)(F): \log H(f(Q_1)) + \log H(f(Q_2)) \leq \log B\}
\end{align*}
We can first fix the sum $Q_1 + Q_2 = R \in E(F)$, then count the points in each fiber $(Sym^2 E)(F) \to E(F)$ over $R$. This gives 
\begin{align*}
& \#\{x \in (\Sym^2 E)(F): H_{\CL}(x) \leq B\} \\
= & \sum_{R \in E(F)} \#\{\{Q_1, Q_2\} \in (\Sym^2 E)(F): Q_1 + Q_2 = R, \, \log H(f(Q_1)) + \log H(f(Q_2)) \leq \log B\}
\end{align*}
For each $R \in E(F)$, we next try to rewrite the inner sum
\begin{equation}
\label{Equation_Sym2_EC_Step1}
\#\{\{Q_1, Q_2\} \in (\Sym^2 E)(F): Q_1 + Q_2 = R, \, \log H(f(Q_1)) + \log H(f(Q_2)) \leq \log B\}.
\end{equation}

By Riemann-Roch theorem, there is a bijection between
$$(H^0(E, O_E(R + O)) - 0) / F^{\times} \leftrightarrow \{\{Q_1, Q_2\} \in (\Sym^2 E)(F): Q_1 + Q_2 = R\}$$
as follows: for $g \in H^0(E, O_E(R + O)) - 0$, its divisor is of the shape
$$\div(g) = \frac{1}{[F(Q_1):F]} \left(Q_1 + Q_2\right) - R - O.$$ 
We can then send $g \to \{Q_1, Q_2\}$.

Hence to count the pair $\{Q_1, Q_2\} \in (\Sym^2 E)(F)$, we can count $g \in (H^0(E, O_E(R + O)) - 0) / F^{\times}$ instead. We need to express the condition
$$\log H(f(Q_1)) + \log H(f(Q_2)) \leq \log B$$
in terms of $g$, which we do next.

For each $g \in H^0(E, O_E(R + O))$, the unique extension $g_X$ on $X$ has divisor
$$
\div (g_X) = \frac{1}{[F(Q_1):F]} \left(D_{Q_1} + D_{Q_2}\right) - D_R - D_O + \sum_{v \in M_{F, finite}} \val_{X_v}(g_X) X_v.
$$

By Proposition \ref{Proposition_Height_As_Intersection}, 
$$\log H(f(Q_1)) + \log H(f(Q_2)) = \frac{1}{[F(Q_1):\BQ]} \langle f^*\mathbf{O(1)}, D_{Q_1} + D_{Q_2} \rangle.$$
Expressing $D_{Q_1} + D_{Q_2}$ in terms of $\div(g_X)$, we get
$$\log H(f(Q_1)) + \log H(f(Q_2)) = \frac{1}{d} \langle f^*\mathbf{O(1)}, D_R + D_O - \sum_{v \in M_{F, finite}} \val_{X_v}(g_X) X_v + \div(g_X) \rangle.$$
Using bilinearity of intersection number, and simplifying each term using Proposition \ref{Proposition_Height_As_Intersection}, Proposition \ref{Proposition_Intersection_Vertical_Divisor} and Proposition \ref{Proposition_Intersection_Principal_Divisor}, we get
\begin{align*}
&\log H(f(Q_1)) + \log H(f(Q_2)) \\
= & \log H(f(R)) + \log H(f(O))+ \frac{1}{d} \left(- \sum_{v \in M_{F, finite}} \val_{X_v}(g_X) \cdot k\log |k(v)| + \right.\\
& \,\,\,\left. \sum_{v \in M_{F, \infty}} \int_{X(\overline{F_v})} \log |g_X|_v c_1\left(\overline{f^*\mathbf{O(1)}_v}\right) \right) \\
= & \log H(f(R)) + \log H(f(O))+ \frac{k}{d} \left(\sum_{v \in M_{F, finite}} \log |g_X|_{X_v} + \sum_{v \in M_{F, \infty}} \int_{X(\overline{F_v})} \log |g_X|_v d \mu_v \right) \\
\end{align*}
The last equation comes from how we normalize $|\cdot|_{X_v}$ (Section \ref{Section_Weil_Divisors_On_Regular_Model}), and how $d\mu_v$ was defined (Example \ref{Example_Curvature_EC}).

Therefore the condition
$$\log H(f(Q_1)) + \log H(f(Q_2)) \leq \log B$$
is equivalent to 
$$\sum_{v \in M_{F, finite}} \log |g_X|_{X_v} + \sum_{v \in M_{F, \infty}} \int_{X(\overline{F_v})} \log |g_X|_v d \mu_v
\leq 
\frac{d}{k}\left(\log B - \log H(f(R)) - \log H(f(O))\right)
$$
and the result follows.
\end{proof}

\subsection{Finding good representative of $(H^0(E, O_E(R + O)) - 0) / F^{\times}$} We next try to pin down a specific representative for each equivalence class of $(H^0(E, O_E(R + O)) - 0) / F^{\times}$, so that the $F^{\times}$-action can be removed.

Fix a $g \in H^0(E, O_E(R + O)) - 0$, and let $g_X$ be its unique extension to $X$.

\textbf{Pinning down the finite places.} 

For finite places $v$, since $|\cdot|_{X_v}$ extends the valuation on $F_v$,
$$\sum_{v \in M_{F, finite}} \log |g_X|_{X_v}$$
equals the valuation of some fractional ideal of $F$. By Assumption \ref{Assumption_Sym2_EC}, $F$ has class number 1, so there exists some $a_g \in F^{\times} / O_F^{\times}$ such that
$$\sum_{v \in M_{F, finite}} \log |g_X|_{X_v} = \sum_{v \in M_{F, finite}} \log |a_g|_v.$$
In other words,
$$\sum_{v \in M_{F, finite}} \log |a_g^{-1} g_X|_{X_v} = 0.$$

We can hence impose the condition
$$\sum_{v \in M_{F, finite}} \log |g_X|_{X_v} = 0,$$
while we still have freedom to change $g_X$ by some factor in $O_F^{\times}$.

\textbf{Pinning down the infinite places.}

For infinite places $v$, we first introduce some notations as in Masser-Vaaler \cite{MasserVaaler}. They would be used in later geometry of numbers arguments as well.
\begin{itemize}
    \item Let $r_1, r_2$ be the number of real/complex embeddings of $F$, so that $r_1 + 2r_2 = d$. 
    \item Denote by $\Sigma \subset \BR^{r_1 + r_2}$ the plane 
    $$\Sigma := \{(x_1, \cdots, x_{r_1 + r_2}) \in \BR^{r_1 + r_2}: x_1 + \cdots + x_{r_1 + r_2} = 0\}.$$
    Also let
    $$\delta = (1, \cdots, 1, 2, \cdots, 2) \in \BR^{r_1 + r_2},$$
    where the first $r_1$ entries are 1, and the last $r_2$ entries are 2. 
    \item By Dirichlet's unit theorem, the log embedding $O_F^{\times} \to \BR^{r_1 + r_2}$ has image a lattice in $\Sigma$, and has kernel the roots of unity $\mu_F$ of $F$. 
    
    Let $\Omega$ be a fundamental domain of this lattice in $\Sigma$. For real $T > 0$, we denote the set $\Omega(T) \subset \BR^{r_1 + r_2}$ to be the vector sum
    $$\Omega(T) := \Omega + (-\infty, \log T] \delta.$$
\end{itemize}

For a place $v$, define the map $\ev_v: H^0(E, O_E(R + O)) \otimes_F F_v  - 0 \to \BR$ by
$$\ev_v(g_v) = \begin{cases}
\displaystyle \log |g_v|_{X_v} & \text{if $v$ is non-archimedean,} \\
\displaystyle \int_{X(\overline{F_v})} \log |g_v|_v d\mu_v & \text{if $v$ is archimedean.}
\end{cases}$$
Note that for $a \in O_F^{\times}$, we have
$$
\ev_v(ag_X) = \begin{cases}
    \ev_v(g_X) & \text{ if $v$ is non-archimedean,} \\
    \log |a|_v + \ev_v(g_X) & \text{ if $v$ is archimedean.}
\end{cases}
$$
Hence the condition 
$$g_X \in (H^0(E, O_E(R + O)) - 0) / O_F^{\times} \text{ and } \sum_{v \in M_{F, \infty}} \ev_v(g_X) \leq d \log T,$$ 
is equivalent to requiring 
$$g_X \in (H^0(E, O_E(R + O)) - 0) / \mu_F \text{ and } (\ev_v(g_X))_{v \in M_{F, \infty}} \in \Omega\left(T\right),$$
via multiplying $g_X$ by a suitable $a \in O_F^{\times}$. Since $\ev_v$ is $\mu_F$-invariant for any place $v$, we can remove the $\mu_F$-action directly.

We summarize how we pin down a good representative as follows.
\begin{prop}
\label{Proposition_Removing_Fcross_Action_1}
There is a bijection between
$$\#\{g_X \in (H^0(E, O_E(R + O)) - 0) / F^{\times}: 
\sum_{v \in M_{F, finite}} \log |g_X|_{X_v} + \sum_{v \in M_{F, \infty}} \int_{X(\overline{F_v})} \log |g_X|_v d \mu_v
\leq \log T\}$$
and
$$\frac{1}{|\mu_F|}\#\left\{g_X \in H^0(E, O_E(R + O)) - 0:
\sum_{v \in M_{F, finite}} \log |g_X|_{X_v} = 0, 
(\ev_v(g_X))_{v \in M_{F, \infty}} \in \Omega\left(T^{1/d}\right)
\right\}.$$
\end{prop}

\begin{defn}
\label{Definition_S_Omega_R}
For each $R \in E(F)$, define
$$S_{\Omega, R}(T) \subset \prod_{v \in M_{F, \infty}} H^0(E, O_E(R + O)) \otimes_F F_v$$
as the subset
$$S_{\Omega, R}(T) := \left\{(g_v) \in \prod_{v \in M_{F, \infty}} H^0(E, O_E(R + O)) \otimes_F F_v: (\ev_v(g_v)) \in \Omega(T)\right\}.$$
\end{defn}
Note that $S_{\Omega, R}(T) = T S_{\Omega, R}(1)$. Hence we can restate Proposition \ref{Proposition_Removing_Fcross_Action_1} as
\begin{prop}
\label{Proposition_Removing_Fcross_Action_2}
There is a bijection between
$$\#\{g_X \in (H^0(E, O_E(R + O)) - 0) / F^{\times}: 
\sum_{v \in M_{F, finite}} \log |g_X|_{X_v} + \sum_{v \in M_{F, \infty}} \int_{X(\overline{F_v})} \log |g_X|_v d \mu_v
\leq \log T\}$$
and
$$\frac{1}{|\mu_F|}\#\left\{g_X \in (H^0(E, O_E(R + O)) - 0):
\sum_{v \in M_{F, finite}} \log |g_X|_{X_v} = 0, 
(g_X)_{v \in M_{F, \infty}} \in T^{1/d} S_{\Omega, R}(1)
\right\}.$$
Here $g_X \in H^0(E, O_E(R + O))$ is embedded diagonally into $\prod_{v \in M_{F, \infty}} (H^0(E, O_E(R + O)) \otimes_F F_v)$.
\end{prop}

Applying this proposition with 
$$T = \exp\left(\frac{d}{k}(\log B - \log H(f(R)) - \log H(f(O)))\right),$$ 
together with Proposition \ref{Proposition_Step_1_Rewrite_Counting_Rational_Functions}, we get
\begin{cor}
\label{Corollary_Step_2_Remove_Fcross_Action}
\begin{align*}
& \#\{x \in (\Sym^2 E)(F): H_{\CL}(x) \leq B\} \\
= & \frac{1}{|\mu_F|} \sum_{R \in E(F)} \#\left\{g_X \in H^0(E, O_E(R + O)) - 0: \sum_{v \in M_{F, finite}} \log |g_X|_{X_v} = 0,\, \right.\\
& \left.\,\,\,(g_X)_{v \in M_{F, \infty}} \in \exp\left(\frac{1}{k}\left(\log B - \log H(f(R)) - \log H(f(O))\right)\right) S_{\Omega, R}(1)\right\}.
\end{align*}
Here $g_X \in H^0(E, O_E(R + O))$ is embedded diagonally into $\prod_{v \in M_{F, \infty}} (H^0(E, O_E(R + O)) \otimes_F F_v)$.
\end{cor}

\subsection{Reinterpreting counting points on $\Sym^2 E$ in terms of $H^0(X, O_X(D_R + D_O))$.}
We can repackage the constraint on finite places
$$\sum_{v \in M_{F, finite}} \log |g_X|_{X_v} = 0$$
using $H^0(X, O_X(D_R + D_O))$ and M\"obius inversion. The end result will be recorded in Corollary \ref{Corollary_Step_3_Reframe_Cohom_X}.

\begin{prop}
\label{Proposition_H0E_vs_H0X}
There is a bijection
$$H^0(X, O_X(D_R + D_O)) \cong \left\{g \in H^0(E, O_E(R + O)): \log |g_X|_{X_v} \le 0 \text{ for all non-arch place $v$}\right\}.$$
\end{prop}
\begin{proof}
Since cohomology commutes with flat base change, we have 
$$H^0(X, O_X(D_R + D_O)) \otimes_{O_F} F \cong H^0(E, O_E(R + O)).$$
Since $X \to \Spec(O_F)$ is flat, so $H^0(X, O_X(D_R + D_O))$ is torsion-free, thus the natural map 
$$H^0(X, O_X(D_R + D_O)) \to H^0(X, O_X(D_R + D_O)) \otimes_{O_F} F \cong H^0(E, O_E(R + O))$$
is injective. For any non-archimedean place $v$, by definition $g_X \in H^0(X, O_X(D_R + D_O))$ has no poles along $X_v$. This is equivalent to $\val_{X_v}(g_X) \ge 0$, or equivalently, $\log |g_X|_{X_v} \leq 0$, as desired.
\end{proof}

\begin{prop}
\begin{align*}
& \#\{g \in H^0(E, O_E(R + O)) - 0: \sum_{v \in M_{F, finite}} \log |g_X|_{X_v} = 0, (g_X)_{v \in M_{F, \infty}} \in S_{\Omega, R}(T)\} \\
= & \sum_{(a) \subset O_F} \mu((a)) \#\left\{g_X \in H^0(X, O_X(D_R + D_O)) - 0: (g_X)_{v \in M_{F, \infty}} \in S_{\Omega, R}\left(\frac{T}{|N_{F/\BQ}(a)|^{1/d}}\right)\right\}.
\end{align*}
Here $\mu(I)$ is the M\"obius function on a fractional ideal $I$ of $F$.
\end{prop}
\begin{proof}
By M\"obius inversion, counting
$$\#\{g \in H^0(E, O_E(R + O)) - 0: \sum_{v \in M_{F, finite}} \log |g_X|_{X_v} = 0\}$$
is equivalent to
$$\sum_{I \subset O_F} \mu(I) \#\{g \in H^0(E, O_E(R + O)) - 0: \log |g_X|_{X_v} \le \log |I|_v \text{ for all non-arch place $v$}\}.$$
Here $I$ is an integral ideal of $O_F$, and $\mu(I)$ is the M\"obius function on fractional ideals of $F$. The absolute value $|\cdot|_v$ on $F$ extends naturally to fractional ideals of $F$, hence we can make sense of $|I|_v$. 

By Assumption \ref{Assumption_Sym2_EC}, $F$ has class number 1, so all ideals are principal. Therefore,
\begin{align*}
& \#\left\{g \in H^0(E, O_E(R + O)) - 0: \sum_{v \in M_{F, finite}} \log |g_X|_{X_v} = 0, (g_X)_{v \in M_{F, \infty}} \in S_{\Omega, R}(T)\right\} \\
= & \sum_{(a) \subset O_F} \mu((a)) \#\left\{g \in H^0(E, O_E(R + O)) - 0: \right.\\
&\,\,\,\left.\log |g_X|_{X_v} \le \log |a|_v \text{ for all non-arch place $v$}, (g_X)_{v \in M_{F, \infty}} \in S_{\Omega, R}(T)\right\}
\end{align*}
Note that 
$$\sum_{v \in M_{F, \infty}} \ev_v (g_X) \leq d \log T,$$
is equivalent to
$$\sum_{v \in M_{F, \infty}} \ev_v (a^{-1}g_X) 
\leq d \log T - \sum_{v \in M_{\infty}} \log |a|_v
= d \log \left(\frac{T}{|N_{F/\BQ}(a)|^{1/d}}\right).
$$
So 
$$
(g_X)_{v \in M_{F, \infty}} \in S_{\Omega, R}(T) 
\Leftrightarrow
(a^{-1}g_X)_{v \in M_{F, \infty}} \in S_{\Omega, R}\left(\frac{T}{|N_{F/\BQ}(a)|^{1/d}}\right).
$$
Hence,
\begin{align*}
& \#\left\{g \in H^0(E, O_E(R + O)) - 0: \sum_{v \in M_{F, finite}} \log |g_X|_{X_v} = 0, (g_X)_{v \in M_{F, \infty}} \in S_{\Omega, R}(T)\right\} \\
= & \sum_{(a) \subset O_F} \mu((a)) \#\left\{g \in H^0(E, O_E(R + O)) - 0: \right.\\
&\,\,\,\left.\sum_{v \in M_{F, finite}} \log |a^{-1}g_X|_{X_v} \le 0,\,(a^{-1}g_X)_{v \in M_{F, \infty}} \in S_{\Omega, R}\left(\frac{T}{|N_{F/\BQ}(a)|^{1/d}}\right)\right\} \\
= & \sum_{(a) \subset O_F} \mu((a)) \#\left\{g \in H^0(E, O_E(R + O)) - 0: \right.\\
&\,\,\,\left.\sum_{v \in M_{F, finite}} \log |g_X|_{X_v} \le 0,\,(g_X)_{v \in M_{F, \infty}} \in S_{\Omega, R}\left(\frac{T}{|N_{F/\BQ}(a)|^{1/d}}\right)\right\} \\
\end{align*}
The result now follows by repackaging the constraint on finite place via Proposition \ref{Proposition_H0E_vs_H0X}.
\end{proof}

Combining the last proposition with Corollary \ref{Corollary_Step_2_Remove_Fcross_Action}, we have
\begin{cor}
\label{Corollary_Step_3_Reframe_Cohom_X}
\begin{align*}
& \#\{x \in (\Sym^2 E)(F): H_{\CL}(x) \leq B\} \\
= & \frac{1}{|\mu_F|} \sum_{R \in E(F)} \sum_{(a) \subset O_F} \mu((a)) \#\left\{g_X \in H^0(X, O_X(D_R + D_O)) - 0: \right.\\
& \left.\,\,\, (g_X)_{v \in M_{F, \infty}} \in \exp\left(\frac{1}{k}\left(\log B - \log H(f(R)) - \log H(f(O))\right) - \frac{1}{d} \log |N_{F/\BQ}(a)|\right) S_{\Omega, R}(1)\right\}
\end{align*}
\end{cor}

We have now transformed the counting problem on $\Sym^2 E$ to counting points of the lattice
$$H^0(X, O_X(D_R + D_O)) \into H^0(X, O_X(D_R + D_O)) \otimes_{\BZ} \BR \cong \prod_{v \in M_{F, \infty}} H^0(E, O_E(R + O)) \otimes_F F_v,$$
lying inside the homogeneous region 
$$\exp\left(\frac{1}{k}\left(\log B - \log H(f(R)) - \log H(f(O))\right) - \frac{1}{d} \log |N_{F/\BQ}(a)|\right) S_{\Omega, R}(1),$$
as $B \to \infty$, for each $R \in E(F)$. This can be approached by geometry of numbers. 

In the next few sections, we will develop the basic properties for this lattice, such as the first successive minimum and the covolume.

\subsection{The first successive minimum of the lattice $H^0(X, O_X(D_R + D_O))$}
For each $R \in E(F)$, we seek to understand the lattice
$$H^0(X, O_X(D_R + D_O)) \into H^0(X, O_X(D_R + D_O)) \otimes_{\BZ} \BR \cong \prod_{v \in M_{F, \infty}} H^0(E, O_E(R + O)) \otimes_F F_v.$$
In this section, we first define a metric on the ambient space $H^0(X, O_X(D_R + D_O) \otimes_{\BZ} \BR$, then prove a lower bound of the first successive minimum of the lattice.

The line bundle $O_E(R + O)$ has a canonical section $1_{R+O}$. For an archimedean place $v$, by Proposition \ref{Proposition_Metric_Determined_By_Curvature} there is a unique metric $\|\cdot\|_v$ on $O_{E(\overline{F_v})}(R + O)$ such that the first Chern form of the metric equals $d\mu_v$, and 
$$\int_{E(\overline{F_v})} \log \|1_{R+O}\|_v d\mu_v = 0.$$
\begin{defn}[Metric on $H^0(X, O_X(D_R + D_O) \otimes_{\BZ} \BR$]
\label{Definition_L2_metric_H0X}
For $g_v \in H^0(E, O_E(R + O)) \otimes_F F_v$, we define the $L^2$-norm
$$\|g_v\|_2^2 = \int_{E(\overline{F_v})} |g_v|^2 \|1_{R+O}\|_v^2 d \mu_v.$$

The $L^2$-norm on $H^0(X, O_X(D_R + D_O) \otimes_{\BZ} \BR \cong \prod_{v \in M_{F, \infty}} H^0(E, O_E(R + O)) \otimes_F F_v$ can now be defined by
$$\|(g_v)_{v \in M_{F, \infty}}\|_2^2 := \sum_{v \in M_{F, \infty}} \|g_v\|_2^2 = \sum_{v \in M_{F, \infty}} \int_{E(\overline{F_v})} |g_v|^2 \|1_{R+O}\|_v^2 d \mu_v.$$
\end{defn}

Recall that the first successive minimum $\lambda_1(\Lambda)$ of a lattice $\Lambda$ is defined to be the length of shortest nonzero vector of the lattice.

\begin{prop}
\label{Proposition_First_Successive_Minimum_H0}
The first successive minimum of $H^0(X, O_X(D_R + D_O))$ satisfies
$$\lambda_1(H^0(X, O_X(D_R + D_O))) \gg_{E, f, F} \frac{1}{(H(f(R))H(f(O)))^{1/k}}.$$
\end{prop}
\begin{proof}
By Proposition \ref{Proposition_H0E_vs_H0X}, we can identify $g_X \in H^0(X, O_X(D_R + D_O)) - 0$ as $g_X \neq 0 \in H^0(E, O_E(R + O)) - 0$ with 
$$\sum_{v \in M_{F, finite}} \log |g_X|_{X_v} \le 0.$$ 
Our goal is to find a lower bound for
$$\|g_X\|_2^2 = \sum_{v \in M_{F, \infty}} \int_{X(\overline{F_v})} |g_X|_v^2 \|1_{R+O}\|_v^2 d \mu_v.$$

Restricting $g_X$ to the generic fiber $E$, suppose that
$$\div(g) = \frac{1}{[F(Q_1):F]}(Q_1 + Q_2) - R - O.$$
As in the proof of Proposition \ref{Proposition_Step_1_Rewrite_Counting_Rational_Functions}, we have 
\begin{align*}
&\log H(f(Q_1)) + \log H(f(Q_2)) \\
= & \log H(f(R)) + \log H(f(O))+ \frac{k}{d} \left(\sum_{v \in M_{F, finite}} \log |g_X|_{X_v} + \sum_{v \in M_{F, \infty}} \int_{X(\overline{F_v})} \log |g_X|_v d \mu_v \right). \\
\end{align*}
Since 
$$\sum_{v \in M_{F, finite}} \log |g_X|_{X_v} \le 0,$$
we see that
$$\sum_{v \in M_{F, \infty}} \int_{X(\overline{F_v})} \log |g_X|_v d \mu_v
\ge \frac{d}{k} \left(\log H(f(Q_1)) + \log H(f(Q_2)) - \log H(f(R)) - \log H(f(O))\right).
$$
Recall that the logarithmic N\'{e}ron-Tate height $h_{NT}$ is a quadratic form on $E(F)$, and satisfies
$$\log H(f(Q)) = h_{NT}(Q) + O_{E, f, F}(1)$$
for any $Q \in E(F)$, with the error $O_{E, f, F}(1)$ independent of $Q$. 

Switching to the N\'{e}ron-Tate height, we see that
$$\sum_{v \in M_{F, \infty}} \int_{X(\overline{F_v})} \log |g_X|_v d \mu_v
\ge \frac{d}{k} \left(h_{NT}(Q_1) + h_{NT}(Q_2) - h_{NT}(R) - h_{NT}(O)\right) + O_{E, f, F}(1).
$$
Since $h_{NT}(O) = 0$, and
\begin{align*}
h_{NT}(Q_1) + h_{NT}(Q_2) = & \, \frac{1}{2} \left(h_{NT}(Q_1 + Q_2) + h_{NT}(Q_1 - Q_2)\right) \\
= & \, \frac{1}{2} \left(h_{NT}(R) + h_{NT}(Q_1 - Q_2)\right) \\
\ge & \, \frac{1}{2}h_{NT}(R),
\end{align*}
we have
$$\sum_{v \in M_{F, \infty}} \int_{X(\overline{F_v})} \log |g_X|_v d \mu_v
\ge - \frac{d}{2k} h_{NT}(R) + O_{E, f, F}(1).
$$
Switching back to the height $H(f(Q))$, we have
$$\sum_{v \in M_{F, \infty}} \int_{X(\overline{F_v})} \log |g_X|_v d \mu_v
\ge - \frac{d}{2k} \left(\log H(f(R)) + \log H(f(O))\right) + O_{E, f, F}(1),
$$
independent of $R \in E(F)$.

By Cauchy-Schwarz inequality and Jensen's inequality ($\log$ is concave), we have for each $v \in M_{F, \infty}$,
\begin{align*}
\int_{X(\overline{F_v})} |g_X|_v^2 \|1_{R+O}\|_v^2 d\mu_v
\ge & \, \left(\int_{X(\overline{F_v})} |g_X|_v \|1_{R+O}\|_v d \mu_v \right)^2 \\
\ge & \, \exp\left(2 \int_{X(\overline{F_v})} \left(\log |g_X|_v - \log \|1_{R+O}\|_v \right) d \mu_v\right) \\
= & \, \exp\left(2 \int_{X(\overline{F_v})} \log |g_X|_v d \mu_v\right).
\end{align*}
Hence,
\begin{align*}
\|g_X\|_2^2 = & \, \sum_{v \in M_{F, \infty}} \int_{X(\overline{F_v})} |g_X|_v^2 \|1_{R+O}\|_v^2 d \mu_v \\
\ge & \, \sum_{v \in M_{F, \infty}} \exp\left(2 \int_{X(\overline{F_v})} \log |g_X|_v d \mu_v\right) \\
\ge & \, (r_1 + r_2) \exp\left(\frac{2}{r_1 + r_2} \sum_{v \in M_{F, \infty}} \int_{X(\overline{F_v})} \log |g_X|_v d \mu_v\right) \text{ (by AM-GM inequality)} \\
\ge & \, (r_1 + r_2) \exp\left(- \frac{d}{k(r_1 + r_2)}\left(\log H(f(R)) + \log H(f(O))\right) + O_{E, f, F}(1)\right)
\end{align*}
Since $d = r_1 + 2r_2 \leq 2(r_1 + r_2)$, we have
$$\|g_X\|_2^2 \gg_{E, f, F} \exp\left(-\frac{2}{k} (\log H(f(R)) + \log H(f(O)))\right) = \frac{1}{(H(f(R))H(f(O)))^{2/k}}.$$
Taking square root on both sides, we get the desired result.
\end{proof}

\begin{prop}
\label{Proposition_Counting_Rational_Functions_Truncation}
If $T \ll_{E, f, F} (H(f(R))H(f(O)))^{-1/(2k)}$ with a sufficiently small big-$O$ constant, then 
$$\#\{g_X \in H^0(X, O_X(D_R + D_O)) - 0: (g_X)_{v \in M_{F, \infty}} \in S_{\Omega, R}(T)\} = 0.$$
\end{prop}
\begin{proof}
Suppose for the sake of contradiction that $g_X \in H^0(X, D_R + D_O) - 0$ exists. By Proposition \ref{Proposition_H0E_vs_H0X}, we can identify $g_X \in H^0(X, O_X(D_R + D_O)) - 0$ as $g_X \neq 0 \in H^0(E, O_E(R + O)) - 0$ with 
$$\sum_{v \in M_{F, finite}} \log |g_X|_{X_v} \le 0.$$ 

As a rational function on $E$, suppose
$$\div(g) = \frac{1}{[F(Q_1):F]}(Q_1 + Q_2) - R - O.$$
As in the proof of Proposition \ref{Proposition_Step_1_Rewrite_Counting_Rational_Functions}, we have 
\begin{align*}
&\log H(f(Q_1)) + \log H(f(Q_2)) \\
= & \log H(f(R)) + \log H(f(O))+ \frac{k}{d} \left(\sum_{v \in M_{F, finite}} \log |g_X|_{X_v} + \sum_{v \in M_{F, \infty}} \int_{X(\overline{F_v})} \log |g_X|_v d \mu_v \right) \\
\end{align*}
Since 
$$\sum_{v \in M_{F, finite}} \log |g_X|_{X_v} \le 0,$$
we see that
\begin{align*}
&\sum_{v \in M_{F, \infty}} \int_{X(\overline{F_v})} \log |g_X|_v d \mu_v \\
\ge & \frac{d}{k} \left(\log H(f(Q_1)) + \log H(f(Q_2)) - \log H(f(R)) - \log H(f(O))\right)
\end{align*}
Recall that the logarithmic N\'{e}ron-Tate height $h_{NT}$ is a quadratic form on $E(F)$, and satisfies
$$\log H(f(Q)) = h_{NT}(Q) + O_{E, f, F}(1)$$
for any $Q \in E(F)$, with the error $O_{E, f, F}(1)$ independent of $Q$. 

Switching to the N\'{e}ron-Tate height, we see that
$$\sum_{v \in M_{F, \infty}} \int_{X(\overline{F_v})} \log |g_X|_v d \mu_v
\ge \frac{d}{k} \left(h_{NT}(Q_1) + h_{NT}(Q_2) - h_{NT}(R) - h_{NT}(O)\right) + O_{E, f, F}(1).
$$
Since $h_{NT}(O) = 0$, and
\begin{align*}
h_{NT}(Q_1) + h_{NT}(Q_2) = & \, \frac{1}{2} \left(h_{NT}(Q_1 + Q_2) + h_{NT}(Q_1 - Q_2)\right) \\
= & \, \frac{1}{2} \left(h_{NT}(R) + h_{NT}(Q_1 - Q_2)\right) \\
\ge & \, \frac{1}{2}h_{NT}(R),
\end{align*}
we have
$$\sum_{v \in M_{F, \infty}} \int_{X(\overline{F_v})} \log |g_X|_v d \mu_v
\ge - \frac{d}{2k} h_{NT}(R) + O_{E, f, F}(1).
$$
Switching back to the height $H(f(Q))$, we have
$$\sum_{v \in M_{F, \infty}} \int_{X(\overline{F_v})} \log |g_X|_v d \mu_v
\ge - \frac{d}{2k} \left(\log H(f(R)) + \log H(f(O))\right) + O_{E, f, F}(1),
$$
independent of $R \in E(F)$.

By definition, $(g_X)_{v \in M_{F, \infty}} \in S_{\Omega, R}(T)$ implies
$$\sum_{v \in M_{F, \infty}} \int_{X(\overline{F_v})} \log |g_X|_v d \mu_v \leq d \log T.$$
Hence the existence of $g_X$ implies that
$$d \log T \ge - \frac{d}{2k} \left(\log H(f(R)) + \log H(f(O))\right) + O_{E, f, F} (1).
$$
which is equivalent to
$$T \gg_{E, f, F} (H(f(R))H(f(O)))^{-1/(2k)}$$
Hence if $T \ll_{E, f, F} (H(f(R))H(f(O)))^{-1/(2k)}$ with a sufficiently small big-$O$ constant, we get a contradiction if $g_X$ exists.
\end{proof}

\subsection{Covolume of the lattice $H^0(X, O_X(D_R + D_O))$}
We defined a $L^2$-norm on $H^0(X, O_X(D_R + D_O))$ in the last section, which induces a Haar measure $\vol_2$ on $H^0(X, O_X(D_R + D_O))$. To apply geometry of numbers, we want to compute the covolume of the lattice
$$\covol_2 \left(H^0(X, O_X(D_R + D_O))\right) := \vol_2 \left(H^0(X, O_X(D_R + D_O)) \otimes_{\BZ} \BR / H^0(X, O_X(D_R + D_O))\right)$$ 

It would be easier to calculate the covolume using another Haar measure on $H^0(X, O_X(D_R + D_O))$: the Faltings volume $\vol_{Fal}$. This measure allows the Faltings-Riemann-Roch theorem to apply. For lattice point counting, we only care about quotient of measures of two sets, so it does not matter which Haar measure we use.

We will work towards stating Faltings-Riemann-Roch theorem (following Faltings \cite{FaltingsRR}), and then calculate the Faltings covolume of $H^0(X, O_X(D_R + D_O))$.

\begin{defn}[Euler characteristic of finitely generated $O_F$-module]
For a finitely generated $O_F$-module $M$, together with a Haar measure $\vol$ on $M \otimes_{\BZ} \BR$, we can define the Euler characteristic of $M$ by
$$\chi(M) = - \log \left(\vol \left(M \otimes_{\BZ} \BR / M\right) / |M_{tor}|\right) + rank_{O_F}(M) \cdot \log \vol \left(O_F \otimes_{\BZ} \BR / O_F\right).$$
\end{defn}
\begin{remark}
By definition of discriminant of number fields,
$$\vol \left(O_F \otimes_{\BZ} \BR / O_F\right) = |\disc(O_F)|^{1/2}.$$
The term $\rank_{O_F}(M) \cdot \log \vol \left(O_F \otimes_{\BZ} \BR / O_F\right)$ is a normalization term to make sure that $\chi(O_F) = 0$.
\end{remark}

For each archimedean place $v$, $X(\overline{F_v})$ is an elliptic curve. In Section \ref{Subsection_Admissible_Metric_EC}, we defined admissible metrics on line bundles over $X(\overline{F_v})$ and Faltings metrics on the determinant of cohomology for admissible line bundles.

\begin{defn}
An metrized line bundle $\CL$ on $X$ is called admissible, if for each archimedean place $v$ the base change $\overline{\CL_v}$ on $X(\overline{F_v})$ is equipped with an admissible metric.
\end{defn}

Let $\CL$ be an admissible line bundle on $X$. For each archimedean place $v$, the Faltings metric gives a Haar measure on the formal difference
$$H^0(X(\overline{F_v}), \overline{\CL_v}) - H^1(X(\overline{F_v}), \overline{\CL_v}).$$
We therefore obtain a Haar measure on the formal difference
\begin{align*}
& \prod_{v \in M_{F, \infty}} H^0(X(\overline{F_v}), \overline{\CL_v}) - \prod_{v \in M_{F, \infty}} H^1(X(\overline{F_v}), \overline{\CL_v}) \\
= & \, H^0(X, \CL) \otimes_{\BZ} \BR - H^1(X, \CL) \otimes_{\BZ} \BR.
\end{align*}
This allows us to define
$$\chi(\CL) := \chi(H^0(X, \CL)) - \chi(H^1(X, \CL)).$$

\begin{example}
\label{Example_Volume_Euler_Characteristic}
Suppose $R \in E(F)$, and let $\CL = O_X(D_R + D_O)$, metrized with the canonical admissible metric at each archimedean place. We want to compute $\chi(\CL)$.

We first calculate $\chi(H^0(X, \CL))$. By definition,
\begin{align*}
& \chi(H^0(X, O_X(D_R + D_O))) \\
= & - \log \covol_{Fal}\left(H^0(X, O_X(D_R + D_O))\right) + \log |H^0(X, O_X(D_R + D_O))_{tor}| \\
& \, + \rank_{O_F} (H^0(X, O_X(D_R + D_O))) \cdot \log \vol_{Fal} \left(O_F \otimes_{\BZ} \BR / O_F\right).
\end{align*}
Recall that 
$$H^0(X, O_X(D_R + D_O))_{tor} = 0$$ 
since $X \to \Spec(O_F)$ is flat. Moreover, by Riemann-Roch theorem for curves,
$$\rank_{O_F} (H^0(X, O_X(D_R + D_O))) = \rank_F H^0(E, O_E(R + O)) = 2.$$
Hence,
\begin{align*}
& \chi(H^0(X, O_X(D_R + D_O))) \\
= & - \log \covol_{Fal}\left(H^0(X, O_X(D_R + D_O))\right) + 2 \log |\disc(O_F)|^{1/2} \\
= & - \log \covol_{Fal}\left(H^0(X, O_X(D_R + D_O))\right) + \log |\disc(O_F)|.
\end{align*}

For $H^1(X, \CL)$, again by Riemann-Roch theorem for curves, 
$$\rank_{O_F} (H^1(X, O_X(D_R + D_O))) = \rank_F H^1(E, O_E(R + O)) = 0,$$
so $H^1(X, O_X(D_R + D_O))$ is torsion. Since $X \to \Spec(O_F)$ is flat, each fiber has arithmetic genus 1, so the same argument actually shows that $H^1(X_v, O_X(D_R + D_O) |_{X_v}) = 0$ for any non-archimedean place $v \in \Spec(O_F)$. Grauert's theorem \cite[Corollary III.12.9]{Hartshorne} then applies to show that $H^1(X, O_X(D_R + D_O))$ is torsion-free. Hence $H^1(X, O_X(D_R + D_O)) = 0$, and
$$\chi(H^1(X, O_X(D_R + D_O))) = 0.$$

Therefore,
\begin{align*}
\chi(\CO_X(D_R + D_O)) 
= & \, \chi(H^0(X, O_X(D_R + D_O))) - \chi(H^1(X, O_X(D_R + D_O))) \\
= & \, - \log \covol_{Fal}\left(H^0(X, O_X(D_R + D_O))\right) + \log |\disc(O_F)|.
\end{align*}
\end{example}

\begin{thm}[Faltings-Riemann-Roch Theorem, {\cite[Theorem 3]{FaltingsRR}}]
Metrize $O_X$ and the relative dualizing sheaf $\omega_{X}$ by admissible metrics. Then for any admissible line bundle $\CL$, we have
$$\chi(\CL) = \frac{1}{2} \langle \CL, \CL - \omega_X \rangle + \chi(O_X).$$
\end{thm}
\begin{remark}
\label{Remark_Definition_Faltings_height}
Applying Faltings-Riemann-Roch theorem to $\omega_X$, we see that $\chi(\omega_X) = \chi(O_X)$. Normalizing, we arrived at the (unstable) Faltings height of an elliptic curve,
$$h_{Fal}(E) := \frac{1}{d} \chi(\CO_X).$$
\end{remark}

\begin{prop}
\label{Proposition_Covolume_H0X}
Metrize $O_X(D_R + D_O)$ with the canonical admissible metric at each archimedean place. Then
$$
\covol_{Fal}\left(H^0(X, O_X(D_R + D_O))\right) = \frac{|N_{F/\BQ} (\Delta_{E/F})|^{1/4} \cdot |\disc(O_F)|}{\exp\left(d \cdot h_{Faltings}(E)\right)} \cdot \frac{1}{H_{NT}(R)^d}.
$$
Here $\Delta_{E/F}$ is the minimal discrminant of $E/F$, $H_{NT}(R)$ is the (multiplicative) N\'{e}ron-Tate height of $R$ and $h_{Faltings}(E)$ is the (logarithmic) Faltings height of $E$.
\end{prop}
\begin{proof}
From Example \ref{Example_Volume_Euler_Characteristic}, we see that
$$
\log \covol_{Fal}\left(H^0(X, O_X(D_R + D_O))\right) = - \chi(\CO_X(D_R + D_O)) + \log |\disc(O_F)|.
$$
By Faltings-Riemann-Roch theorem, we have
\begin{align*}
\chi(\CO_X(D_R + D_O)) = & \frac{1}{2} \langle \CO_X(D_R + D_O), \CO_X(D_R + D_O) - \omega_X \rangle + \chi(\CO_X) \\
= & \frac{1}{2} \langle \CO_X(D_R + D_O), \CO_X(D_R + D_O) \rangle - \frac{1}{2} \langle \CO_X(D_R + D_O), \omega_X \rangle + \chi(\CO_X).
\end{align*}
By adjunction formula (Proposition \ref{Proposition_Adjunction_Formula}), 
$$\langle O_X(D_R), O_X(D_R) + \omega_X \rangle = 0, \,\,\, \langle O_X(D_O), O_X(D_O) + \omega_X \rangle = 0.$$
Hence by linearity of intersection pairing,
\begin{align*}
& \, \chi(\CO_X(D_R + D_O)) \\
= & \, \frac{1}{2} \langle \CO_X(D_R + D_O), \CO_X(D_R + D_O) \rangle + \frac{1}{2} \langle \CO_X(D_R), \CO_X (D_R) \rangle + \frac{1}{2} \langle \CO_X(D_O), \CO_X (D_O) \rangle + \chi(\CO_X) \\
= & \, \frac{3}{2} \langle \CO_X(D_R), \CO_X(D_R) \rangle + \frac{3}{2}\langle \CO_X(D_O), \CO_X(D_O) \rangle - \frac{1}{2} \langle \CO_X(D_R - D_O), \CO_X(D_R - D_O) \rangle + \chi(\CO_X).
\end{align*}
By Proposition \ref{Proposition_Faltings_Hriljac},
$$\langle \CO_X(D_R - D_O), \CO_X(D_R - D_O) \rangle = -2d \log H_{NT}(R),$$
where $H_{NT}$ is the multiplicative N\'{e}ron-Tate height. 

By Assumption \ref{Assumption_Sym2_EC}, $E$ is semi-stable over $F$. Proposition \ref{Proposition_RobindeJong} then implies that for any $R \in E(F)$,
$$\langle \CO_X(D_R), \CO_X(D_R) \rangle = \langle \CO_X(D_O), \CO_X(D_O) \rangle = - \frac{1}{12} \log |N_{F/\BQ}(\Delta_{E/F})|,$$
where $\Delta_{E/F}$ is the minimal discriminant of the elliptic curve $E/F$. Therefore,
\begin{align*}
& \, \chi(\CO_X(D_R + D_O)) \\
= & \, 3\langle \CO_X(D_O), \CO_X(D_O) \rangle + d \log H_{NT} (R) + \chi(\CO_X) \\
= & \, -\frac{1}{4} \log |N_{F/\BQ} (\Delta_{E/F})| + d \log H_{NT} (R) + d \cdot h_{Faltings}(E),
\end{align*}
by the definition of Faltings height. Hence,
\begin{align*}
& \log \covol_{Fal}\left(H^0(X, O_X(D_R + D_O))\right) \\
= & - \chi(\CO_X(D_R + D_O)) + \log |\disc(O_F)| \\
= & \frac{1}{4} \log |N_{F/\BQ} (\Delta_{E/F})| - d \log H_{NT} (R) - d \cdot  h_{Faltings}(E) + \log |\disc(O_F)|
\end{align*}
and
\begin{align*}
& \covol_{Fal}\left(H^0(X, O_X(D_R + D_O))\right) \\
= & \exp\left(\frac{1}{4} \log |N_{F/\BQ} (\Delta_{E/F})| - d \log H_{NT} (R) - d \cdot  h_{Faltings}(E) + \log |\disc(O_F)|\right) \\
= & \frac{|N_{F/\BQ} (\Delta_{E/F})|^{1/4} \cdot |\disc(O_F)|}{\exp\left(d \cdot h_{Faltings}(E)\right)} \cdot \frac{1}{H_{NT}(R)^d}.
\end{align*}
\end{proof}

\subsection{Some lemmas on $\int_{X(\overline{F_v})} \log |g_v|_v d\mu_v$}
Recall that after Corollary \ref{Corollary_Step_3_Reframe_Cohom_X}, we want to count points of the lattice
$$H^0(X, O_X(D_R + D_O)) \into H^0(X, O_X(D_R + D_O)) \otimes_{\BZ} \BR \cong \prod_{v \in M_{F, \infty}} H^0(E, O_E(R + O)) \otimes_F F_v,$$
lying inside the homogeneous region 
$$\exp\left(\frac{1}{k}\left(\log B - \log H(f(R)) - \log H(f(O))\right) - \frac{1}{d} \log |N_{F/\BQ}(a)|\right) S_{\Omega, R}(1),$$
as $B \to \infty$, for each $R \in E(F)$. We will use geometry of numbers to do this counting in Section 13. To control the error term in the geometry of numbers argument, we will show that the boundary $\partial S_{\Omega, R}(1)$ can be parametrized by finitely many Lipschitz maps with controlled Lipschitz constants in Section 12. In this section, we develop properties of $\int_{X(\overline{F_v})} \log |g_v|_v d\mu_v$, to be used in the next section.

Let $v \in M_{F, \infty}$ be an archimedean place.

\subsubsection{Interpretation of $\int_{X(\overline{F_v})} \log |g_v|_v d\mu_v$ as Mahler measure}

Let $C(\overline{F_v})$ be a smooth projective curve, and $\CL = (L, \|\cdot\|)$ a metrized line bundle of class (S) on $C(\overline{F_v})$.

\begin{defn}[Mahler measure]
Consider a rational function $G: C(\overline{F_v}) \to \BP^1(\overline{F_v})$, with
$$\div(G) = \sum_{\gamma} m_{\gamma} \gamma.$$
Consider a section $s \in H^0(C(\overline{F_v}), L)$, with $\div(s) = \sum_{\eta} n_{\eta} \eta$ disjoint from $\div(G)$.

We define the Mahler measure of $G$ with respect to $\CL$ to be
$$M_{v, \CL, s}(G) := \prod_{\gamma \in \supp(G)} \|s(\gamma)\|^{-m_{\gamma}} \prod_{\eta \in \supp(s)} |G(\eta)|_v^{n_{\eta}}.$$
\end{defn}

\begin{prop}
\label{Proposition_Properties_Mahler_Measure}
The Mahler measure has the following properties:
\begin{enumerate}[label=(\alph*)]
    \item $M_{v, \CL, s}(G)$ does not depend on the section $s$. In fact,
    $$\log M_{v, \CL, s}(G) = \int_{C(\overline{F_v})} \log |G|_v c_1(\CL),$$
    for any section $s$ where $\div(s)$ and $\div(G)$ are disjoint. We may hence denote the Mahler measure as $M_{v, \CL}(G)$.
    \item If $G_1, G_2$ are two rational functions on $C(\overline{F_v})$ and $G_1G_2$ is their pointwise product, then 
    $$M_{v, \CL}(G_1G_2) = M_{v, \CL}(G_1)M_{v, \CL}(G_2).$$
    \item (Pushforward) If $\phi: C'(\overline{F_v}) \to C(\overline{F_v})$ is a morphism, and $G'$ is a rational function on $C'(\overline{F_v})$, then
    $$M_{v, \phi^*\CL} (G') = M_{v, \CL} (\phi_*(G')).$$
    Here $\phi_* := N_{\BC(C(\overline{F_v})) / \BC(\phi^* C'(\overline{F_v}))}$ is the norm map on function fields.
    \item (Pullback) If $\phi: C'(\overline{F_v}) \to C(\overline{F_v})$ is a morphism, and $G$ is a rational function on $C(\overline{F_v})$, then
    $$M_{v, \phi^*\CL} (\phi^*(G)) = M_{v, \CL} (G)^{deg(\phi)}.$$
    Here $\phi^*(G) := G \circ \phi$ is the usual pullback.
\end{enumerate}
\end{prop}
\begin{proof}
For part (a), this is \cite[Proposition B.2]{PineiroSzpiroTucker}; for completeness we sketch the proof here. Let $s$ be a nonzero, meromorphic section of $\CL$, where $\div(s)$ and $\div(G)$ are disjoint. Away from $\div(s)$, 
$$c_1(\CL) = \frac{\partial \overline{\partial}}{\pi i} \log \|s\|.$$ 
For $x \in \div(G) \cup \div(s)$, let $D_{x, \epsilon}$ be a small disk around $x$ of radius $\epsilon > 0$ and $C_{x, \epsilon}$ be a positively oriented circle around $x$ of radius $\epsilon > 0$. We take $\epsilon$ small enough so that these disks are disjoint. Let 
$$Y_{\epsilon} = C(\overline{F_v}) - \bigcup_{x \in \div(G) \cup \div(s)} D_{x, \epsilon},$$ 
and let 
$$\omega = \log |G| \cdot \frac{\overline{\partial}}{\pi i} \log \|s\| + \log \|s\| \cdot \frac{\partial}{\pi i} \log |G|$$
be a 1-form on $Y_{\epsilon}$. As $G$ is meromorphic, $\log |G|$ is harmonic on $Y_{\epsilon}$, hence
\begin{align*}
d \omega = & \log |G| \cdot \frac{\partial \overline{\partial}}{\pi i} \log \|s\| + \log \|s\| \cdot \frac{\overline{\partial} \partial}{\pi i} \log |G| \\
= & \log |G| \cdot c_1(\CL)
\end{align*}
on $Y_{\epsilon}$. By Stokes' theorem,
$$\int_{Y_{\epsilon}} \log |G|_v c_1(\CL) = - \sum_{x \in \div(G) \cup \div(s)} \int_{C_{x, \epsilon}} \omega = - \sum_{\gamma \in \div(G)} \int_{C_{\gamma, \epsilon}} \omega - \sum_{\eta \in \div(s)} \int_{C_{\eta, \epsilon}} \omega$$
(the minus sign comes from the fact that applying Stokes' theorem to the outside of $C_{x, \epsilon}$, so that the orientation is reversed). For $\gamma \in \div(G)$, 
$$\int_{C_{\gamma, \epsilon}} \log |G| \cdot \frac{\overline{\partial}}{\pi i} \log \|s\| = O(\epsilon \log \epsilon),$$
and
\begin{align*}
\int_{C_{\gamma, \epsilon}} \log \|s\| \cdot \frac{\partial}{\pi i} \log |G| = & \, \frac{1}{2 \pi i} \int_{C_{\gamma, \epsilon}} \log \|s\| \partial \log (G \overline{G}) \\
= & \, \frac{1}{2 \pi i} \int_{C_{\gamma, \epsilon}} \log \|s\| \frac{G'(z)}{G(z)} dz \\
= & \, m_{\gamma} \log \|s(\gamma)\|
\end{align*}
by residue theorem. Hence
$$\lim_{\epsilon \to 0} \int_{C_{\gamma, \epsilon}} \omega = m_{\gamma} \log \|s(\gamma)\|.$$
Similarly for $\eta \in \div(s)$,
$$\lim_{\epsilon \to 0} \int_{C_{\eta, \epsilon}} \omega = -n_{\eta} \log |G(\eta)|.$$
Hence,
\begin{align*}
\int_{C(\overline{F_v})} \log |G| c_1(\CL) 
= & \, \lim_{\epsilon \to 0} \int_{Y_{\epsilon}} \log |G| c_1(\CL)  \\
= & \, \lim_{\epsilon \to 0} \left(- \sum_{\gamma \in \div(G)} \int_{C_{\gamma, \epsilon}} \omega - \sum_{\eta \in \div(s)} \int_{C_{\eta, \epsilon}} \omega \right) \\
= & \, - \sum_{\gamma \in \div(G)} m_{\gamma} \log \|s(\gamma)\| + \sum_{\eta \in \div(s)} n_{\eta} \log |G(\eta)| \\
= & \, \log M_{v, \CL, s}(G)
\end{align*}
as desired.

Part (b) follows by taking a section $s$ with $\div(s)$ disjoint from both $\div(G_1)$ and $\div(G_2)$, and the definition of Mahler measure.

For part (c), take a section $s$ of $\CL$ with $\div(s)$ disjoint from $\div(\phi_*(G'))$. Then $\phi^* s$ is a section of $\phi^*\CL$ and $\div(\phi^* s) = \phi^* \div(s)$ has disjoint support from $\div(G')$. By definition,
$$
M_{v, \phi^*\CL} (G')
= \prod_{\gamma' \in \supp(G')} \|s (\phi(\gamma'))\|^{- m_{\gamma'}} \prod_{\eta' \in \supp(\phi^* s)} |G'(\eta'))|^{- n_{\eta'}}.
$$
Note that $\div(\phi_* (G')) = \phi_* \div(G')$, and for $\eta \in C(\overline{F_v})$ with $\phi^*([\eta]) = \sum_{\eta'} n_{\eta'} \eta'$, we have \cite[Exercise 2.10]{SilvermanEllipticCurve}
$$\phi_*(G')(\eta) = \prod_{\phi(\eta') = \eta} G'(\eta')^{n_{\eta'}}.$$
Hence,
$$M_{v, \phi^*\CL}(G') 
= \prod_{\gamma \in \supp(\phi_*(G'))} \|s (\gamma)\|^{- m_{\gamma}} \prod_{\eta \in \supp(s)} |(\phi_* G')(\eta))|^{- n_{\eta}}
= M_{v, \CL} \left(\phi_*(G')\right)$$
as desired.

For part (d), we apply $G' = \phi^*G$ to part (c). Then
$$M_{v, \phi^* \CL} (\phi^*G) = M_{v, \CL}(\phi_*\phi^*G) = M_{v, \CL}(G^{\deg(\phi)}).$$
Using property (b), we see that $M_{v, \CL}(G^{\deg(\phi)}) = M_{v, \CL}(G)^{\deg(\phi)}$. Hence,
$$M_{v, \phi^* \CL} (\phi^*G) = M_{v, \CL}(G)^{\deg(\phi)}$$
as desired.
\end{proof}
\begin{example}
Consider $C = \BP^1$, and $\CL = \mathbf{O(1)}$ on $\BP^1$ with standard metric. If $G: \BP^1 \to \BP^1$ is a polynomial, i.e.
$$G([x, y]) = [a(x - \alpha_1 y) \cdots (x - \alpha_k y), y^k]$$ where $a \neq 0$, we claim that
$$M_{v, \mathbf{O(1)}}(G) = |a|_v \prod_{i=1}^n \max\{1, |\alpha_i|_v\}$$
recovers the classical Mahler measure of a polynomial.

The reasoning is as follows. Note that $G$ has a pole at $[1, 0]$ of order $k$. For small $\epsilon > 0$, one consider the section $s_{\epsilon}$, where
$$s_{\epsilon} ([x_0, x_1]) = \frac{\epsilon x_0 + x_1}{x_i}$$
on the open set $x_i \neq 0$; this is a section with a simple zero at $[1, -\epsilon]$ and no poles. For most $\epsilon$, $\div(s_{\epsilon})$ and $\div(G)$ are disjoint. Moreover, under the standard norm on $\mathbf{O(1)}$,
$$\|s_{\epsilon} ([x_0, x_1]) \| = \frac{|\epsilon x_0 + x_1|_v}{\max\{|x_0|_v, |x_1|_v\}}$$

By definition, we have
\begin{align*}
M_{v, \mathbf{O(1)}, s_{\epsilon}}(G) = & \, \|s_{\epsilon}([1, 0])\|^{k} \cdot \prod_{i=1}^k \|s_{\epsilon}([\alpha_i, 1])\|^{-1} \cdot |G([1, -\epsilon])|_v \\
= & \, |\epsilon^k G([1, -\epsilon])|_v \cdot \prod_{i=1}^k \frac{\max\{1, |\alpha_i|\}}{|1 + \epsilon \alpha_i|_v} \\
\to & \, |a|_v \prod_{i=1}^k \max\{1, |\alpha_i|_v\}
\end{align*}
as $\epsilon \to 0$. Since $M_{v, \mathbf{O(1)}}(G)$ does not depend on the choice of section $s_{\epsilon}$, we have 
$$M_{v, \mathbf{O(1)}}(G) = |a|_v \prod_{i=1}^k \max\{1, |\alpha_i|_v\}$$
as desired.
\end{example}
\begin{cor}
Back to our set up, where $X(\overline{F_v})$ is an elliptic curve, and $f: X(\overline{F_v}) \to \BP^1(\overline{F_v})$ is a morphism of degree $k$. Then
$$\int_{X(\overline{F_v})} \log |g_v| d\mu_v = \frac{1}{k} \log M_{v, f^*\mathbf{O(1)}}(g_v).$$
If we define the rational function $G_v: \BP^1(\overline{F_v}) \to \BP^1(\overline{F_v})$ via the norm map
$$G_v \circ f = N_{\BC(X(\overline{F_v})) / \BC(f)}(g_v),$$
then
$$\int_{X(\overline{F_v})} \log |g_v| d\mu_v = \frac{1}{k} \log M_{v, \mathbf{O(1)}}(G_v).$$
\end{cor}
\begin{proof}
By definition, $d\mu_v = \frac{1}{k} c_1(f^*\mathbf{O(1)})$. Hence, by Proposition \ref{Proposition_Properties_Mahler_Measure}(a), 
$$\int_{X(\overline{F_v})} \log |g_v| d\mu_v = \frac{1}{k} \log M_{v, f^*\mathbf{O(1)}}(g_v).$$
The second equality follows from Proposition \ref{Proposition_Properties_Mahler_Measure}(c).
\end{proof}

\subsubsection{Continuity of $\int_{X(\overline{F_v})} \log |g_v|_v d\mu_v$ and consequences}
\label{Section_H0E_Locally_Free_In_R}

The vector space $H^0(E, O_E(R + O))$ vary nicely as $R$ varies in $E$; they are the fibers of a vector bundle $\CN$ over $E$, defined over $F$. (Rigorously, this follows from the existence of Poincare bundle on $E \times \Pic^2(E)$.) Base change to $\overline{F_v}$ gives us the vector bundle $\CN \times_F \overline{F_v}$ over $E(\overline{F_v})$, whose fiber over $R \in E(\overline{F})$ is $H^0(E(\overline{F_v}), O_E(R + O))$.

Each fiber $H^0(E(\overline{F_v}), O_E(R + O))$ is equipped with the $L^2$-metric $\|\cdot\|_{R, 2}$, defined using the Arakelov-Green function (see Definition \ref{Definition_L2_metric_H0X}; when the point $R$ is clear, we may denote the metric as $\|\cdot\|_2$). The smoothness of Arakelov-Green function implies that this $L^2$-metric glues to a metric on $\CN$. We can thus form the unit sphere bundle $M$ over $E(\overline{F_v})$. As both the base $E(\overline{F_v})$ and the fiber are compact, $M$ is compact.

\begin{lemma}
\label{Lemma_Log_Integral_Uniformly_Bounded}
For $g_{v, R} \neq 0 \in H^0(E(\overline{F_v}), O_E(R + O))$, we have an error bound
$$\left|\int_{X(\overline{F_v})} \log |g_{v, R}|_v d\mu_v - \log \|g_{v, R}\|_{2}\right| = O_{E, v}(1),$$
uniform over $g_{v, R}$, and $R \in E(\overline{F_v})$.
\end{lemma}
\begin{proof} 
For $g_{v, R} \in H^0(E(\overline{F_v}), O_E(R + O))$ with $\|g_{v, R}\|_2 = 1$, note that the map
$$g_{v, R} \to \int_{X(\overline{F_v})} \log |g_{v, R}|_v d\mu_v$$
is continuous on the compact space $M$ (This is clear using the expression as Mahler measure). Hence $\int_{X(\overline{F_v})} \log |g_{v, R}|_v d\mu_v$ is bounded uniformly over $R \in E(\overline{F_v})$ and $\|g_{v, R}\|_2 = 1$. 

For arbitrary $g_{v, R} \neq 0 \in H^0(E(\overline{F_v}), O_E(R + O))$, by normalizing $g_{v, R}$ to have norm 1 we have
$$\left|\int_{X(\overline{F_v})} \log |g_{v, R}|_v d\mu_v - \log \|g_{v, R}\|_{2}\right| = \left|\int_{X(\overline{F_v})} \log \frac{|g_{v, R}|_v}{\|g_{v, R}\|_{2}} d\mu_v \right| = O_{E, v}(1),$$
as desired.
\end{proof}

\subsection{Lipschitz parametrizability of $\partial S_{\Omega, R}(1)$}
In this section, we will show that the $\partial S_{\Omega, R}(1)$ is Lipschitz parametrizable uniformly over $R \in E(F)$ (Proposition \ref{Proposition_Uniformity_R}). This will be used to control the error term in geometry of numbers argument in next section.

We start with the definition of Lipschitz parametrizability, and its basic properties.
\begin{defn}[Lipschitz parametrizability]
A set $S \subset \BR^n$ is \textbf{Lipschitz parametrizable of codimension $k$}, if there are $M$ maps $\phi_1, \cdots, \phi_M: [0, 1]^{n-k} \to \BR^n$ satisfying the Lipschitz condition
$$\|\phi_i(\mathbf{x}) - \phi_i(\mathbf{y})\|_2 \leq L\|\mathbf{x}-\mathbf{y}\|_2$$
for each $i = 1, \cdots, M$, such that $S$ is covered by the image of $\phi_i$.
\end{defn}

\begin{defn}[Uniform Lipschitz parametrizability]
Let $D$ be a set, and suppose we have a family of sets $S_R \subset \BR^n$ parametrized by $R \in D$. 

We say that the family is \textbf{uniformly Lipschitz parametrizable of codimension $k$ over $D$}, if each $S_R$ is Lipschitz parametrizable of codimension $k$, and the constants $M_R$ (number of parametrizations), $L_R$ (Lipschitz constant) for each $S_R$ can be chosen independent of $R$.
\end{defn}

Our main tool of showing (uniform) Lipschitz parametrizability is Gromov's algebraic lemma \cite{Gromov}; we will use the version stated in \cite{Burguet}. This lemma was first used in \cite{Barroero_Widmer} in similar problems.

Recall that a semi-algebraic set of $\BR^n$ is a finite union of sets defined by polynomial equalities or inequalities. The degree of a semi-algebraic set is the smallest sum of degree of polynomials appearing in a complete description of the set.

\begin{thm}[Gromov's algebraic lemma, {\cite[Theorem 1]{Burguet}}]
For all integers $r \ge 1$, $m \ge 0$, $E \ge 0$, there exists $M(m, r, E) < \infty$ with the following properties.

For any compact semialgebraic subset $A \subset [0, 1]^m$, of dimension $n$, and degree at most $E$, there exists an integer $N \leq M(m, r, E)$ and $C^r$-maps $\phi_1, \cdots, \phi_N: [0,1]^n \to [0, 1]^m$ such that 
\begin{itemize}
    \item $\bigcup_{i=1}^N \phi_i ([0, 1]^n) = A$, and
    \item $\|\phi_{i}\|_{r} := \max_{\beta: |\beta| \leq r} \|\partial^{\beta} \phi_i\| \leq 1$.
\end{itemize}
\end{thm}

\begin{cor}
\label{Corollary_Gromov_Lemma}
Let $\BF = \BR$ or $\BC$. Suppose $A_R \subset \BF^m$ is a family of compact semialgebraic subsets parametrized by $R \in D$, such that 
\begin{itemize}
    \item Each $A_R$ lies inside a ball of diameter $l$, independent of $R \in D$.
    \item Each $A_R$ is of dimension $n$, with degree at most $E$, both of which independent of $R \in D$.
\end{itemize}
Then $A_R$ is uniformly Lipschitz parametrizable over $R \in D$.
\end{cor}
\begin{proof}
Translating if necessary, suppose $A_R \subset [0, l]^d$. Apply Gromov's algebraic lemma with $r = 1$ to $\frac{1}{l} A_R \subset [0, 1]^d$, we find $\phi_1, \cdots, \phi_N: [0,1]^n \to \frac{1}{l}A_R$ with first derivative bounded by 1; hence $\phi_i$ are Lipschitz with Lipschitz constant $\sqrt{d}$. With our assumptions, $N$ does not depend on $R \in D$.

Consider the maps $l \phi_1, \cdots, l \phi_N: [0,1]^n \to A_R$. Then there are $N$ Lipschitz maps of Lipschitz constants $l\sqrt{d}$ that covers $A_R$ for each $R \in D$. Moreover, the number of maps/Lipschitz constants are independent of $R \in D$. Hence we proved the desired uniform Lipschitz parametrizability.
\end{proof}

Recall Definition \ref{Definition_S_Omega_R}, where for $R \in E(F)$ we defined
\begin{align*}
S_{\Omega, R}(1) = &  \left\{(g_{v, R}) \in \prod_{v \in M_{F, \infty}} H^0(E, O_E(R + O)) \otimes_F F_v: \left(\int_{X(\overline{F_v})} \log |g_{v, R}|_v d\mu_v\right) \in \Omega + (-\infty, 0]\delta\right\} \\
= & \left\{(g_{v, R}) \in \prod_{v \in M_{F, \infty}} H^0(E(F_v), O_E(R + O)): \left(\int_{X(\overline{F_v})} \log |g_{v, R}|_v d\mu_v\right) \in \Omega + (-\infty, 0]\delta\right\} 
\end{align*}

The main result in this section is as follows.
\begin{prop}
\label{Proposition_Uniformity_R}
For each $R \in E(F)$, $\partial S_{\Omega, R}(1)$ can be covered by $W$ maps $\phi: [0, 1]^{2d-1} \to \prod_{v \in M_{F, \infty}} H^0(E(F_v), O_E(R + O))$ satisfying 
$$\|\phi(\mathbf{x_1}) - \phi(\mathbf{x_2})\|_2 \leq L\|\mathbf{x_1} - \mathbf{x_2}\|_2.$$
Moreover, we can take $W, L = O_{E, f, F}(1)$ independent of $R$. In other words, $\partial S_{\Omega, R}(1)$ is uniformly Lipschitz parametrizable of codimension 1 over $R \in E(F)$.
\end{prop}
The rest of the section is devoted to proving this result.
\begin{remark}
Strictly speaking, we need to fix an isomorphism of vector spaces
$$\prod_{v \in M_{F, \infty}} H^0(E(F_v), O_E(R + O)) \cong \BR^{2d}$$
before discussing Lipschitz parametrizability. We will fix such an isomorphism in the course of proof.
\end{remark}
By definition,
\begin{align*}
&\partial S_{\Omega, R}(1) \\
\subset & \left\{(g_{v, R}) \in \prod_{v \in M_{F, \infty}} H^0(E(F_v), O_E(R + O)): \left(\int_{X(\overline{F_v})} \log |g_{v, R}|_v d\mu_v\right) \in \partial\left(\Omega + (-\infty, 0] \delta\right)\right\} \\
= & \left\{(g_{v, R}) \in \prod_{v \in M_{F, \infty}} H^0(E(F_v), O_E(R + O)): \left(\int_{X(\overline{F_v})} \log |g_{v, R}|_v d\mu_v\right) \in \Omega \cup \left(\partial \Omega + (-\infty, 0] \delta\right)\right\} \\
= & \left\{(g_{v, R}) \in \prod_{v \in M_{F, \infty}} H^0(E(F_v), O_E(R + O)): \left(\int_{X(\overline{F_v})} \log |g_{v, R}|_v d\mu_v\right) \in \Omega \right\} \\
& \bigcup \left\{(g_{v, R}) \in \prod_{v \in M_{F, \infty}} H^0(E(F_v), O_E(R + O)): \left(\int_{X(\overline{F_v})} \log |g_{v, R}|_v d\mu_v\right) \in \partial \Omega + (-\infty, 0] \delta \right\}
\end{align*}
Let
$$T_{R, 1} := \left\{(g_{v, R}) \in \prod_{v \in M_{F, \infty}} H^0(E(F_v), O_E(R + O)): \left(\int_{X(\overline{F_v})} \log |g_{v, R}|_v d\mu_v\right) \in \Omega \right\},$$
$$T_{R, 2} := \left\{(g_{v, R}) \in \prod_{v \in M_{F, \infty}} H^0(E(F_v), O_E(R + O)): \left(\int_{X(\overline{F_v})} \log |g_{v, R}|_v d\mu_v\right) \in \partial \Omega + (-\infty, 0] \delta \right\}.$$
We will focus on proving the Lipschitz parametrizability of $T_{R, 1}$. The second set can be handled similarly, and we will sketch the proof at the end of this section.

We will now fix an isomorphism of vector spaces
$$\prod_{v \in M_{F, \infty}} H^0(E(F_v), O_E(R + O)) \cong \BR^{2d}.$$
For each $v \in M_{F, \infty}$, locally around $R_v \in E(F_v)$ we consider a basis $g_{v, R_v, 1}, g_{v, R_v, 2}$ of $H^0(E(F_v), O_E(R_v + O))$ that varies smoothly with $R_v \in E(F_v)$ (Section \ref{Section_H0E_Locally_Free_In_R}). For $g_{v, R_v} \in H^0(E(F_v), O_E(R_v + O))$, we may then write
$$g_{v, R_v} = a_v g_{v, R_v, 1} + b_v g_{v, R_v, 2}$$
for some $a_v, b_v \in F_v$. This gives us an isomorphism of vector spaces
$$H^0(E(F_v), O_E(R_v + O)) \cong F_v^2$$
defined by $g_{v, R_v} \to (a_v, b_v)$. Since $E(F_v)$ is compact, we only need to consider a finite open cover of $E(F_v)$, each of which has coordinates $g_{v, R_v, 1}, g_{v, R_v, 2}$ for $H^0(E(F_v), O_E(R_v + O))$ locally. 

In the following, we use $|\cdot|$ to represent the Euclidean norm on $F_v$, and $|\cdot|_v$ to represent the normalized absolute value on $F_v$. In particular when $F_v \cong \BC$, $|\cdot|$ would mean the Euclidean norm on $\BC$, while $|\cdot|_v$ would mean Euclidean norm squared.

\begin{prop}
\label{Proposition_TRd1_Lipschitz_parametrizable_Implies_TR1}
To show that $T_{R, 1}$ is uniformly Lipschitz parametrizable of codimension 1 over $R \in E(F)$, it suffices to show that for $\delta > 0$,
$$T_{(R_v), \delta, 1}' := \left\{((a_v, b_v)) \in \prod_{v \in M_{F, \infty}} F_v^2: 
\begin{array}{l}
\prod_{v \in M_{F, \infty}} M_{v, \mathbf{O(1)}}\left(N_{v, f}(a_v g_{v, R_v, 1} + b_vg_{v, R_v, 2})\right) = 1, \\
\delta \leq |a_v|^2 + |b_v|^2 \leq \delta^{-1} \text{ for each $v \in M_{F, \infty}$.}
\end{array}
\right\}$$
is uniformly Lipschitz parametrizable of codimension 1 over $(R_v) \in \prod_{v \in M_{F, \infty}} E(F_v)$. Here $N_{v, f} := N_{\BC(E(\overline{F_v})) / f^*\BC(\BP^1(\overline{F_v}))}$ is the norm map on function fields, induced by $f: E \to \BP^1$.
\end{prop}
\begin{proof}
Using $(a_v,b_v)$ coordinates, the Lipschitz parametrizability of $T_{R, 1}$ is equivalent to that of 
\begin{equation*}
\left\{((a_v, b_v)) \in \prod_{v \in M_{F, \infty}} F_v^2: \left(\int_{X(\overline{F_v})} \log |a_v g_{v, R, 1} + b_v g_{v, R, 2}|_v d\mu_v\right) \in \Omega\right\}
\end{equation*}
for $R \in E(F)$, embedded diagonally inside $\prod_{v \in M_{F, \infty}} E(F_v)$. We will show more generally that
\begin{equation}
\label{Set_Arakelov_Coeff}
\left\{((a_v, b_v)) \in \prod_{v \in M_{F, \infty}} F_v^2: \left(\int_{X(\overline{F_v})} \log |a_v g_{v, R_v, 1} + b_v g_{v, R_v, 2}|_v d\mu_v\right) \in \Omega\right\}
\end{equation}
is uniformly Lipschitz parametrizable over $(R_v) \in \prod_{v \in M_{F, \infty}} E(F_v)$.

Hence consider $R_v \in E(F_v)$, and let 
$$g_{v, R_v} = a_v g_{v, R_v, 1} + b_v g_{v, R_v, 2}$$ for some $a_v, b_v \in F_v$. One has
$$\frac{\|a_v g_{v, R_v, 1} + b_v g_{v, R_v, 2}\|_2}{(|a_v|^2 + |b_v|^2)^{1/2}} \asymp_{R_v} 1$$
for each $R_v \in E(F_v)$, since norms on Euclidean space are equivalent. Using compactness of $E(F_v)$ and Lemma \ref{Lemma_Log_Integral_Uniformly_Bounded}, we have
\begin{equation}
\label{Equation_Log_Integral_Uniformly_Bounded}
\left|\int_{X(\overline{F_v})} \log |a_v g_{v, R_v, 1} + b_v g_{v, R_v, 2}|_v d\mu_v - \frac{1}{2}\log(|a_v|^2 + |b_v|^2)\right| = O_{E, v}(1)
\end{equation}
independent of $R_v \in E(F_v)$.

Note that 
\begin{itemize}
    \item By Equation \ref{Equation_Log_Integral_Uniformly_Bounded} and the compactness of $\Omega$, we see that $((a_v, b_v))$ lies in the set \ref{Set_Arakelov_Coeff} would imply that 
    $$\delta \leq |a_v|^2 + |b_v|^2 \leq \delta^{-1},$$
    for some $\delta > 0$, uniformly over $R_v \in E(F_v)$.
    \item By definition, $\Omega$ has sum of coordinates equal 0. So $((a_v, b_v))$ lies in the set \ref{Set_Arakelov_Coeff} would imply that 
    $$\sum_{v \in M_{F, \infty}} \int_{X(\overline{F_v})} \log |a_v g_{v, R_v, 1} + b_v g_{v, R_v, 2}|_v d\mu_v = 0.$$
    By property (a) and (c) of Mahler measure (Proposition \ref{Proposition_Properties_Mahler_Measure}),
    \begin{align*}
    \int_{X(\overline{F_v})}\log |a_v g_{v, R_v, 1} + b_v g_{v, R_v, 2}|_v d\mu_v = & \, \log M_{v, f^*\mathbf{O(1)}} \left(a_v g_{v, R_v, 1} + b_v g_{v, R_v, 2}\right)\\
    = & \, \log M_{v, \mathbf{O(1)}} \left(N_{v, f} (a_v g_{v, R_v, 1} + b_v g_{v, R_v, 2})\right).
    \end{align*}
    Hence
    $$\sum_{v \in M_{F, \infty}} \int_{X(\overline{F_v})} \log |a_v g_{v, R_v, 1} + b_v g_{v, R_v, 2}|_v d\mu_v = 0$$
    is equivalent to
    $$\sum_{v \in M_{F, \infty}} \log M_{v, \mathbf{O(1)}} \left(N_{v, f} (a_v g_{v, R_v, 1} + b_v g_{v, R_v, 2})\right) = 0.$$
    In other words,
    $$\prod_{v \in M_{F, \infty}} M_{v, \mathbf{O(1)}} \left(N_{v, f} (a_v g_{v, R_v, 1} + b_v g_{v, R_v, 2})\right) = 1.$$
\end{itemize}
Therefore the set \ref{Set_Arakelov_Coeff} is a subset of $T_{(R_v), \delta, 1}'$, and thus the Lipschitz parametrizability of $T_{(R_v), \delta, 1}'$ would imply that of the set \ref{Set_Arakelov_Coeff}, hence that of $T_{R, 1}$.
\end{proof}

Let $G_{v, R_v}$ be the rational function on $\BP^1$ defined by the norm map of function fields
$$G_{v, R_v} \circ f = N_{v, f}(a_vg_{v, R_v, 1} + b_v g_{v, R_v, 2}).$$
Since $g_{v, R_v, 1}, g_{v, R_v, 2}$ is of degree $\leq 2$ and has at worst a pole at $R_v$ and $O$, we see that $G_{v, R_v}$ is also degree $\leq 2$, with at worst a pole at $f(R_v)$ and $f(O)$. Moreover, as $f$ has degree $k$, we can write
$$G_{v, R_v} = a_v^k G_{v, R_v, 0} + a_v^{k-1} b_v G_{v, R_v, 1} + \cdots + b_v^k G_{v, R_v, k}$$
where $G_{v, R_v, 0}, \cdots, G_{v, R_v, k}$ are degree $\leq 2$ rational functions on $\BP^1$, with at worst a pole at $f(R_v)$ and $f(O)$, that varies holomorphically with $R_v$. We can thus write
$$G_{v, R_v} = \frac{E_{v, R_v}}{(z - f(R_v))(z - f(O))}$$
for some polynomials $E_{v, R_v}$ of degree $\leq 2$, where
$$E_{v, R_v}(z) = H_{v, R_v, 0}(a_v, b_v) + H_{v, R_v, 1}(a_v, b_v) z + H_{v, R_v, 2}(a_v, b_v) z^2$$
has coefficients $H_{v, R_v, 0}, H_{v, R_v, 1}, H_{v, R_v, 2}$, each of which a homogeneous degree $k$ polynomial in $F_v[a_v, b_v]$, whose coefficient varies smoothly in $R_v$. Shrinking the neighborhood for local basis if needed, we can assume that coefficients of $H_{v, R_v, 0}, H_{v, R_v, 1}, H_{v, R_v, 2} \ll 1$ uniformly over $R_v \in E(F_v)$, since $E(F_v)$ is compact. 

With this framing,
\begin{align}
\label{Equation_T_R_d_1_Alternate_Form}
T_{(R_v), \delta, 1}' = & \, \left\{((a_v, b_v)) \in \prod_{v \in M_{F, \infty}} F_v^2: 
\begin{array}{l}
\prod_{v \in M_{F, \infty}} M_{v, \mathbf{O(1)}}\left(G_{v, R_v})\right) = 1, \\
\delta \leq |a_v|^2 + |b_v|^2 \leq \delta^{-1} \text{ for each $v \in M_{F, \infty}$.}
\end{array}
\right\}
\nonumber\\
= & \, \left\{((a_v, b_v)) \in \prod_{v \in M_{F, \infty}} F_v^2: 
\begin{array}{l}
\prod_{v \in M_{F, \infty}} M_{v, \mathbf{O(1)}}\left(E_{v, R_v})\right) \\ 
= \prod_{v \in M_{F, \infty}} M_{v, \mathbf{O(1)}}\left((z - f(R_v))(z - f(O))\right), \\
\delta \leq |a_v|^2 + |b_v|^2 \leq \delta^{-1} \text{ for each $v \in M_{F, \infty}$.}
\end{array}
\right\},
\end{align}
by the multiplicativity of Mahler measure (Proposition \ref{Proposition_Properties_Mahler_Measure}(b)).

\begin{lemma}
\label{Lemma_Mahler_Measure_Roughly_Constant_Size}
If $((a_v, b_v)) \in T_{(R_v), \delta, 1}'$, then for each $v \in M_{F, \infty}$ we have
$$M_{v, \mathbf{O(1)}}(E_{v, R_v}) \asymp_{k, \delta} 1,$$
independent of $R_v \in E(F_v)$. 
\end{lemma}
\begin{proof}
Suppose 
$$E_{v, R_v}(z) = H_{v, R_v, 0}(a_v, b_v) + H_{v, R_v, 1}(a_v, b_v) z + H_{v, R_v, 2}(a_v, b_v) z^2$$
has two roots $\alpha$ and $\beta$. By definition of Mahler measure, we always have 
\begin{align*}
M_{v, \mathbf{O(1)}}(E_{v, R_v}) = & \, |H_{v, R_v, 2}(a_v, b_v)|_v \max\{1, |\alpha|_v\} \max\{1, |\beta|_v\} \\
\leq & \, |H_{v, R_v, 2}(a_v, b_v)|_v |\alpha|_v |\beta|_v \\
\leq & \, |H_{v, R_v, 0}(a_v, b_v)|_v.
\end{align*}
Since $H_{v, R_v, 0}(a_v, b_v)$ is homogeneous of degree $k$, $|a_v|^2 + |b_v|^2 \ll_{\delta} 1$, and the coefficients of $H_{v, R_v, 0}$ are $\ll 1$, we have 
$$M_{v, \mathbf{O(1)}}(E_{v, R_v}) \leq |H_{v, R_v, 0}(a_v, b_v)|_v \ll_{k, \delta} 1,$$
uniformly over $R_v \in E(F_v)$.

For any given place $v \in M_{F, \infty}$, apply the forementioned upper bound of $M_{w, \mathbf{O(1)}}(E_{w, R_v})$ for all $w \neq v$ to the condition 
$$\prod_{v \in M_{F, \infty}} M_{v, \mathbf{O(1)}}\left(N_f(a_v g_{v, R_v, 1} + b_vg_{v, R_v, 2})\right) = 1,$$
we see that the other inequality $M_{v, \mathbf{O(1)}}(E_{v, R_v}) \gg_{k, \delta} 1$ is also true.
\end{proof}

\begin{prop}
\label{Proposition_Mahler_measure_as_coordinate}
For each archimedean place $v \in M_{F, \infty}$, the set 
$$T_v := \left\{(a_v, b_v) \in F_v^2: \delta \leq |a_v|^2 + |b_v|^2 \leq \delta^{-1}, M_{v, \mathbf{O(1)}}(E_{v, R_v}) \asymp_{k, \delta} 1\right\}$$ 
can be covered by the image of $\ll_{k, \delta} 1$ maps $\phi$, which satisfies:
\begin{itemize}
    \item Domain of $\phi$ is a compact semi-algebraic set with degree $\ll_{k, \delta} 1$, lying inside a ball of diameter $\ll_{k, \delta} 1$, with ambient space having dimension over $\BR$ $\ll 1$.
    \item The map $\phi$ is a polynomial map, which is Lipschitz on the domain with Lipschitz constant $\ll 1$.
    \item Each point $(a_v, b_v)$ lies in the image of some map $\phi$, say $\phi(\vec{x}) = (a_v, b_v)$, such that 
    $$M_{v, \mathbf{O(1)}}(E_{v, R_v}) \circ \phi (\vec{x}) = |M_{\phi}(\vec{x})|$$ 
    for a polynomial map $M_{\phi}$ of degree $\ll_k 1$.
\end{itemize}
Moreover, all the constants above are independent of $R_v \in E(F_v)$.
\end{prop}
\begin{proof}
We will stratify 
$$\{(a_v, b_v) \in F_v^2: \delta \leq |a_v|^2 + |b_v|^2 \leq \delta^{-1}\}$$ 
for each $R_v \in E(F_v)$, based on the degree of $E_{v, R_v}$, and whether the two roots of $E_{v, R_v}(a_v, b_v)$ has norm less than or greater than 1. 

\textbf{Case 1: $\deg(E_{v, R_v}) = 0$.} In this case, $E_{v, R_v}(z) \equiv H_{v, R_v, 0}(a_v, b_v)$, and $M_v(E_{v, R_v}) = |H_{v, R_v, 0}(a_v, b_v)|_v$. This can be covered by $\phi \equiv id$ on the domain
$$\{a_v, b_v) \in F_v^2: \delta \leq |a_v|^2 + |b_v|^2 \leq \delta^{-1}, |H_{v, R_v, 0}(a_v, b_v)|_v \asymp_{k, \delta} 1\}.$$
We verify this map satisfies the conditions we need.
\begin{itemize}
    \item Recall that $H_{v, R_v, 0}$ is a homogeneous polynomial of degree $k$; varying $R_v$ affects the coefficients but not the degree. Squaring if necessary, the condition $|H_{v, R_v, 0}(a_v, b_v)|_v \asymp_{k, \delta} 1$ can be defined by two inequalities of a polynomial of degree $2k$. Hence the domain of $\phi$ is semi-algebraic with degree $\ll_{k} 1$.
    \item The domain lies inside $\delta \leq |a_v|^2 + |b_v|^2 \leq \delta^{-1}$, so it lies inside a ball of diameter $\ll_{\delta} 1$.
    \item The domain lies inside $F_v^2$, with dimension over $\BR$ at most 4.
    \item $\phi = id$ is clearly a polynomial map, and is Lipschitz with Lipschitz constant 1.
    \item $(a_v, b_v)$ covered in this case satisfies
    $$M_v(E_{v, R_v}) = |H_{v, R_v, 0}(a_v, b_v)|_v,$$
    so we may take $M_{\phi} = H_{v, R_v, 0}$ (when $F_v \cong \BR$), or $H_{v, R_v, 0}^2$ (when $F_v \cong \BC$). In either case, $M_{\phi}$ has degree $\ll_k 1$.
\end{itemize}

\textbf{Case 2: $\deg(E_{v, R_v}) = 1$.} In this case, 
$$E_{v, R_v}(z) = H_{v, R_v, 0}(a_v, b_v) + H_{v, R_v, 1}(a_v, b_v)z$$
has a unique root $\alpha = - \frac{H_{v, R_v, 0}(a_v, b_v)}{H_{v, R_v, 1}(a_v, b_v)}$. In this case, Mahler measure equals
$$M_{v, \mathbf{O(1)}}(E_{v, R_v}) = |H_{v, R_v, 1}(a_v, b_v)|_v \max\{1, |\alpha|_v\}.$$

\textbf{Case 2(a): $\deg(E_{v, R_v}) = 1$, and the root $\alpha$ has norm $\leq 1$.} In this case,
$$M_{v, \mathbf{O(1)}}(E_{v, R_v}) = |H_{v, R_v, 1}(a_v, b_v)|_v.$$
As in case 1, this can be covered by $\phi \equiv id$ on the domain
$$\{a_v, b_v) \in F_v^2: \delta \leq |a_v|^2 + |b_v|^2 \leq \delta^{-1}, |H_{v, R_v, 1}(a_v, b_v)|_v \asymp_{k, \delta} 1\}.$$

\textbf{Case 2(b): $\deg(E_{v, R_v}) = 1$, and the root $\alpha$ has norm $\geq 1$.} In this case,
$$M_{v, \mathbf{O(1)}}(E_{v, R_v}) = |H_{v, R_v, 0}(a_v, b_v)|_v.$$
As in case 1, this can be covered by $\phi \equiv id$ on the domain
$$\{a_v, b_v) \in F_v^2: \delta \leq |a_v|^2 + |b_v|^2 \leq \delta^{-1}, |H_{v, R_v, 0}(a_v, b_v)|_v \asymp_{k, \delta} 1\}.$$

\textbf{Case 3: $\deg(E_{v, R_v}) = 2$.} In that case, 
$$E_{v, R_v}(z) = H_{v, R_v, 0}(a_v, b_v) + H_{v, R_v, 1}(a_v, b_v)z + H_{v, R_v, 2}(a_v, b_v) z^2$$
has two roots $\alpha_1, \alpha_2$. In this case, Mahler measure equals
$$M_{v, \mathbf{O(1)}}(E_{v, R_v}) = |H_{v, R_v, 2}(a_v, b_v)|_v \max\{1, |\alpha_1|_v\} \max\{1, |\alpha_2|_v\}.$$

\textbf{Case 3(a): $\deg(E_{v, R_v}) = 2$, and both $\alpha_1, \alpha_2$ has norm $\leq 1$.} In this case,
$$M_{v, \mathbf{O(1)}}(E_{v, R_v}) = |H_{v, R_v, 2}(a_v, b_v)|_v.$$ 
As in case 1, this can be covered by $\phi \equiv id$ on the domain
$$\{a_v, b_v) \in F_v^2: \delta \leq |a_v|^2 + |b_v|^2 \leq \delta^{-1}, |H_{v, R_v, 2}(a_v, b_v)|_v \asymp_{k, \delta} 1\}.$$

\textbf{Case 3(b): $\deg(E_{v, R_v}) = 2$, and both $\alpha_1, \alpha_2$ has norm $\geq 1$.} In this case,
$$M_{v, \mathbf{O(1)}}(E_{v, R_v}) = |H_{v, R_v, 0}(a_v, b_v)|_v.$$ 
As in case 1, this can be covered by $\phi \equiv id$ on the domain
$$\{a_v, b_v) \in F_v^2: \delta \leq |a_v|^2 + |b_v|^2 \leq \delta^{-1}, |H_{v, R_v, 0}(a_v, b_v)|_v \asymp_{k, \delta} 1\}.$$

\textbf{Case 3(c): $\deg(E_{v, R_v}) = 2$, $|\alpha_1|_v \ge 1 \ge |\alpha_2|_v$.}
In this case,
$$M_{v, \mathbf{O(1)}}(E_{v, R_v}) = |H_{v, R_v, 2}(a_v, b_v) \alpha_1|_v.$$
Let $u = H_{v, R_v, 2}(a_v, b_v) \alpha_1$. Then
$$|u|_v = M_{v, \mathbf{O(1)}}(E_{v, R_v}) \asymp_{k, \delta} 1,$$
and as $\alpha_1$ is a root of $E_{v, R_v}$, $u$ satisfies the quadratic equation
$$Q_{R_v}(a_v, b_v, u) := H_{v, R_v, 0}(a_v, b_v) H_{v, R_v, 2}(a_v, b_v) + H_{v, R_v, 1}(a_v, b_v) u + u^2 = 0.$$
Define $\phi(a_v, b_v, u) := (a_v, b_v)$ from the domain
$$\left\{
(a_v, b_v, u) \in F_v^2: \delta \leq |a_v|^2 + |b_v|^2 \leq \delta^{-1}, |u|_v \asymp_{k, \delta} 1, Q_{R_v}(a_v, b_v, u) = 0
\right\}.$$
We verify this map satisfies the conditions we need.
\begin{itemize}
    \item Recall that $Q_{R_v}$ is a polynomial of degree $2k$; varying $R_v$ affects the coefficients but not the degree. So the domain of $\phi$ is semi-algebraic with degree $\ll_{k} 1$.
    \item The domain lies inside $\delta \leq |a_v|^2 + |b_v|^2 \leq \delta^{-1}, |u| \asymp_{k, \delta} 1$, so it lies inside a ball of diameter $\ll_{k, \delta} 1$.
    \item The domain lies inside $F_v^3$, with dimension over $\BR$ at most 6.
    \item $\phi(a_v, b_v, u) = (a_v, b_v)$ is clearly a polynomial map, and is Lipschitz with Lipschitz constant 1.
    \item $(a_v, b_v)$ covered in this case satisfies
    $$M_v(E_{v, R_v}) = |u|_v,$$
    so we may take $M_{\phi} = u$ (when $F_v \cong \BR)$ or $u^2$ (when $F_v \cong \BC$). In either case, $M_{\phi}$ has degree $\ll_k 1$.
\end{itemize}
\end{proof}

\begin{prop}
\label{Proposition_TR1_Lipschitz_parametrizable}
$T_{(R_v), \delta, 1}'$ is uniformly Lipschitz parametrizable of codimension 1 over $(R_v) \in \prod_{v \in M_{F, \infty}} E(F_v)$. Hence by Proposition \ref{Proposition_TRd1_Lipschitz_parametrizable_Implies_TR1}, $T_{R, 1}$ is uniformly Lipschitz parametrizable over $R \in E(F)$.
\end{prop}
\begin{proof}
Recall from equation \ref{Equation_T_R_d_1_Alternate_Form} that
\begin{equation*}
T_{(R_v), \delta, 1}' = \left\{((a_v, b_v)) \in \prod_{v \in M_{F, \infty}} F_v^2: 
\begin{array}{l}
\prod_{v \in M_{F, \infty}} M_{v, \mathbf{O(1)}}\left(E_{v, R_v})\right) \\ 
= \prod_{v \in M_{F, \infty}} M_{v, \mathbf{O(1)}}\left((z - f(R_v))(z - f(O))\right), \\
\delta \leq |a_v|^2 + |b_v|^2 \leq \delta^{-1} \text{ for each $v \in M_{F, \infty}$.}
\end{array}
\right\}.
\end{equation*}
By Lemma \ref{Lemma_Mahler_Measure_Roughly_Constant_Size} and Proposition \ref{Proposition_Mahler_measure_as_coordinate}, one can cover $T_{(R_v), \delta, 1}'$ with $\ll_{k, \delta, F} 1$ maps of the shape $\Phi := \prod_{v \in M_{F, \infty}} \phi_v$, that satisfies:
\begin{itemize}
    \item Domain of $\Phi$ is a compact semi-algebraic set with degree $\ll_{k, \delta, F} 1$, lying inside a ball of diameter $\ll_{k, \delta, F} 1$, with ambient space having dimension over $\BR \ll_F 1$.
    \item The map $\phi$ is a polynomial map, and is Lipschitz with Lipschitz constant $\ll_F 1$.
    \item Each point $((a_v, b_v))$ lies in the image of some map $\Phi$, say $\Phi(\vec{x}) = ((a_v, b_v))$, such that 
    $$\prod_{v \in M_{F, \infty}} M_{v, \mathbf{O(1)}}(E_{v, R_v}) \circ \Phi(\vec{x}) = \left|\prod_{v \in M_{F, \infty}} M_{\phi_v}(\vec{x})\right|$$ 
    for a polynomial map $\prod_{v \in M_{F, \infty}} M_{\phi_v}$ of degree $\ll_{k, F} 1$.
    \item Moreover, the constants here are all independent of $(R_v) \in \prod_{v \in M_{F, \infty}} E(F_v)$.
\end{itemize}
This is a set of one dimension higher than what we want, because so far we only utilized the constraint $M_{v, \mathbf{O(1)}}\left(E_{v, R_v})\right) \asymp_{k, \delta} 1$. We now try to carve out the condition
$$
\prod_{v \in M_{F, \infty}} M_{v, \mathbf{O(1)}}\left(E_{v, R_v})\right)
= \prod_{v \in M_{F, \infty}} M_{v, \mathbf{O(1)}}\left((z - f(R_v))(z - f(O))\right)
$$
as well. For a parametrization $\Phi = \prod_{v \in M_{F, \infty}} \phi_v$, the condition becomes
$$
\left|\prod_{v \in M_{F, \infty}} M_{\phi_v}(\vec{x})\right|
= \prod_{v \in M_{F, \infty}} M_{v, \mathbf{O(1)}}\left((z - f(R_v))(z - f(O))\right).
$$
It now suffices to show that
$$\left\{
\vec{x} \in \prod_{v \in M_{F, \infty}} Domain(\phi_v):
\begin{array}{l}
\left|\prod_{v \in M_{F, \infty}} M_{\phi_v}(\vec{x})\right|
= \prod_{v \in M_{F, \infty}} M_{v, \mathbf{O(1)}}\left((z - f(R_v))(z - f(O))\right), \\
\delta \leq |\phi_v(\vec{x})|^2 \leq \delta^{-1} \text{ for each $v \in M_{F, \infty}$.}
\end{array}
\right\}
$$
is uniformly Lipschitz parametrizable of codimension 1 over $(R_v) \in \prod_{v \in M_{F, \infty}} E(F_v)$. Since 
$$\left|\prod_{v \in M_{F, \infty}} M_{\phi_v}(\vec{x})\right|
= \prod_{v \in M_{F, \infty}} M_{v, \mathbf{O(1)}}\left((z - f(R_v))(z - f(O))\right)$$
is an algebraic condition on $\vec{x}$ (after squaring), and that $R_v$ only ever affects the coefficients (but not the degree) of the involved polynomial, the uniform Lipschitz parametrizability follows from Corollary \ref{Corollary_Gromov_Lemma} to Gromov's algebraic lemma.
\end{proof}

\begin{proof}[Proof of Proposition \ref{Proposition_Uniformity_R}]
Recall the discussion right after Proposition \ref{Proposition_Uniformity_R}, where we defined $T_{R, 1}$ and $T_{R, 2}$ and showed that 
$$\partial S_{\Omega, R}(1) \subset T_{R, 1} \cup T_{R, 2}.$$
To show uniform Lipschitz parametrizability of $\partial S_{\Omega, R}(1)$, it suffices to show the that for $T_{R, 1}$ and $T_{R, 2}$.

We just showed the uniform Lipschitz parametrizability of $T_{R, 1}$ in Proposition \ref{Proposition_TR1_Lipschitz_parametrizable}. One can show that $T_{R, 2}$ is uniformly Lipschitz parametrizable in a similar way; we sketch the argument as follows.

Recall that
$$T_{R, 2} := \left\{(g_{v, R}) \in \prod_{v \in M_{F, \infty}} H^0(E(F_v), O_E(R + O)): \left(\int_{X(\overline{F_v})} \log |g_{v, R}|_v d\mu_v\right) \in \partial \Omega + (-\infty, 0] \delta \right\}.$$
If there is only one infinite place (i.e. $r_1 + r_2 = 1$), then $\Omega$ is a point, and $T_{R, 2}$ is empty; the Lipschitz parametrizability of $T_{R, 2}$ is then vacuously true. Henceforth we will assume there are at least two infinite places.

Let
$$T'_{R, 2} = \left\{(g_{v, R}) \in \prod_{v \in M_{F, \infty}} H^0(E(F_v), O_E(R + O)): \left(\int_{X(\overline{F_v})} \log |g_{v, R}|_v d\mu_v\right) \in \partial\Omega\right\}.$$  
We first show that $T'_{R, 2}$ is uniformly Lipschitz parametrizable of codimension 2 over $R \in E(F)$. 
\begin{itemize}
    \item Since $\Omega$ is a parallelepiped in $\BR^{r_1 + r_2 - 1}$, we see that $\partial \Omega$ is the union of $2(r_1 + r_2 - 1) \ll_F 1$ parallelepiped. It suffices to show that for each parallelepiped $P$, the set
    $$T'_{R, 2, P} = \left\{(g_{v, R}) \in \prod_{v \in M_{F, \infty}} H^0(E(F_v), O_E(R + O)): \left(\int_{X(\overline{F_v})} \log |g_{v, R}|_v d\mu_v\right) \in P\right\}$$
    is uniformly Lipschitz parametrizable over $R \in E(F)$.

    As in Proposition \ref{Proposition_TRd1_Lipschitz_parametrizable_Implies_TR1} we strive to prove instead the larger set
    $$T'_{(R_v), 2, P} = \left\{(g_{v, R_v}) \in \prod_{v \in M_{F, \infty}} H^0(E(F_v), O_E(R_v + O)): \left(\int_{X(\overline{F_v})} \log |g_{v, R_v}|_v d\mu_v\right) \in P\right\}$$
    is uniformly Lipschitz parametrizable over $(R_v) \in \prod_{v \in M_{F, \infty}} E(F_v)$.
    \item For each parallelepiped $P$, similar to Proposition \ref{Proposition_TRd1_Lipschitz_parametrizable_Implies_TR1}, the condition
    $$\left(\int_{X(\overline{F_v})} \log |g_{v, R_v}|_v d\mu_v\right) \in P$$
    is covered by
    $$\delta \leq |a_v|^2 + |b_v|^2 \leq \delta^{-1} \text{ for each $v \in M_{F, \infty}$},$$
    and two conditions on Mahler measures:
    $$\prod_{v \in M_{F, \infty}} M_{v, \mathbf{O(1)}}\left(N_{v, f}(a_v g_{v, R_v, 1} + b_vg_{v, R_v, 2})\right) = 1$$
    and
    $$\prod_{v \in M_{F, \infty}} M_{v, \mathbf{O(1)}}\left(N_{v, f}(a_v g_{v, R_v, 1} + b_vg_{v, R_v, 2})\right)^{d_v} = 1$$
    for some $(d_v) \in \BR^{r_1 + r_2}$ linearly independent with $(1, \cdots, 1)$.
    \item Proposition \ref{Proposition_Mahler_measure_as_coordinate} still applies. Using those parametrizations there, we need to show uniform Lipschitz parametrizability of a codimension 2 set, similar to Proposition \ref{Proposition_TR1_Lipschitz_parametrizable}. 
    
    As before, varying $(R_v)$ only affects the coefficients of polynomials defining the semi-algebraic set, but not the degrees. Hence uniform Lipschitz parametrizability follows from Corollary \ref{Corollary_Gromov_Lemma} to Gromov's algebraic lemma.
\end{itemize} 
Finally, if $\psi: [0,1]^{r_1 + r_2 - 2} \to T'_{R, 2}$ is a Lipschitz map, then the map $\widetilde{\psi}: [0,1]^{r_1 + r_2 - 2} \times (0,1] \to T_{R, 2}$
defined by
$$\widetilde{\psi}(\mathbf{x}, u) = u \cdot \psi(\mathbf{x})$$
is clearly Lipschitz. By Lipschitz continuity, one can extend $\widetilde{\psi}$ to $[0,1]^{r_1 + r_2 - 1}$. The image of $\psi$'s covers $T_{R, 2}$, and the number of maps/Lipschitz constant are still uniform over $R \in E(F)$. This shows that $T_{R, 2}$ is uniformly Lipschitz parametrizable of codimension 1 over $R \in E(F)$, and wraps up the proof of our main result Proposition \ref{Proposition_Uniformity_R}.
\end{proof}

\subsection{Point counting the lattice $H^0(X, O_X(D_R + D_O))$}
Recall that after Corollary \ref{Corollary_Step_3_Reframe_Cohom_X}, we want to count points of the lattice
$$H^0(X, O_X(D_R + D_O)) \into H^0(X, O_X(D_R + D_O)) \otimes_{\BZ} \BR \cong \prod_{v \in M_{F, \infty}} H^0(E, O_E(R + O)) \otimes_F F_v,$$
lying inside the homogeneous region 
$$\exp\left(\frac{1}{k}\left(\log B - \log H(f(R)) - \log H(f(O))\right) - \frac{1}{d} \log |N_{F/\BQ}(a)|\right) S_{\Omega, R}(1),$$
as $B \to \infty$, for each $R \in E(F)$. The main goal of this section is Corollary \ref{Corollary_Counting_Rational_Functions_Asymptotics}, which provides an asymptotic formula for this point count.

We start with a version of Lipschitz principle, as stated in Masser-Vaaler \cite{MasserVaaler}.
\begin{lemma}[{\cite[Lemma 2]{MasserVaaler}}]
\label{Lemma_Geometry_Of_Numbers}
Let $S \subset \BR^D$ be a bounded set whose boundary $\partial S$ can be covered by the image of at most $W$ maps from $[0,1]^{D-1}$ to $\BR^D$ satisfying Lipschitz conditions
$$\|\phi(x_1) - \phi(x_2)\|_2 \leq L\|x_1 - x_2\|_2.$$
Then $S$ is Lebesgue measurable. Further, let $\Lambda$ in $\BR^D$ be a lattice with first successive minimum $\lambda_1(\Lambda)$. Then the number $Z$ of points in $S \cap \Lambda$ satisfies
$$\left|Z - \frac{\vol(S)}{\covol(\Lambda)} \right| \leq c W \left(\frac{L}{\lambda_1(\Lambda)} + 1\right)^{D-1}$$
for some $c = c(D)$ depending only on $D$.
\end{lemma}

\begin{prop}
\label{Proposition_Counting_Rational_Functions_Asymptotics}
For each $R \in E(F)$, if $T \ll_{E, f, F} (H(f(R)) H(f(O)))^{-1/2k}$ with a sufficiently small big-$O$ constant, then
$$\#\left\{g_X \in H^0(X, O_X(D_R + D_O)) - 0: (g_X)_{v \in M_{F, \infty}} \in T S_{\Omega, R}(1)\right\} = 0.$$
On the other hand, if $T \gg_{E, f, F} (H(f(R)) H(f(O)))^{-1/k}$, then
\begin{align*}
& \#\left\{g_X \in H^0(X, O_X(D_R + D_O)) - 0: (g_X)_{v \in M_{F, \infty}} \in T S_{\Omega, R}(1)\right\} \\
= & \frac{\vol_{Fal}(S_{\Omega, R}(1))}{\covol_{Fal}(H^0(X, O_X(D_R + D_O)))} T^{2d} + O_{E, f, F}\left(\left(H(f(R)) H(f(O))\right)^{(2d-1)/k} T^{2d-1}\right).
\end{align*}
as $T \to \infty$. 
\end{prop}
\begin{proof}
The first part follows from Proposition \ref{Proposition_Counting_Rational_Functions_Truncation}.

For the second part, we apply Lemma \ref{Lemma_Geometry_Of_Numbers} to $S = S(T) = T S_{\Omega, R}(1)$, $\Lambda = H^0(X, O_X(D_R + D_O))$, $D = 2d$. As $T$ varies, note that 
$$S(T) = T S(1), \, \vol(S(T)) = T^{2d} \vol(S(1)), \, W(T) = W(1), \, L(T) = T L(1).$$
Moreover, Proposition \ref{Proposition_First_Successive_Minimum_H0} gives us a lower bound
$$\lambda_1(\Lambda) \gg_{E,f, F} (H(f(R)) H(f(O)))^{-1/k},$$
hence the error term in Lemma \ref{Lemma_Geometry_Of_Numbers} contributes
$$O_{E, f, F}\left(W \left(L T (H(f(R)) H(f(O)))^{1/k}+ 1\right)^{2d-1}\right).$$
Here we used the uniformity of $W$ and $L$ from Proposition \ref{Proposition_Uniformity_R}, so that $W, L$ are independent of $R \in E(F)$.

The condition $T \gg_{E, f, F} (H(f(R)) H(f(O)))^{-1/k}$ guarantees that 
$$L T (H(f(R)) H(f(O)))^{1/k} \gg_{E, f, F} 1.$$ 
Hence the error term contributes
$$O_{E, f, F}\left(W \left(L T (H(f(R)) H(f(O)))^{1/k}\right)^{2d-1}\right) = O_{E, f, F}\left(\left(H(f(R)) H(f(O))\right)^{(2d-1)/k} T^{2d-1}\right).$$

Finally, note that any Haar measure on $\prod_{v \in M_{F, \infty}} H^0(E, O_E(R + O)) \otimes_F F_v$ gives the same main term; in particular, we can use Faltings volume as desired.
\end{proof}

We next evaluate $\vol_{Fal}(S_{\Omega, R}(1))$ that appears in the numerator of main term.
\begin{lemma}
\label{Lemma_Faltings_Volume_S_Omega}
\begin{align*}
&\vol_{Fal}(S_{\Omega, R}(1)) \\
= & 2^{r_1 + r_2 - 1} Reg_F \prod_{v \in M_{F, \infty}} \vol_{Fal}\left(
\left\{g_v \in H^0(X, O_X(D_R + D_O)) \otimes_{O_F} F_v: \int_{X(\overline{F_v})} \log |g_v|_v d\mu_v \leq 0\right\}
\right).
\end{align*}
\end{lemma}
\begin{proof}
The same argument in \cite[Lemma 4]{MasserVaaler} gives
\begin{align*}
& \vol_{Fal}(S_{\Omega, R}(1)) \\
= & 2^{r_1 + r_2 - 1} \frac{\vol_{Lebesgue}(\Omega)}{(r_1 + r_2)^{1/2}} \prod_{v \in M_{F, \infty}} \vol_{Fal}\left(\left\{g_v \in H^0(X, O_X(D_R + D_O)) \otimes_{O_F} F_v: \right.\right.\\
& \,\,\,\,\,\left.\left.\int_{X(\overline{F_v})} \log |g_v|_v d\mu_v \leq 0\right\}\right).
\end{align*}
Here recall that $\Omega \subset \{x_1 + \cdots + x_{r_1 + r_2} = 0\} \subset \BR^{r_1 + r_2}$ is a fundamental domain for the log embedding for $O_F^{\times}$.

Finally, note that the regulator of $F$ equals
$$Reg_F = \frac{\vol_{Lebesgue}(\Omega)}{(r_1 + r_2)^{1/2}}.$$
This is because the regulator is the absolute value of any maximal subdeterminant of the matrix with $r_1 + r_2 - 1$ rows and $r_1 + r_2$ columns for the log embedding of units, while the volume $\vol(\Omega)$ is the square root of the sum of squares of these same subdeterminants. Since all such subdeterminants are the same, the formula for regulator of $F$ follows, and this completes the proof.
\end{proof}

Substituting the calculated volumes (Lemma \ref{Lemma_Faltings_Volume_S_Omega}, Proposition \ref{Proposition_Covolume_H0X}) into Proposition \ref{Proposition_Counting_Rational_Functions_Asymptotics}, we get

\begin{cor}
\label{Corollary_Counting_Rational_Functions_Asymptotics}
For each $R \in E(F)$, if $T \ll_{E, f, F} (H(f(R)) H(f(O)))^{-1/2k}$ with a sufficiently small big-$O$ constant, then
$$\#\left\{g_X \in H^0(X, O_X(D_R + D_O)) - 0: (g_X)_{v \in M_{F, \infty}} \in T S_{\Omega, R}(1)\right\} = 0.$$
On the other hand, if $T \gg_{E, f, F} (H(f(R)) H(f(O)))^{-1/k}$, then
\begin{align*}
& \#\left\{g_X \in H^0(X, O_X(D_R + D_O)) - 0: (g_X)_{v \in M_{F, \infty}} \in T S_{\Omega, R}(1)\right\} \\
= & T^{2d} \cdot \frac{2^{r_1 + r_2 - 1}Reg_F}{|\disc(O_F)|} \cdot \frac{\exp\left(d \cdot h_{Faltings}(E)\right)}{|N_{F/\BQ}(\Delta_{E/F})|^{1/4}} \cdot H_{NT}(R)^d \cdot \\
&\,\,\, \prod_{v \in M_{F, \infty}}\vol_{Fal}\left(g_v \in H^0(X, O_X(D_R + D_O)) \otimes_F F_v: \int_{X(\overline{F_v})} \log |g_v|_v d\mu_v\leq 0\right) \\
&\,\,\, + O_{E, f, F}\left(T^{2d-1}(H(f(R))H(f(O)))^{(2d-1)/k}\right),
\end{align*}
as $T \to \infty$. 
\end{cor}

For later purpose, we prove an upper bound of $\vol_{Fal}(S_{\Omega, R}(1))$, uniform over $R \in E(F)$.

\begin{prop}
\label{Proposition_Faltings_Volume_S_Omega_Uniformly_Bounded}
We have a uniform bound over $R \in E(F)$,
$$\vol_{Fal}(S_{\Omega, R}(1)) \ll_{E, F} 1.$$
\end{prop}
\begin{proof}
By Lemma \ref{Lemma_Faltings_Volume_S_Omega},
\begin{align*}
&\vol_{Fal}(S_{\Omega, R}(1)) \\
\ll_F & \prod_{v \in M_{F, \infty}} \vol_{Fal}\left(g_v \in H^0(X, O_X(D_R + D_O)) \otimes_{O_F} F_v: \int_{X(\overline{F_v})} \log |g_v|_v d\mu_v \leq 0\right)
\end{align*}
Now for each archimedean place $v$, by Lemma \ref{Lemma_Log_Integral_Uniformly_Bounded}, the set
$$\left\{g_v \in H^0(X, O_X(D_R + D_O)) \otimes_{O_F} F_v: \int_{X(\overline{F_v})} \log |g_v|_v d\mu_v \leq 0\right\}$$
lies in a ball $B_{R, v, r} \subset H^0(X, O_X(D_R + D_O)) \otimes_{O_F} F_v$, with radius $r \ll_{E,v} 1$, bounded uniformly over $R \in E(F)$. Hence 
$$\vol_{Fal}\left(\left\{g_v \in H^0(X, O_X(D_R + D_O)) \otimes_{O_F} F_v: \int_{X(\overline{F_v})} \log |g_v|_v d\mu_v \leq 0\right\}\right) 
\ll_{E, v} \vol_{Fal} (B_{R, v, 1}).$$
But by Proposition \ref{Proposition_Faltings_Volume_Unit_ball_Uniformly_Bounded}, we saw that the volume of unit ball $\vol_{Fal}(B_{R, v, 1})$ is bounded above,  uniformly over $R \in E(F)$. Thus for each non-archimedean place $v$, we have
$$\vol_{Fal}\left(\left\{g_v \in H^0(X, O_X(D_R + D_O)) \otimes_{O_F} F_v: \int_{X(\overline{F_v})} \log |g_v|_v d\mu_v \leq 0\right\}\right) 
\ll_{E, v} 1.
$$
Multiplying the bound over all archimedean places $v$, we get the desired result.
\end{proof}

\subsection{Proof of Proposition \ref{Theorem_Main_Degree2_CountingResult_Sym2_EC}}
We now have all the ingredients to prove Proposition \ref{Theorem_Main_Degree2_CountingResult_Sym2_EC}.
\begin{proof}
By Proposition \ref{Proposition_Step_1_Rewrite_Counting_Rational_Functions}, Corollary \ref{Corollary_Step_2_Remove_Fcross_Action} and Corollary \ref{Corollary_Step_3_Reframe_Cohom_X}, we reframed the problem of counting points on $\Sym^2 E(F)$ to counting rational functions on $X$ with prescribed poles as follows.
\begin{align*}
& \#\{x \in (\Sym^2 E)(F): H_{\CL}(x) \leq B\} \\
= & \frac{1}{|\mu_F|} \sum_{R \in E(F)} \sum_{(a) \subset O_F} \mu((a)) \#\left\{g_X \in H^0(X, O_X(D_R + D_O)) - 0: \right.\\
& \left.\,\,\, (g_X)_{v \in M_{F, \infty}} \in \exp\left(\frac{1}{k}\left(\log B - \log H(f(R)) - \log H(f(O))\right) - \frac{1}{d} \log |N_{F/\BQ}(a)|\right) S_{\Omega, R}(1)\right\}
\end{align*}
By Corollary \ref{Corollary_Counting_Rational_Functions_Asymptotics}, this is really a finite sum, where we only count $a$'s satisfying
$$\exp\left(\frac{1}{k}\left(\log B - \log H(f(R)) - \log H(f(O))\right) - \frac{1}{d} \log |N_{F/\BQ}(a)|\right) \gg_{E, f, F} (H(f(R))H(f(O)))^{-1/(2k)}.$$
This is equivalent to 
$$|N_{F/\BQ}(a)|^{1/d} (H(f(R)) H(f(O)))^{1/(2k)} \ll_{E, f, F} B^{1/k}.$$
In particular, these implies that $|N_{F/\BQ}(a)| \ll_{E, f, F} B^{d/k}$, and $H(f(R)) \ll_{E, f, F} B^2$.

For the $a$'s that show up in the sum, Corollary \ref{Corollary_Counting_Rational_Functions_Asymptotics} applied to 
$$T = \exp\left(\frac{1}{k}\left(\log B - \log H(f(R)) - \log H(f(O))\right) - \frac{1}{d} \log |N_{F/\BQ}(a)|\right)$$
shows that
\begin{align*}
& \#\{x \in (\Sym^2 E)(F): H_{\CL}(x) \leq B\} \\
= & \frac{1}{|\mu_F|} \sum_{\stackrel{R \in E(F)}{H(f(R)) \ll B^2}} \sum_{\stackrel{(a) \subset O_F}{|N_{F/\BQ}(a)| \ll B^{d/k}}} \mu((a)) \\
&\,\,\, \left(\frac{2^{r_1 + r_2 - 1}Reg_F}{|\disc(O_F)|} \cdot \frac{\exp\left(d \cdot h_{Faltings}(E)\right)}{|N_{F/\BQ}(\Delta_{E/F})|^{1/4}} \cdot \right. H_{NT}(R)^d  \cdot \\
&\,\,\, \prod_{v \in M_{F, \infty}} \vol_{Fal}\left(g_v \in H^0(X, O_X(D_R + D_O)) \otimes_{O_F} F_v: \int_{X(\overline{F_v})} \log |g_v|_v d\mu_v \leq 0\right) \cdot \\
&\,\,\,\left(\frac{B^{1/k}}{(H(f(R))H(f(O)))^{1/k} |N_{F/\BQ}(a)|^{1/d}}\right)^{2d} \\
&\,\,\, \left.+ O_F\left((H(f(R))H(f(O)))^{(2d-1)/k} \left(\frac{B^{1/k}}{(H(f(R))H(f(O)))^{1/k} |N_{F/\BQ}(a)|^{1/d}}\right)^{2d-1}\right)\right)
\end{align*}
The error term simplifies to
\begin{align*}
& O_F \left(\frac{B^{(2d-1)/k}}{|\mu_F|} \left(\sum_{\stackrel{R \in E(F)}{H(f(R)) \ll B^2}} 1 \right) \left(\sum_{\stackrel{(a) \subset O_F}{|N_{F/\BQ}(a)| \ll B^{d/k}}} \frac{1}{N_{F/\BQ}(a)^{2 - 1/d}}\right)\right)
\end{align*}
By Theorem \ref{Theorem_GeneralCountingResult_AV},
$$\sum_{\stackrel{R \in E(F)}{H(f(R)) \ll B^2}} 1 = O_{E, f, F}((\log B)^{r/2})$$
where $r$ is the rank of Mordell-Weil group $E(F)$. Also, the number of ideals of norm up to $x$ is $\asymp_F x$, since the rightmost pole of $\zeta_F(s)$ is at 1 and is simple. Hence
$$\sum_{\stackrel{(a) \subset O_F}{|N_{F/\BQ}(a)| \ll B^{d/k}}} \frac{1}{N_{F/\BQ}(a)^{2 - 1/d}} = O\left(\sum_{n \ll B^{d/k}} \frac{1}{n^{2 - 1/d}}\right) = O_{E, f, F} \left(\log B \right).$$
Thus the error term contributes
$$
O_{E, f, F} \left(\frac{B^{(2d-1)/k}}{|\mu_F|} \cdot (\log B)^{r/2} \cdot \log B \right)
= O_{f, E, F} \left(B^{(2d-1)/k} (\log B)^{r/2 + 1} \right).
$$
For the main term, it simplifies to
\begin{align*}
& \text{Main term of } \#\{x \in (\Sym^2 E)(F): H_{\CL}(x) \leq B\} \\
= & \frac{2^{r_1 + r_2 - 1} Reg_F}{|\mu_F| |\disc(O_F)|} \cdot \frac{\exp\left(d \cdot h_{Faltings}(E)\right)}{|N_{F/\BQ}(\Delta_{E/F})|^{1/4} \cdot H(f(O))^{2d/k}} B^{2d/k} \\
& \sum_{\stackrel{R \in E(F)}{H(f(R)) \ll B^2}} \left(\frac{H_{NT}(R)}{H(f(R))^{2/k}}\right)^d  \\
& \,\,\, \left(\prod_{v \in M_{F, \infty}} \vol_{Fal}\left(g_v \in H^0(X, O_X(D_R + D_O)) \otimes_{O_F} F_v: \int_{X(\overline{F_v})} \log |g_v|_v d\mu_v \leq 0\right)\right) \\
& \,\,\, \sum_{\stackrel{(a) \subset O_F}{|N_{F/\BQ}(a)| \ll B^{d/k}}} \frac{\mu((a))}{|N_{F/\BQ}(a)|^2}.
\end{align*}
Again as the number of ideals of norm up to $x$ is $\asymp_F x$, we have
$$\sum_{\stackrel{(a) \subset O_F}{|N_{F/\BQ}(a)| \ll B^{d/k}}} \frac{\mu((a))}{|N_{F/\BQ}(a)|^2} = \frac{1}{\zeta_F(2)} + O\left(\sum_{\stackrel{(a) \subset O_F}{|N_{F/\BQ}(a)| \gg B^{d/k}}} \frac{1}{|N_{F/\BQ}(a)|^2}\right) = \frac{1}{\zeta_F(2)} + O_F\left(B^{-d/k}\right).$$
Note also that the series
\begin{align*}
& \sum_{R \in E(F)} \left(\frac{H_{NT}(R)}{H(f(R))^{2/k}}\right)^d \\ 
& \,\,\, \left(\prod_{v \in M_{F, \infty}} \vol_{Fal}\left(g_v \in H^0(X, O_X(D_R + D_O)) \otimes_{O_F} F_v: \int_{X(\overline{F_v})} \log |g_v|_v d\mu_v \leq 0\right)\right)
\end{align*}
converges, because 
$$\frac{H_{NT}(R)}{H(f(R))^{1/k}} \asymp_E 1 $$
uniformly over $R \in E(F)$ (by property of N\'{e}ron-Tate height), hence
$$\left(\frac{H_{NT}(R)}{H(f(R))^{2/k}}\right)^d \asymp_E \frac{1}{H(f(R))^{d/k}}$$
decays exponentially in $R$; and 
$$\prod_{v \in M_{F, \infty}} \vol_{Fal}\left(g_v \in H^0(X, O_X(D_R + D_O)) \otimes_{O_F} F_v: \int_{X(\overline{F_v})} \log |g_v|_v d\mu_v \leq 0\right) \ll_{E, F} 1,$$
by Proposition \ref{Proposition_Faltings_Volume_S_Omega_Uniformly_Bounded}.

Hence the main term simplifies to
\begin{align*}
& \text{Main term of } \#\{x \in (\Sym^2 E)(F): H_{\CL}(x) \leq B\} \\
= & \frac{2^{r_1 + r_2 - 1} Reg_F}{|\mu_F| |\disc(O_F)|} \cdot \frac{1}{\zeta_F(2)} \cdot \frac{\exp\left(d \cdot h_{Faltings}(E)\right)}{|N_{F/\BQ}(\Delta_{E/F})|^{1/4} \cdot H(f(O))^{2d/k}} B^{2d/k} \\
& \sum_{R \in E(F)} \left(\frac{H_{NT}(R)}{H(f(R))^{2/k}}\right)^d  \\
& \,\,\, \left(\prod_{v \in M_{F, \infty}} \vol_{Fal}\left(g_v \in H^0(X, O_X(D_R + D_O)) \otimes_{O_F} F_v: \int_{X(\overline{F_v})} \log |g_v|_v d\mu_v \leq 0\right)\right) \\
& \,\,\, + O_{E, f, F} (B^{d/k}).
\end{align*}
Putting both main term and error term together, we get
\begin{align*}
& \#\{x \in (\Sym^2 E)(F): H_{\CL}(x) \leq B\} \\
= & \frac{2^{r_1 + r_2 - 1} Reg_F}{|\mu_F| |\disc(O_F)|} \cdot \frac{1}{\zeta_F(2)} \cdot \frac{\exp\left(d \cdot h_{Faltings}(E)\right)}{|N_{F/\BQ}(\Delta_{E/F})|^{1/4} \cdot H(f(O))^{2d/k}} B^{2d/k} \\
& \sum_{R \in E(F)} \left(\frac{H_{NT}(R)}{H(f(R))^{2/k}}\right)^d  \\
& \,\,\, \left(\prod_{v \in M_{F, \infty}} \vol_{Fal}\left(g_v \in H^0(X, O_X(D_R + D_O)) \otimes_{O_F} F_v: \int_{X(\overline{F_v})} \log |g_v|_v d\mu_v \leq 0\right)\right) \\
& \,\,\, + O_{E, f, F} (B^{(2d-1)/k} (\log B)^{r/2 + 1}).
\end{align*}
as desired (By Assumption \ref{Assumption_Sym2_EC}, $F$ has class number 1; which is why $|Cl(F)|$ shows up in the proposition, but not here).
\end{proof}

\begin{remark}
We briefly mention changes needed without Assumption \ref{Assumption_Sym2_EC}.
\begin{itemize}
    \item On recasting Proposition \ref{Theorem_Main_Degree2_CountingResult_Sym2_EC} into a lattice point counting problem.
    \begin{itemize}
        \item Proposition \ref{Proposition_Step_1_Rewrite_Counting_Rational_Functions} generalizes: for each finite place $v$, consider the irreducible components $C_{v, i}$ of $X_v$, which corresponds to discrete absolute value $|\cdot|_{C_{v,i}}$ on $K(X)$. Instead of $\log |g_X|_{X_v}$, we consider $\log |g_X|_{C_{v, i}}$, weighted by the number of times a section of $f^*O(1)$ goes through $C_{v, i}$ divided by $k$.
        \item Corollary \ref{Corollary_Step_2_Remove_Fcross_Action} generalizes: we can require $\sum_{v \in M_{F, finite}} \log |g_X|_{X_v}$ to be equal to a fixed representative of $\dfrac{\text{Vertical divisors of $X$}}{\text{Principal divisors of $F$}}$. To study these representatives, we use the exact sequence
        $$0 \to Cl(F) \to \frac{\text{Vertical divisors of $X$}}{\text{Principal divisors of $F$}} \to \frac{\text{Vertical divisors of $X$}}{\text{Vertical fibers of $X$}} \to 0.$$
        This is how class number of $F$ shows up in the main term at the end.
        \item Corollary \ref{Corollary_Step_3_Reframe_Cohom_X} generalizes: instead of $D_R + D_O$, we would need to handle $D_R + D_O + V$ for some vertical divisor $V$.
    \end{itemize}
    \item On the first successive minimum and covolume of lattice
    \begin{itemize}
    \item Proposition \ref{Proposition_First_Successive_Minimum_H0} and \ref{Proposition_Counting_Rational_Functions_Truncation} generalizes directly.
    \item Proposition \ref{Proposition_Covolume_H0X} generalizes. For $\covol_{Fal}(H^0(X, O_X(D_R + D_O))$,
    \begin{itemize}
        \item $D_R - D_O$ is no longer perpendicular to all vertical divisors; hence one needs to add an auxiliary vertical divisor to fix this issue, before Faltings-Hriljac theorem can be applied (see \cite[Section 7]{Zhang}). This produces an extra multiple of the shape $\exp(\langle V_0, V_0 \rangle)$ for some vertical divisor $V_0 = V_0(R)$ in the final expression.
        \item Szpiro's theorem (Proposition \ref{Proposition_RobindeJong}) does not apply when $E$ is not semi-stable, and we need to keep the $\langle O_X(D_O), O_X(D_O) \rangle$ without simplifying it further.
    \end{itemize}
    
    We would also need to handle $\covol_{Fal}(H^0(X, O_X(D_R + D_O + V)))$ for vertical divisors $V$. 
    \begin{itemize}
        \item The intersection product of $V$ with $D_R$, $\omega_X$ and $V$ itself can be handled, so the Faltings-Riemann-Roch calculation generalizes. 

        A particularly nice case is when $V$ is a sum of vertical fibers: if $I = \sum_v c_v v$ is an integral ideal of $F$, and $V = \sum_v c_v V_v$ is the sum of corresponding vertical fibers, then
        $$\covol_{Fal}(H^0(X, O_X(D_R + D_O + V))) = \covol_{Fal}(H^0(X, O_X(D_R + D_O))) + 2 \log |N(I)|.$$
    \end{itemize}
    Corollary \ref{Corollary_Counting_Rational_Functions_Asymptotics} would generalize after the modifications above.
    \end{itemize}

    \item The final main term would be different from our simplified case in the following ways: in the sum over $R \in E(F)$,
    \begin{itemize}
        \item There is an extra multiple of shape $\exp(\langle V_0, V_0 \rangle)$, for some vertical divisor $V_0$ depending on $R$, needed for the application of Faltings-Hriljac theorem.
        \item The $|N_{F/\BQ}(\Delta_{E/F})|^{-1/4}$ term becomes $\exp\left(3 \langle O_X(D_O), O_X(D_O) \rangle\right)$, since Szpiro's theorem (Proposition \ref{Proposition_RobindeJong}) does not apply when $E$ is not semi-stable.
        \item There is an extra infinite sum over $V \in \dfrac{\text{Vertical divisors of $X$}}{\text{Vertical fibers of $X$}}$, where we effectively sum over $\dfrac{\covol_{Fal}(H^0(X, O_X(D_R + D_O))}{\covol_{Fal}(H^0(X, O_X(D_R + D_O + V))}$. The sum would be independent over $R$, and can be broken into a product over finite places of bad reduction.
    \end{itemize}
\end{itemize}
\end{remark}

\bibliographystyle{plain}
\bibliography{refs}

\end{document}